\documentclass{article}

\usepackage[utf8]{inputenc}
\usepackage{lipsum}
\usepackage{url}
\usepackage{mathpazo}
\usepackage[T1]{fontenc}
\usepackage[utf8]{inputenc}
\usepackage[english]{babel}
\usepackage[active]{srcltx}
\usepackage{colorspace}
\usepackage{tikz}
\usepackage{sansmath}
\usetikzlibrary{shadings,intersections,patterns}
\usepackage{pgfplots}

\pgfplotsset{compat=1.10}
\usepgfplotslibrary{fillbetween}
\usetikzlibrary{patterns}
\usetikzlibrary{intersections, calc, angles}
\usepackage{tkz-euclide}
\pgfdeclarelayer{pre main}
\usepackage{environ}
\makeatletter
\newsavebox{\measure@tikzpicture}
\NewEnviron{scaletikzpicturetowidth}[1]{%
  \def\tikz@width{#1}%
  \begin{lrbox}{\measure@tikzpicture}%
  \BODY
  \end{lrbox}%
  \pgfmathparse{#1/\wd\measure@tikzpicture}%
  \BODY
}

\makeatother
\usetikzlibrary{arrows,automata}
\usetikzlibrary{positioning}
\usetikzlibrary{calc,through}
\usepackage{etoolbox}%
\apptocmd{\thebibliography}{\fontsize{11}{15}\selectfont}{}{}%

\tikzset{
    state/.style={
           rectangle,
           rounded corners,
           draw=black, very thick,
           minimum height=2em,
           inner sep=2pt,
           text centered,
           },
}

\usepackage{mathrsfs}
\usepackage{xfrac}

\usepackage[left=1.5cm, right=1.5cm, top=2cm, bottom=2cm]{geometry}

\usepackage{amsmath}
\usepackage{amssymb}
\usepackage{amsthm}
\usepackage{hyperref}
\usepackage{graphicx}
\usepackage{subcaption}
\usepackage{caption}
\usepackage{xcolor}
\usepackage{bigints}
\usepackage{float}
\usepackage[capitalise, nameinlink, noabbrev]{cleveref}%Leave it as the last package

\usepackage{longtable}
\usepackage{verbatimbox}

\usepackage[title]{appendix}
\usepackage{bm}
\usepackage{empheq}
\usepackage{placeins}
\theoremstyle{plain}
\newtheorem{theorem}{Theorem}[section]
\newtheorem{lemma}[theorem]{Lemma}
\newtheorem{proposition}[theorem]{Proposition}

\newtheorem{corollary}[theorem]{Corollary}
\newtheorem{remark}[theorem]{Remark}
\newtheorem{definition}[theorem]{Definition}
\theoremstyle{definition}
\theoremstyle{remark}
\numberwithin{equation}{section}

\newcommand{\be}{\begin{equation}}
	\newcommand{\ee}{\end{equation}}
 \newcommand{\bea}{\begin{equation*}\begin{aligned}}
		\newcommand{\eea}{\end{aligned}\end{equation*}}

\DeclareMathOperator{\diver}{div}

\newcommand{\e}{\varepsilon}

\newcommand{\pa}{\partial}

\theoremstyle{plain}% default
\newtheorem*{theorem*}{Theorem}
\newtheorem*{corollary*}{Corollary}

\theoremstyle{definition}

\newtheorem*{notation*}{Notation}

\numberwithin{equation}{section}
\numberwithin{figure}{section}

\definecolor{myred}{rgb}{0.9,0,0}

\definecolor{vargreen}{rgb}{0.0, 0.5, 0.0}

\newcommand{\tens}[1]{\mathsf{#1}}
\makeatletter
\newcommand{\mylabel}[2]{#2\def\@currentlabel{#2}\label{#1}}
\makeatother

\newcommand{\R}{\mathbb{R}}
\newcommand{\N}{\mathbb{N}}
\newcommand{\abs}[1]{\left\lvert#1\right\rvert}
\newcommand{\norm}[1]{\left\lVert#1\right\rVert}

\renewcommand{\rho}{\varrho}
\renewcommand{\theta}{\vartheta}

\newcommand{\eps}{\varepsilon}
\definecolor{racing}{rgb}{0.7,0.1,0.2}
\definecolor{french}{rgb}{0,0.2,0.7}

\usepackage[small, sc]{titlesec}

\usepackage[
backend=bibtex,
style=numeric,
sorting=nty,
maxbibnames=50,
giveninits=true
]{biblatex}
\usepackage{csquotes}
\renewbibmacro{in:}{%
 \ifentrytype{article}{}{}}
\DeclareFieldFormat[article]{volume}{\textbf{#1}}
\DeclareFieldFormat[article]{title}{\textit{#1}}
\DeclareFieldFormat[article]{journaltitle}{#1}
\DeclareFieldFormat[article]{number}{\textbf{#1}}
\renewbibmacro*{volume+number+eid}{%
  \printfield{volume}%
\setunit*{\textbf{\addcolon}}%<---- was \setunit*{\adddot}%
  \printfield{number}\setunit{\addcomma\space}%
  \printfield{eid}}

\usepackage{xcolor}
\usepackage{amsfonts}

\usepackage{hyperref} %Leave it as the but one  last package

\hypersetup{colorlinks, 
linkcolor=blue,
citecolor={red!80!black},
urlcolor={cyan}
}

\begin{document}
\title{\textsc{Free boundary regularity for a\\ tumor growth model with obstacle}}

\author{
\textsc{Giulia Bevilacqua}\\[7pt]
\small Dipartimento di Matematica\\ 
\small Università di Pisa\\ 
\small Largo Bruno Pontecorvo 5, I–56127 Pisa, Italy\\
\small \href{mailto:giulia.bevilacqua@dm.unipi.it}{giulia.bevilacqua@dm.unipi.it}
\and
\textsc{Matteo Carducci}\\[7pt]
\small  Classe di Scienze, \\
\small Scuola Normale Superiore
\\ 
\small Piazza dei Cavalieri 7, I-56126 Pisa, Italy
\\
\small \href{mailto:matteo.carducci@sns.it}{matteo.carducci@sns.it}
}

\date{}

\maketitle

\begin{abstract}
   \noindent
   We develop an existence and regularity theory for solutions to a geometric free boundary problem motivated by models of tumor growth.
   In this setting, the tumor invades an accessible region $D$, its motion is directed along a constant vector $V$, and it cannot penetrate another region $K$ acting as an obstacle to the spread of the tumor. 
   Due to the non variational structure of the problem, we show existence of viscosity solutions via Perron's method. 
   Subsequently, we prove interior regularity for the free boundary near regular points by means of an improvement of flatness argument. We further analyze the boundary regularity and we prove that the free boundary meets the obstacle as a $C^{1,\alpha}$ graph. 
   A key step in the analysis of the boundary regularity involves the study of a thin obstacle problem with oblique boundary conditions, for which we establish $C^{1,\alpha}$ estimates.
\end{abstract}

\bigskip
\bigskip



\bigskip
\bigskip

\maketitle

\tableofcontents

\section{Introduction}
In recent decades, the modeling of tissue growth has significantly improved, thanks to the development of new analytical tools and the introduction of efficient numerical methods. One of the main challenges in studying cancer development lies in the presence of various interacting cell types, each requiring specific models to accurately capture their behavior.
In the literature, two main approaches are commonly used to describe tumor growth at the macroscopic level. The first is a partial differential equation (PDE) 
approach, in which cancer cells evolve according to a prescribed evolution equation \cite{greenspan1976growth, byrne1996modelling, ribba2006multiscale, ciarletta2011radial}. The second is a free boundary approach, where tissue dynamics are governed by the motion of its boundary \cite{greenspan1972models, cui2008asymptotic}.
The PDE-based model is generally more tractable from a mathematical viewpoint, indeed it is often known as the mechanical model because it can be derived from fundamental physical principles. However, the free boundary approach tends to provide a more faithful representation of the biological processes occurring in pathological tissues. Moreover, there is a link between these two approaches called {\em incompressible limit}: the pressure in a population-based model becomes stiffer and stiffer to get a free boundary description \cite{perthame2014hele, mellet2017hele, kim2018porous, bubba2020hele}.

\subsection{Porous media equation for tumor growth}

The Porous Medium Equation (PME) has been widely employed in the study of tumor growth.
Let  $x \in \mathbb{R}^d$ and $t \in (0, +\infty)$, the density of cells $n = n(x,t)$ evolves according to the continuity equation
\begin{equation*}
    \label{e:PME}
    \frac{\partial n}{\partial t} + \diver\left(n U\right) = 0,
\end{equation*}
where the velocity field $U= U(x,t)$ 
depends 
on the pressure, reflecting the effects of cell proliferation and apoptosis.
Specifically, if a linear dependence is taken into account, according to Darcy's law \cite{Dar1856}, we have
$$
U = - \nabla p,
$$
where $p = p(x,t)$ is the pressure field. 
To close the system, we need to prescribe a constitutive relation linking the cell density $n$ to the pressure $p$. The minimal physical assumptions on an equation of state are
$$
p = p(n), \qquad p(0) = 0 \qquad \text{and} \qquad \partial_n p > 0.
$$
In the literature, the prototypical constitutive relation is the power-law pressure (see, for instance, \cite{PQTV14, bubba2020hele, david2021free})
\begin{align*}
    p(n) = n^\gamma,
\end{align*}
where $\gamma >1$.
Furthermore, due to this constitutive relation, the pressure $p = p(n)$ evolves according to
\begin{equation*}
    \label{e:equazione_pressione}
    \frac{\partial p}{\partial t} = \abs{\nabla p}^2 + q \Delta p,
\end{equation*}
where
$$
q(n)  = n\partial_n p(n).
$$
Precisely, for the power-law case it results that $q(n) = \gamma p$. 

We remark that the study of Porous Medium Equation in tissue growth has been widely generalized. 
Extensions include source terms
to model cell proliferation \cite{byrne1996modelling}; local and non-local drifts to account for chemical concentration of cells \cite{kim2010degenerate,alexander2014quasi}; coupling with additional parabolic equations to describe nutrient dynamics \cite{perthame2014hele, david2021free}; and multi-species models involving cross-reaction terms between different cell populations \cite{carrillo2018splitting, gwiazda2019two, bubba2020hele}.

\subsection{From population-based models to free boundary ones: the incompressible limit}

Even though the density-based approach to model tissue growth is well developed and supported by a variety of analytical and numerical tools, the geometric free boundary formulation offers a closer description of the biological tumor evolution, as first suggested in \cite{greenspan1972models}. 

A rigorous connection between the two approaches can be established by considering the so-called {\em incompressible limit}: as the pressure law becomes stiffer and stiffer, for instance in the power-law case $p(n) = n^\gamma$ as $\gamma \to +\infty$, the dynamics governed by the Porous Medium Equation converge to those of a free boundary Hele-Shaw-type problem \cite{perthame2014hele, perthame2015incompressible, kim2018porous, bubba2020hele}.
Precisely, letting $p = n^\gamma$, for $\gamma > 1$, and considering the pressure equation
\[
\partial_t p = |\nabla p|^2 + \gamma p\Delta p, 
\]
we divide both sides by $\gamma$ and formally let $\gamma \to +\infty$. Under suitable regularity and compactness assumptions, the sequence of solutions $(n_\gamma, p_\gamma)$ of the PME and of the pressure equation converge to a pair $(n_\infty, p_\infty)$ solving, for all $x \in \mathbb{R}^d$ and $t > 0$,
$$
\left\{
\begin{aligned}
\partial_t n_\infty - \diver (n_\infty \nabla p_\infty) &= 0, \\
p_\infty \Delta p_\infty &= 0,
\end{aligned}
\right.
$$
where the second equation holds in a weak sense and is often referred to as the {\em complementarity relation}.
In addition, the limiting density $n_\infty$ and pressure $p_\infty$ satisfy the condition
$$
0 \leq n_\infty \leq 1 \qquad \hbox{and} \qquad  p_\infty(n_\infty - 1) = 0,
$$
which implies that $p_\infty > 0$ only in regions where the cell density saturates the maximal packing constraint, that is, $n_\infty = 1$. This establishes a precise link between the tumor region $\{n_\infty = 1\}$ and the positivity set of the pressure $\Omega(t) := \{x \in \R^d : p_\infty(x,t) > 0\}$. It can be shown that these two sets coincide almost everywhere \cite{mellet2017hele}.

To fully characterize the evolution of the tumor interface, one must track the motion of the free boundary $\partial \Omega(t)$ over time. In \cite{perthame2014hele}, the following kinematic condition was formally derived
\begin{equation*}
   \label{e:condizione_iniziale}
    U_\infty \cdot \nu_{\partial \Omega(t)} = \abs{\nabla p_\infty}  \quad \text{on } \partial \Omega(t),
\end{equation*}
where $U_\infty$ denotes the normal velocity of the moving boundary and $\nu_{\partial\Omega(t)}$ is the unit normal to $\partial\Omega(t)$.
Physically, this condition expresses that the expansion of the saturated region (the tumor core) is driven by the pressure gradient at the boundary.
In the case when the PME is modified by a growth term $G(p)$, this relation has been rigorously justified by compactness arguments along approximating sequences \cite{kim2018porous}.

A key step in rigorously establishing the incompressible limit is to obtain strong compactness of the pressure $p_\gamma$, which is crucial due to the nonlinear term $|\nabla p_\gamma|^2$ in the equation of the pressure.
For the power-law equation of state, this is achieved in \cite{perthame2014hele} via a comparison principle, which leads to an Aronson-B\'enilan type estimate for the Laplacian of the pressure field \cite{AB79}, namely
$$
\Delta p_\gamma  \geq - \frac{C}{\gamma t},
$$
for some constant $C > 0$.
Uniform bounds on $\Delta p_\gamma$ are sufficient to infer strong compactness of $\nabla p_\gamma$, which allows to pass to the limit in the nonlinear terms. 
For this reason, Aronson-B\'enilan type estimates have received much attention in recent years, and a wide variety of generalizations have been developed. 
Such estimates are known to hold for other nonlinear diffusion equations, such as the filtration equation $\partial_t n = \Delta \phi(n)$ \cite{crandall1982}, and for multi-species models \cite{bubba2020hele}. They have also been extended to PME systems coupled with nutrient diffusion or external drift terms \cite{david2021free, david2024incompressible}.
More recently, in \cite{bevilacqua2022aronson}, a weaker form of the Aronson-B\'enilan estimate valid in $L^p$ spaces for any $p \in [1, \infty]$ has been derived, allowing the application of the comparison principle even in the presence of negative source terms.

We also point out that the incompressible limit procedure allows the transfer information from the Porous Medium Equation to the geometric Hele-Shaw flow. Roughly speaking, one can construct a solution to the PME and establish uniform estimates with respect to the parameter $\gamma$ (in the case where the equation of state is the power-law one). Then, by taking the limit as $\gamma \to +\infty$, the limiting functions are shown, under suitable conditions, to solve the corresponding free boundary problem, see \cite{gil2001convergence, gil2003boundary, kim2003uniqueness}.
A significant contribution in such a direction is \cite{kim2016free}, where the Authors apply the viscosity solutions framework to the PME to derive existence and regularity results for the geometric limiting problem. Precisely, they obtain pointwise information on the evolution of the free boundary and on the propagation speed of the tumor region, see \cite[Section 4]{kim2016free} and \cite[Section 5]{kim2016free} respectively.
We also highlight \cite{kim2018porous}, where two key results are obtained. On one hand, the velocity law on the free boundary $\partial \Omega(t)$,
\[
U_\infty \cdot \nu_{\partial \Omega(t)} = \frac{\abs{\nabla p_\infty}}{1- \min\left(1, n_\infty\right)},
\]
originally conjectured in \cite{perthame2014hele}, is rigorously proved. On the other hand, using again a viscosity solution approach and considering a more general class of initial data, the authors show that both the density $n_\gamma$ and the pressure $p_\gamma$ converge locally uniformly to $(n_\infty, p_\infty)$, so that $p_\infty$ is a solution to the Hele-Shaw problem in the viscosity sense.
Finally, we remark that in several extensions of the Porous Medium Equation, such as the inclusion of source and drift terms \cite{kim2019singular}, modifications to the equation of state \cite{kim2018uniform}, and multi-species interactions \cite{gwiazda2019two}, the incompressible limit procedure still provides a robust framework to characterize the existence and regularity of solutions to the Hele-Shaw flow along converging sequences.

\subsection{Main results}

In this paper, we study the existence and regularity of solutions to a geometric free boundary problem arising from a model of tissue growth. Despite the huge theory on the Porous Media Equation, we do not consider any incompressible limit procedure, but we conjecture that such a limit could be rigorously derived (see the discussion in \cref{subsec:phisicalmotivation} below). Instead, we start directly from an elliptic free boundary problem, 
we investigate the existence of viscosity solutions 
and we analyze the regularity of the associated free boundary. The problem is set in a bounded domain where part of such a box is inaccessible and acts as an obstacle.
Precisely, in the rest of the paper, we always assume that:
\begin{itemize}
    \item[($\mathcal{B}$)] the box $\mathcal{B}\subset\R^d$ is a smooth bounded open set (box). Physically, it is the container where the tumor evolves;
    \item[(K)] the obstacle $K\subset\overline{\mathcal{B}}$ is a compact set with smooth boundary. Physically, it is a non-accessible region by the tumor;
    \item[(D)] we define $D := \mathcal{B}\setminus K$. Physically, it is the accessible region by the tumor;
    \item[(V)] we consider $V\in \R^d$ a constant vector field. Physically, it is main direction along with the tumor changes. 
\end{itemize}

In the above setting, the precise mathematical formulation of the problem we want to study is the following
\begin{equation}\label{e:viscosity-sol-K_D}
\left\{
    \begin{aligned}
         &\Delta u = 0 && \hbox{in }  \Omega_u\cap D := \{u>0\} \cap D,\\
     &\abs{\nabla u}^2 = \nabla u\cdot V&& \hbox{on } \pa \Omega_u\cap D,\\
     &\abs{\nabla u}^2\ge\nabla u\cdot V && \hbox{on } \pa \Omega_u\cap K\cap \mathcal{B},\\
     &u=0&& \hbox{in } K,
    \end{aligned}
     \right.
\end{equation}
and we refer to \cref{fig:geo_setting} for a graphical representation.
\begin{figure}[H]
    \centering
    \begin{tikzpicture}[rotate=90, scale= 1]
\coordinate (O) at (0,0);
\draw[fill=blue, opacity=0.05] (0,0) circle [radius = 20mm];
\draw[thick] (O) circle [very thick, radius=2cm, name path=c];
\draw [thick, color=red, name path=2] plot [smooth] coordinates {(-1.12,1.65) (-0.8,0.6) (-0.8,-0.4) (-1.12,-1.65)};

%\draw [thick, color=blue!5, name path=1] plot [smooth] coordinates {(1.12,-1.65) (0.75,-0.6) (0.77,0.4) (1.12,1.65) };
\draw [thick, color=blue!5, name path=1] plot [smooth] coordinates {(1.12,-1.65) (0.75,-0.6) (0.77,0.4) (1.12,1.65) };

\begin{scope}[transparency group, opacity=0.2]
\draw[draw=none, fill=french] (0,0) -- +(125:2cm) arc (125:235:2cm);
\end{scope}
%\tikzfillbetween[of=1 and 2] {color=blue!5};
\tikzfillbetween[of=1 and 2] {color=blue!5};

\draw[thick] (O) circle [very thick, radius=2cm];

%\draw [very thick, color=red, name = due] plot [smooth, tension=0.9] coordinates {(-0.82,-0.5)(-0.5, -0.25)(-0.25,-0.1) (0.1,0) (0.5,0.5) (1.1,0.9) (1.5,1.65) (1.9,2)};
\draw [very thick, color=red, name = due] plot [smooth, tension=0.9] coordinates {(-0.964,-1.1) (-0.75,-0.65)(-0.5, -0.35)
(-0.3,-0.2) (0.1,0) (0.5,0.5) (1.1,0.9) (1.29,1.53)};
\draw [very thick, color=french, name = uno] plot [smooth] coordinates {(-1.12,1.65) (-0.8,0.6) (-0.8,-0.4) (-1.12,-1.65)};

\draw[very thick, color= french] (O) circle [very thick, radius=2cm];

\draw[thick, color=cyan, ->] (-0.8,-2.5) -- (0.8,-2.5);
%\draw node [color = red]  at (1.15,0.1) {$\partial \Omega_u\cap D$};
%\draw node [color = red]  at (05,0.1) {$\Omega_u$};
\draw node [color = black]  at (-0.2,1) {$\{u=0\}$};
\draw node [color = red]  at (1.15,0.4) {$\partial \Omega_u$};
%\draw node [color = black]  at (0.5,-0.5) {$\Omega_u$};
\draw node [color = french] at (-1.38,-0.5) {$K$};
\draw node [color = black] at (0.3,-1.1) {$\Omega_u\cap D$};
\draw node[color = french] at (1,2.25) {$\mathcal{B}$};
\draw node[color = cyan] at (-1.25,-2.5) {$V$};
\end{tikzpicture}

    \caption{Graphical representation of the geometrical setting of the problem.}
    \label{fig:geo_setting}

\end{figure}
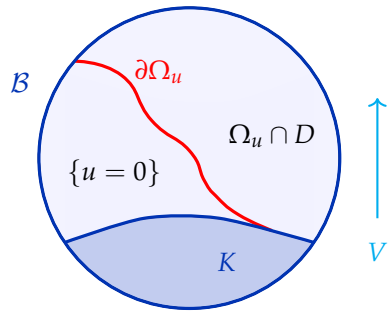

\subsubsection{Physical motivations}\label{subsec:phisicalmotivation}
From a physical point of view, the tumor invades an accessible region $D$, its motion is directed along a preferred constant vector $V$, but it cannot penetrate the entire domain $\mathcal{B}$. Moreover, we assume that tumor invasion occurs only in a localized region of the body, without metastasis, and that the compact set $K$ acts as an obstacle to the spread of the tumor. In addition, \eqref{e:viscosity-sol-K_D} can be obtained by taking formally a traveling wave solution to the Hele-Shaw problem. 

Although a rigorous derivation of the incompressible limit is beyond the scope of this paper, we expect that \eqref{e:viscosity-sol-K_D} can be obtained as 
an incompressible limit of the following Porous Medium-type system
\[
\left\{
\begin{aligned}
    &\frac{\partial n}{\partial t} - \operatorname{div}\left(n \nabla p\right) = 0 && \text{in } D, \\
    &n = 0 && \text{in } K,
\end{aligned}
\right.
\]
and the system must be complemented with appropriate initial conditions (see \cite{kim2010degenerate, david2024incompressible}). 
Assuming the power-law equation of state $p = n^\gamma$, the pressure satisfies
\[
\left\{
\begin{aligned}
    &\frac{\partial p}{\partial t} = \gamma p \Delta p + |\nabla p|^2 &&\text{in } D\\
    &p = 0&&\text{in } K
\end{aligned}
\right.
\]
Formally, dividing by $\gamma > 1$ and taking the limit as $\gamma \to +\infty$
and looking for solutions of the form 
$$
u(x) := p_{\infty}(x+ t V),
$$
one expects that the limit function $u$ solves \eqref{e:viscosity-sol-K_D}.
We conjecture that a rigorous passage to the limit, establishing the connection between the Porous Medium system and the associated free boundary problem in the presence of an obstacle, can be constructed. However, such an analysis lies beyond the scope of the present work.

\subsubsection{Statement of the main results}

In this section, we list the main results of our paper.  The first one is to provide existence of viscosity solutions of \eqref{e:viscosity-sol-K_D} in the sense of \cref{def:def-sol_sezione_esistenza} via Perron's method (see \cite{crandalllions-viscosity,useguide,caffarelli2005geometric}).
\begin{theorem}[Existence of viscosity solutions]
 \label{t:existence-intro} 
    Let $\phi:\partial \mathcal{B}\to\R$ be an assigned non-negative continuous function such that $\phi=0$ on $\partial \mathcal{B} \cap \partial K$.
    Let $\mathcal{A}$ be the family of supersolutions defined as follows
    $$
     \mathcal{A}:=\left\{w\in C^0(\overline{\mathcal{B}},\R^+): \ \Delta w \leq 0 \ \hbox{ in } \Omega_w\cap D,\ \abs{\nabla w}^2\le \nabla w\cdot V  \ \hbox{ on } \partial\Omega_w\cap D,\ \hbox{ and } \ \ w\equiv 0\ \hbox{ in } \ K
    \right\}.
    $$
    Then, the function $u:\overline{\mathcal{B}}\to\R^+$ defined as
    \begin{align*}
        u(x):=\inf\left\{w(x):\ w\in \mathcal{A} , \ w\ge \phi\, \hbox{ on }\,\partial\mathcal{B}\right\}
    \end{align*}
    is well-defined and it is a viscosity solution of \eqref{e:viscosity-sol-K_D}. 
    Moreover, $u$ is locally Lipschitz and it is continuous up to the boundary $\partial\mathcal{B}$, with $u=\phi$ on $\partial \mathcal{B}$.
\end{theorem}

Then, we investigate the interior regularity of the free boundary $\partial \Omega_u \cap D$ for viscosity solutions $u$ of \eqref{e:viscosity-sol-K_D}. 
Precisely, we say that $x_0$ is an {\em interior regular point} if there exists $\nu\in\partial B_1$ such that
$$
u_{x_0, r}(x) := \frac{u(x_0 + rx)}{r} \to c (x \cdot \nu)_{+} \qquad \hbox{where } c := V\cdot \nu,
$$
as $r\to 0^+$ up to a subsequence.
Specifically, around interior regular points
we prove $C^{1,\alpha}$ regularity via an improvement of flatness lemma, in the spirit of \cite{de2011free}. That is, if the solution is sufficiently flat in the direction $\nu$ at a given scale and the non-degeneracy assumption $V\cdot\nu>0$ is satisfied, then the free boundary $\partial \Omega_u \cap D$ is locally a $C^{1,\alpha}$ graph.
To obtain higher regularity, and in particular to show that $\partial \Omega_u \cap D$ is analytic, we perform a hodograph-type transformation (as introduced in \cite{kinderlehrer1977regularity}). The precise statement is given below.

\begin{theorem}[Interior regularity of the free boundary]\label{thm:piatto-implica-regolare-intro}
    Let $u:\mathcal{B}\to\R^+$ be a viscosity solution of \eqref{e:viscosity-sol-K_D} and $x_0\in\partial\Omega_u\cap D$ be an interior regular point with unit normal $\nu$. Suppose that $V\cdot \nu>0$, then:
    \begin{itemize}
        \item the free boundary $\partial\Omega_u$ is a $(d-1)$ analytic graph in a neighborhood of $x_0$;
        \item the function $u$ is analytic in a neighborhood of $x_0$ in $\overline{\Omega}_u$.
    \end{itemize}
\end{theorem}

We point out that \cref{thm:piatto-implica-regolare-intro}
does not provide any information about free boundary points $x_0 \in \partial \Omega_u$ lying on the obstacle $K$. To address boundary regularity at such contact points, we first show that all contact points $x_0 \in \partial \Omega_u \cap K \cap \mathcal{B}$ have as a blow-up limit 
a nontrivial plane, and thus they are regular. Then, we get $C^{1,\alpha}$ regularity of the free boundary around contact points via an improvement of flatness lemma in the spirit of \cite{changlarasavin}. The precise statement is the following.

\begin{theorem}[Boundary regularity for contact points]
\label{th:boundary_reg_contact_points}
    Let $u:\mathcal{B}\to\R^+$ be a viscosity solution of \eqref{e:viscosity-sol-K_D} and $x_0\in\partial\Omega_u\cap K\cap\mathcal{B}$ be a contact point. 
    If $V\cdot \nu_K(x_0)>0$, where $\nu_K(x_0)$ is the outer normal to $\partial K$ in $x_0$, then for some $\alpha \in (0,1)$, 
    $\partial\Omega_u$ is a $C^{1,\alpha}$-graph in a neighborhood of $x_0$. 
\end{theorem}
The key step to prove \cref{th:boundary_reg_contact_points} is to study the $C^{1,\alpha}$ regularity for viscosity solutions of the free boundary oblique thin obstacle problem given by
 \begin{equation}
 \label{e:thin_intro}
        \left\{
\begin{aligned}
    & \Delta v = 0 && \hbox{ in } B_{1}^+:= B_1 \cap\{x_d>0\},\\
    & \nabla v\cdot W=0 && \hbox{ on } B_{1}'\cap\Omega_v := B_1 \cap\{x_d=0\} \cap \Omega_v,\\
    & \nabla v\cdot W\le 0 && \hbox{ on } B_{1}',\\
    &v\ge0 && \hbox{ on } B_{1}',\\
\end{aligned}
        \right.
    \end{equation}
where $W$ is an oblique constant vector field, i.e.~$W_{d}:=W\cdot e_d >0$. 
The regularity theory for free boundary thin obstacle problems ($W=e_d$) has been extensively developed over the past decades \cite{caf79,ac04,acs08,gp09,fs18,fr21,sy23,cv24} and we also mention \cite{survey} as a recent survey on the topic. We remark that the $C^{1, \alpha}$ regularity for the corresponding non-local version of \eqref{e:thin_intro} is known \cite{criticaldrift}. However, their method cannot be apply in the local context and our proof is in the spirit of \cite{caf79}. Thus, the precise statement of the $C^{1,\alpha}$ regularity for \eqref{e:thin_intro} is the following.

\begin{theorem}[$C^{1,\alpha}$ regularity for the oblique thin obstacle problem]
    \label{thm:oblthin2}
    For all $\delta >0$, there are constants $C>0$ and $\alpha_0 \in (0,1)$ such that the following holds. Suppose that for some unit vector $W\in\partial B_1$, we have $W \cdot e_d\ge\delta$. 
  Let $v:\overline {B_1^+}\to\R$ be a viscosity solution of the oblique thin obstacle problem \eqref{e:thin_intro}.
     Then, $v\in C^{1,\alpha_0}(\overline{B_{\sfrac{1}{2}}^+})$ and the following estimate
     holds true
     $$
     \|v\|_{C^{1,\alpha_0}(\overline{B_{\sfrac{1}{2}}^+})}\le C\|v\|_{L^\infty(B_1^+)}.
     $$
\end{theorem}
We point out that the constant $\alpha$ in \cref{th:boundary_reg_contact_points} can be chosen as any number strictly less than the exponent $\alpha_0$ appearing in \cref{thm:oblthin2} (see \cref{th:eps_reg_bordo}). Moreover, we expect that that the optimal exponent $\alpha$ in \cref{th:boundary_reg_contact_points} coincides with the optimal $C^{1,\alpha}$ regularity that can be established for \eqref{e:thin_intro}, and such a regularity is sharp.

\subsection{Difficulties and strategy of the proof}
Here below, we discuss our main results and sketch the fundamental steps of the proof, emphasizing the principal challenges.
\\
\\
{\it Existence of solutions and interior regularity: free boundary problem is not a limit of PME.}
Differently from \cite{kim2018porous}, we do not derive \eqref{e:viscosity-sol-K_D} as a limiting case of the PME, where the pressure law becomes stiffer and stiffer (e.g., as $\gamma \to +\infty$ in a power-law equation of state).  
As a result, we cannot transfer \emph{a priori} estimates or regularity properties, 
from the approximating solutions to the limiting ones.
Since \eqref{e:viscosity-sol-K_D} is not obtained as a limiting problem, we must adopt a different approach to prove existence of solutions. Moreover, we point out that \eqref{e:viscosity-sol-K_D} does not have a variational structure, as it cannot be derived via the minimization of an energy functional.  
Inspired by the seminal result of \cite{caffarelli1988harnack} on the existence of viscosity solutions for the two-phase problem, we establish \cref{t:existence-intro} in \cref{sec:viscosity_solution_existence} via Perron's method. Specifically, we slide modify the admissible class $\mathcal{A}$ of supersolutions showing that the definition of the function $u$ in \cref{t:existence-intro} is not affected (see \cref{subsec:replacement}). Then, we prove that such a function $u$ is locally Lipschitz and harmonic in $\Omega_u \cap D$ (see \cref{subsec:u_lipschitz_armonica}). Finally, in \cref{subsec:freeboundarycond} and \cref{subsec:u_continua_fino_al_bordo}, we get that $u$ is a solution on the free boundary and it is continuous up to the boundary $\partial \mathcal{B}$.
While the theory of viscosity solutions for one-phase free boundary problems is nowadays well-developed (see \cite{wang2003existence, desilva2015perron, cv24fully}), our setting presents an additional challenge.
In particular, the non-degeneracy condition may fail due to the 
free boundary condition
\[
\abs{\nabla u}^2 = \nabla u \cdot V \qquad \text{on } \partial \Omega_{u} \cap D.
\]
The possible failure of the non-degeneracy condition does not affect the validity of the proof of \cref{t:existence-intro}, while, it plays a crucial role in the proof of \cref{thm:piatto-implica-regolare-intro}. Indeed, after a rescaling at interior free boundary points, \eqref{e:viscosity-sol-K_D} is a prototype for anisotropic one-phase free boundary problems (i.e.~the behavior of the free boundary $\partial \Omega_u$ depends on the direction of the normal vector $\nu$) of the following type
$$
\left\{
\begin{aligned}
    &\Delta w = 0&&\hbox{in }  B_1\cap \Omega_w,\\
    &\abs{\nabla w} = g(\nu)&&\hbox{on }  B_1\cap \partial\Omega_w,\\
\end{aligned}\right.
$$
where $w\ge0$ in $B_1$, $\nu$ is the interior unit normal vector to $\partial \Omega_u$ and $g$ has not a priori a sign. In our setting, we have $g(\nu) = V \cdot \nu$. If $g(\nu)\geq \delta >0$, then existence and regularity results are known \cite{cv24fully}. Thus, the natural condition to show interior regularity 
is to require 
\begin{equation}
    \label{e:non_degenearcy-intro}
    V \cdot \nu \geq \delta >0.
\end{equation}
The proof of the interior $\eps$-regularity theorem in \cref{thm:piatto-implica-regolare-intro} is divided into two steps obtained in \cref{sec:interior}. 
\begin{itemize}
    \item[(i)] {\em $C^{1,\alpha}$ regularity via iteration of the improvement of flatness \cref{lemma:IOF}.} In the spirit of \cite{de2011free}, we prove a partial Harnack inequality in \cref{subsec:partialharnack} which ensures compactness of linearized sequences. Then, in \cref{subsec:IOF}, we show that the limit function $\widetilde u$ solves the linear oblique elliptic problem
\begin{equation}\label{e:obliquo1_intro}
        \left\{
\begin{aligned}
    &\Delta \widetilde u = 0 && \hbox{ in } B_{{1}} \cap \{x \cdot \nu > 0\},\\
    & \nabla \widetilde u \cdot W =0 && \hbox{ on } B_{{1}}\cap \{x \cdot \nu = 0\},\\
\end{aligned}
        \right.
    \end{equation}
    where 
    $$W:=2(\nu \cdot V)\nu - V.$$
Thanks to the $C^{1,\alpha}$ regularity of \eqref{e:obliquo1_intro}, we close the argument in the proof of \cref{lemma:IOF}. Then, we show that the non-degeneracy assumption \eqref{e:non_degenearcy-intro} is preserved throughout the iteration of the improvement of flatness \cref{lemma:IOF}, which concludes the proof of the $\eps$-regularity theorem in \cref{eps-sub}.
\item[(ii)] {\em Higher regularity via hodograph-type transform.} In \cref{subsec:odograha}, we apply a hodograph map as in \cite{kinderlehrer1977regularity}, and we transform the interior free boundary problem into a nonlinear elliptic problem on the half ball $B_1^+$ with a nonlinear oblique boundary condition on $B_1'$. The higher regularity follows by classical Schauder estimates for oblique problems \cite{adn,morrey-multiple-integrals,lieberman2013oblique}.
\end{itemize}

\medskip

\noindent 
    {\it Boundary regularity: presence of the obstacle.} 
    Due to the presence of the non-accessible region $K$, points of the free boundary $\partial \Omega_{u}$ can touch the obstacle $K$ only from one side, which means that only the inequality
   \be\label{eq:bondary-intro}
\abs{\nabla u}^2 \ge \nabla u \cdot V \qquad \text{on } \partial\Omega_{u} \cap K\cap\mathcal{B}.
\ee
   is satisfied.
In \cref{sec:reg_contact_points}, we provide the $C^{1,\alpha}$ regularity of the free boundary $\partial\Omega_u$ near a contact point $x_0\in \partial\Omega_u\cap K\cap\mathcal{B}$. First of all, in \cref{subsec:flatproblem} (for details we refer to \cref{sec:change_coordinates}), we make a change of coordinates to make flat the boundary of the obstacle $K$. After the change of coordinates, the problem reads as
 \begin{equation*}
\left\{\begin{aligned}
&{\rm div}\left(\tens{A}(x) \nabla u\right) =0 &&\Omega_u \cap B_1^+,\\
&\tens{A}(x) \nabla u \cdot \nabla u = \tens{A}(x) \nabla u \cdot {V}(x) &&\partial \Omega_u \cap B_1^+,\\
&\tens{A}(x) \nabla u \cdot \nabla u \geq \tens{A}(x) \nabla u \cdot {V}(x) &&\partial \Omega_u \cap B_1',\\
&u = 0 &&\{x_d \leq 0\},
\end{aligned}
\right.
\end{equation*}
where $\tens{A}$ is a uniform elliptic matrix and 
the vector $V$ is no longer a constant vector, but it depends on $x$. Precisely, in the new coordinates, the obstacle $K$ is $\{x_d\le0\}$.
Then, the proof of \cref{th:boundary_reg_contact_points} proceed as follows.
\begin{itemize}
    \item[(i)] {\em Every contact point is a regular point.} In \cref{subsec:branching_points_are_regular},
    adapting the strategy from \cite[Lemma 11.17]{caffarelli2005geometric}, we prove that at every contact point $x_0$, 
    the function $u$ admits the expansion
\[
u(x) = \beta (x - x_0) \cdot e_d + o(|x - x_0|),
\]
for some $\beta \geq 0$.
However, since the non-degeneracy may fail, this expansion alone does not guarantee non trivial blow-ups. Thus, we need to rule out the degenerate case $\beta = 0$.
Precisely, we show a dichotomy: either $\beta$ is exactly equal to the specific value
\[
\omega_{x_0} := \frac{\tens{A}(x_0) V(x_0)\cdot e_d}{\tens{A}(x_0)e_d\cdot e_d},
\]
or $\beta > \omega_0 > 0$, then the free boundary is locally collapsed on the obstacle near $x_0$, i.e.~there exists a radius $\rho > 0$ such that $u > 0$ in $B_\rho(x_0)$.
We get such a bound on $\beta$ in \cref{prop:punti_regolari_al_bordo} constructing two {\em ad hoc} barriers.

    \item[(ii)] {\em $C^{1,\alpha}$ regularity of $\partial \Omega_u$ around contact points.} We point out that when the free boundary is collapsed on the obstacle, the $C^{1,\alpha}$ regularity is immediate. Thus, in \cref{subsec:iof_branching}, we only study the regularity of the free boundary near branching points, namely points like 
    \[
x_0 \in \partial \Omega_u \cap B_1' \qquad \text{such that} \qquad \forall\, r > 0:\; \partial \Omega_u \cap B_r^+(x_0)\neq \emptyset.
\]
To establish $C^{1,\alpha}$ regularity near such points, we show an improvement of flatness lemma around branching points \cref{lemma:IOF_bordo}, in the spirit of \cite{changlarasavin}. Differently form the interior version \cref{lemma:IOF}, the proof of a partial Harnack inequality is more delicate. Indeed, the competitor used in the interior partial Harnack \cref{lemma:PH_inequality} cannot be employed on one-side due to free boundary condition \eqref{eq:bondary-intro} on the obstacle. We overcome this difficulty by constructing another barrier in \cref{lemma:useful_non_piace_Giulia}. As a consequence of the partial Harnack inequality, the linearized sequences converge to some function $v$ which solves the oblique thin obstacle problem \eqref{e:thin_intro}. The $C^{1,\alpha}$ regularity for \eqref{e:thin_intro} is addressed in \cref{susec:reg_linarized_proble_bordo}, and we postpone its discussion below. With $C^{1,\alpha}$ regularity of the linearized problem, we can conclude the proof of \cref{lemma:IOF_bordo}. 
Finally, in \cref{subsec:final_bordo}, with a geometric projection argument (see also \cref{fig:eps_reg_bordo}), we get the $C^{1,\alpha}$ regularity of $\pa\Omega_u$ around branching points. Precisely, if $y\in\partial\Omega_u$ is a free boundary point near $x_0$, we can consider $y_0$, the projection of $y$ on the set of branching points and apply the boundary improvement of flatness \cref{lemma:IOF_bordo} at $y_0$. As a consequence, $y$ is an interior regular point. Combining the rate of convergence for both the interior regular points \cref{remark:interior} and for branching points \cref{lemma:proprieta_blowup_boundary}, we conclude in \cref{prop:grafico} that $\partial \Omega_u$ is a $C^{1,\alpha}$-graph around $x_0$.
\end{itemize}

\medskip

\noindent 
{\it Regularity of the linearized problem: the oblique thin obstacle problem.} The proof of the regularity of the oblique thin obstacle problem is delicate. First, the non variational structure of \eqref{e:thin_intro} requires the use of viscosity arguments to provide the $C^{1,\alpha}$ regularity. Our approach follows the strategy developed by Caffarelli \cite{caf79} (see also \cite{milakissilvestre-signorini,fullyfernandez}) for the classical thin obstacle problem ($W=e_d$). The main steps of the proof are the following.
\begin{itemize}
    \item[(i)] In \cref{lipandsemiconvest}, similarly as one must prove in the classical thin obstacle problem, we show that the function $v$ is locally Lipschitz and semiconvex in the tangential directions, namely, for every $e=(e',0)\in \mathbb{S}^{d-1}$, we have
    $$\inf_{B_{\sfrac{1}{2}}}\partial_{ee}v \ge - C\|v\|_{L^\infty(B_1)}.$$
This is obtained via viscosity arguments (construction of barriers) and by applying the Bernstein technique (see \cite{cdv20}) to estimate $(\partial_{ee} v)^{-}$, similarly as done in \cite{fullyfernandez,fj21,cc24}. As a consequence, we get that
$$\sup_{B_{\sfrac{1}{2}}^+}\partial_{dd}v \le  C\|v\|_{L^\infty(B_1)} \qquad \hbox{and } \qquad \sup_{B_{\sfrac{1}{2}}^+}\partial_{\widetilde WW}v\le C\|v\|_{L^\infty(B_1)}$$
where we denoted by $\widetilde W:=(-W',W_d)$ the adjoint direction of $W$. 
 \item[(ii)] In \cref{estimatefinal0}, for technical reason (see \cref{lemma:penalized}), we must to introduce a family of auxiliary functions $\{v_\ell\}_{\ell \in (0,1)}$ which solves \eqref{e:thin_intro} with boundary datum on $v+\ell$ on $\partial B_1$.
 We show that, since $\partial_{\widetilde W W} v_\ell$ is bounded, then it is well-defined 
$$\sigma_{W}(x'):=\lim_{t\to0^+}\nabla v_\ell(x'+t\widetilde W)\cdot W.$$
The key point to prove \cref{thm:oblthin2} is to show that the function $\sigma_{W}$ is $C^{0,\alpha}$ uniformly on $\ell$, leading the $C^{1,\alpha}$ regularity of $v_\ell$, by oblique Schauder estimates up to the boundary \cite{lieberman2013oblique}. 
We stress that, to compute the derivative of $v_\ell$ along the direction $W$ on the thin space $\{ x_d = 0 \}$, it is necessary to approach the thin space along an oblique direction. Since $\partial_{W W} v_\ell$ is generally unbounded by above, the direction $W$ itself cannot be used to approach $\{ x_d = 0 \}$. Therefore, we introduce an adjoint direction $\widetilde{W}$, for which we know that $\partial_{\widetilde{W} W} v_\ell$ is bounded by above.

The viscosity condition $\nabla v_\ell\cdot W\le0$ in $B_1'$ suggests that $\sigma_{W}(x')\le0$ on $B_1'$. This is proved by a penalization argument, similarly as done in \cite{milakissilvestre-signorini}.
A fundamental ingredient to get the desired regularity is to ``link'' the normal derivative with the one along $W$. Precisely, in \cref{rem:q.o.}, as a consequence of the semiconvexity of $v_\ell$, we show that, for almost every point $(x',0)$ on the contact set $\Lambda(v_\ell):=\{v_\ell=0\}\cap B_1'$, it holds
$$\sigma_d(x'):=\lim_{t\to0^+}\partial_{d}v_\ell(x'+t\widetilde W)=\sigma_{W}(x').$$
\item[(iii)] In \cref{estimatefinal}, we show that, given a point $x_0\in\Omega_{v_\ell}\cap B_1'$ and $\gamma>0$ small, we can find a small ball in $\{x_d=0\}$ of radius comparable to $\gamma$ which is almost everywhere contained in $S_{\gamma}:=\{\sigma_W(x')>-\gamma\}$ (see \cref{lemmab7}).
This means that the derivative along $W$ is not too negative somewhere near points belonging to $\Omega_{v_\ell}\cap B_1'$ (we remark that $\sigma_W\equiv0$ in $\Omega_{v_\ell}\cap B_1'$, by oblique elliptic regularity). Even though, differently from the proof of the $C^{1,\alpha}$ regularity of the classical thin obstacle problem in \cite{caf79}, where the small ball is entire contained in $S_{\gamma}$ (not just almost everywhere), we can still apply a similar barrier argument for harmonic functions with boundary data defined almost everywhere (see \cref{teclemma-w}).
Combining the above barrier arguments with the semiconcave estimates, we get the $C^{0,\alpha}$ regularity of $\sigma_{W}$ in $\Omega_{v_\ell}$, which concludes the proof of \cref{thm:oblthin2}.

\end{itemize}
Throughout the paper, we do not explicitly indicate the dependence of constants $C$ on the dimension $d$ of the ambient space, nor on the geometric parameters of the problem, such as $\mathcal{B}$, $D$, $K$, and $\abs{V}$. Whenever a different dependence is present, it will be specified.

\section{Existence of viscosity solutions}
\label{sec:viscosity_solution_existence}
The aim of this section is to construct a function $u$, with a given boundary datum, which is a viscosity solution of the problem \eqref{e:viscosity-sol-K_D} in the following sense.

\begin{definition}[Viscosity solutions of \eqref{e:viscosity-sol-K_D}]
\label{def:def-sol_sezione_esistenza}
Let $u: \mathcal{B} \to \R^+$ be a non-negative continuous function. We say that {\em $u$ is a viscosity solution} of the problem \eqref{e:viscosity-sol-K_D} 
if $u\equiv0$ on $K$ and for every $x_0 \in \overline\Omega_u \cap \mathcal{B}$ and $\varphi \in C^{\infty}(\mathcal{B})$, we have
\begin{itemize}
    \item if $x_0 \in \Omega_u\cap D$ and $\varphi$ touches $u$ from below at $x_0$, then $\Delta \varphi(x_0) \leq 0$;
    \item if $x_0 \in \Omega_u\cap D$ and $\varphi$ touches $u$ from above at $x_0$, then $\Delta \varphi(x_0) \geq 0$;
    \item if $x_0 \in \pa \Omega_u \cap D$ and $\varphi$ touches $u$ from below at $x_0$, then $\abs{\nabla \varphi(x_0)}^2\le\nabla \varphi(x_0)\cdot V $;
    \item if $x_0 \in \pa \Omega_u \cap D$ and $\varphi^+:= \max\{\varphi,0\}$ touches $u$ from above at $x_0$, then $\abs{\nabla \varphi(x_0)}^2\ge\nabla \varphi(x_0)\cdot V$.
    \item if $x_0 \in \pa \Omega_u \cap  K\cap\mathcal{B}$ and $\varphi^+:= \max\{\varphi,0\}$ touches $u$ from above at $x_0$, then $\abs{\nabla \varphi(x_0)}^2\ge\nabla \varphi(x_0)\cdot V $.
\end{itemize}
\end{definition}

In order to construct a viscosity solution $u$ according to \cref{def:def-sol_sezione_esistenza}, we introduce the class of admissible functions we take into account.

\begin{definition}[Admissible class $\mathcal{A}$ of supersolutions]
\label{def:a-supersolution}
We define the class of admissible supersolutions $\mathcal{A}$ for the problem \eqref{e:viscosity-sol-K_D} the set
\begin{align*}
    \mathcal{A}:=\left\{w\in C^0(\overline{\mathcal{B}},\R^+): \ \Delta w \leq 0 \ \hbox{ in } \Omega_w\cap D,\ \abs{\nabla w}^2\le \nabla w\cdot V  \ \hbox{ on } \partial\Omega_w\cap D,\ \hbox{ and } \ \ w\equiv 0\ \hbox{ in } \ K
    \right\}
\end{align*} 
where the two inequalities hold true in the viscosity sense.
\end{definition}

The main result of this section is the following theorem, which is exactly \cref{t:existence-intro}.
    \begin{theorem}[Existence of viscosity solutions]
    \label{t:existence} 
    Let $\phi:\partial \mathcal{B}\to\R$ be an assigned non-negative continuous function such that $\phi=0$ on $\partial \mathcal{B} \cap \partial K$.
    Let $\mathcal{A}$ be the set of supersolutions defined in \cref{def:a-supersolution}.
    Then, the function $u:\overline{\mathcal{B}}\to\R^+$ defined as
    \begin{align}
    \label{e:funzione_u_teorema_esistenza}
        u(x):=\inf\left\{w(x):\ w\in \mathcal{A} , \ w\ge \phi \hbox{ on }\partial\mathcal{B}\right\}
    \end{align}
    is well-defined and it is a viscosity solution to \eqref{e:viscosity-sol-K_D} in the sense of \cref{def:def-sol_sezione_esistenza}. Moreover, $u$ is locally Lipschitz and it is continuous up to the boundary $\partial\mathcal{B}$, with $u=\phi$ on $\partial \mathcal{B}$.
\end{theorem}
\begin{proof}
    The proof is a combination of the following steps split in four subsections:
    \begin{itemize}
\item {\it replacement of the supersolutions $w \in \mathcal{A}$}, shown in \cref{subsec:replacement};
\item {\it $u$ is locally Lipschitz and harmonic in $\Omega_u \cap D$}, obtained in \cref{subsec:u_lipschitz_armonica};
\item {\it $u$ is a solution on the free boundary}, in particular we prove in \cref{subsec:freeboundarycond} that both $\abs{\nabla u}^2 = \nabla u \cdot V$ on $\partial \Omega_u \cap D,$ and $\abs{\nabla u}^2 \ge \nabla u \cdot V$ on $\partial \Omega_u \cap K\cap\mathcal{B}$ hold true;
\item {\it  $u$ is continuous up to the boundary $\partial \mathcal{B}$}, obtained in \cref{subsec:u_continua_fino_al_bordo}.
\end{itemize}
Thus, combining the above steps, the function $u$ is a viscosity solution to \eqref{e:viscosity-sol-K_D} in the sense of \cref{def:def-sol_sezione_esistenza}, yielding the thesis. 
\end{proof}

The function $u$ in \eqref{e:funzione_u_teorema_esistenza} is called a {\em Perron's solution} of \eqref{e:viscosity-sol-K_D}. Moreover, we notice that all viscosity solutions $u$ of \eqref{e:viscosity-sol-K_D} have the following graphical property.

\begin{lemma}[Graphical property of the free boundary]\label{lemma:graph}
    Let $u:\mathcal{B}\to\R^+$ be a viscosity solution of \eqref{e:viscosity-sol-K_D}. Suppose that $x_0\in \Omega_u\cap D,$ then, for every $t>0$ such that $x_0+tV\in D$, we have that $x_0+tV\in\Omega_u$. Equivalently, if $x_0\in D$ is such that $u(x_0)=0$, then, for every $t>0$ such that $x_0-tV\in D,$ we have that $u(x_0-tV)=0$. In particular, $\partial\Omega_u\cap D$ is a graph.
\end{lemma}
\begin{proof}
    Take $x_0\in \Omega_u\cap D$ and $\rho>0$ such that $B_\rho(x_0)\subset \Omega_u\cap D$. We need to prove that for all $x \in B_\rho(x_0)$ and for every $t>0$ such that $x+tV\in D$, then $x+tV\in \Omega_u$. 
    Arguing by contradiction, let $t_0>0$ be the smallest $t$ such that $\overline B_\rho(x_0+ t_0 V)\cap \partial \Omega_u \cap D = \{y_0\}$. By defining $z_0 := x_0+ t_0 V$, we have  
    \be\label{eq:utiledopo}
    (z_0-y_0)\cdot V\le0.\ee
    We take the function 
    $$
    \psi: \overline B_\rho(z_0) \to \R^+ \quad \hbox{ such that } 
    \quad \left\{
\begin{aligned}
    \Delta \psi &= 0 &&\text{in } B_\rho(z_0) \setminus \overline B_{\sfrac{\rho}{2}}(z_0),\\
    \psi&= 0
    &&\text{on } \partial B_\rho(z_0),\\
    \psi &= 1 &&\text{in } \overline B_{\sfrac{\rho}{2}}(z_0).
 \end{aligned}
\right.$$ 
Since $u>0$ in $B_\rho(z_0),$ then there is a constant $c>0$ such that $u\ge c\psi$ in $B_{\sfrac{\rho}{2}}(z_0)$. Then, by comparison principle, $u\ge c\psi$ in $B_\rho(z_0)$. Thus, $c\psi$ is a test function touching $u$ from below at $y_0$, implying that 
$$c^2|\nabla \psi(y_0)|^2\le c\nabla\psi(y_0)\cdot V.$$ 
On the other hand, by Hopf lemma, 
$\nabla \psi(y_0)=\beta(z_0-y_0)$, with $\beta>0$. By \eqref{eq:utiledopo}, we end up with
$$c^2\beta^2|z_0-y_0|^2\le c\beta(z_0-y_0)\cdot V\le0,$$
yielding to a contradiction since $\beta>0$.
\end{proof}

\subsection{Replacement of supersolutions}
\label{subsec:replacement}
In this subsection, we show that we can modify the class of admissible supersolutions according to the following proposition.
\begin{proposition}
\label{lemma:wlog} 
Let $w\in\mathcal{A}$ be an admissible supersolution to \eqref{e:viscosity-sol-K_D}, according with \cref{def:a-supersolution}. Then, there exists $\overline w\in\mathcal{A}$, with 
\be\label{eq:wlog0}\overline w\le w\quad\text{in }\overline{\mathcal{B}}\qquad \text{and} \qquad \overline w=w\quad\text{on }\partial\mathcal{B},\ee
such that
\be\label{eq:wlog1}\Delta \overline w=0\quad\text{in }\Omega_{\overline w}\cap D.\ee 
\end{proposition}

By \cref{lemma:wlog}, we can introduce the following subclass of admissible supersolutions to \eqref{e:viscosity-sol-K_D}.

\begin{definition}[Modified admissible class $\mathcal{A}^+$ of supersolutions]
\label{def:a-supersolution-modificata}
We define
\begin{align*}
    \mathcal{A}^+:=\left\{ w \in \mathcal{A}: \ \Delta w  = 0 \ \hbox{ in } \Omega_w\cap D
    \right\},
\end{align*} 
to be the modified admissible class of supersolutions for the problem \eqref{e:viscosity-sol-K_D}.
\end{definition}

As an immediate consequence of \cref{lemma:wlog} is the following corollary.
\begin{corollary}\label{corollario:usosempre}
    Let $u:\overline{\mathcal{B}}\to\R^+$ be the function defined in \eqref{e:funzione_u_teorema_esistenza}. Then
 $$u(x)=\inf\left\{w(x):\ w\in \mathcal{A}^+ , \ w\ge \phi \hbox{ on }\partial\mathcal{B}\right\} .$$
\end{corollary}

The following lemma will be useful throughout this section whenever we define a harmonic function with Dirichlet boundary conditions.

\begin{lemma}[\protect{\cite[Theorem 6.6]{han2011elliptic}}]\label{lemma:classical-perron}
Let $\Omega \subset \R^d$ be a bounded open set and let $\zeta: \partial \Omega \to \R^+$ be a continuous assigned function. Then, there exists $g:\overline \Omega\to\R^+$ such that 
$$g(x):=\inf\left\{v(x): v\in C^0(\overline\Omega,\R^+): \ \Delta v\le 0 \text{ in }\Omega ,\ v\ge \zeta\text{ on }\partial \Omega\right\},$$
which solves
$$\Delta g=0.$$
Moreover, if $x_0 \in \partial \Omega$ is such that 
$\partial\Omega$ is smooth in a neighborhood of $x_0$,
then
$$
\lim_{x\to x_0} g(x) = g(x_0) = \zeta(x_0).
$$
In particular, if $\partial\Omega$ is smooth, then $g\in C^0(\overline\Omega)$ and $ g= \zeta$ on $\partial \Omega$.
\end{lemma}

To show \cref{lemma:wlog}, we first provide the next lemma, which will be fundamental in the rest of the section.

\begin{lemma}\label{lemma:min2soprasol} 
Let $\mathcal{A}$ be the class of admissible supersolutions defined in \cref{def:a-supersolution} and let $w_1\in\mathcal{A}$. Let $B \subset \mathcal{B}$ be a smooth bounded open set and let $w_2$ be such that
$$
\begin{aligned}
    w_2 \in C^0(\overline{B}), \quad \Delta w_2 \leq 0  \text{  in  } \Omega_{w_2} \cap D, \quad \abs{\nabla w_2}^2 \leq \nabla w \cdot V \hbox{ on  } \partial \Omega_{w_2} \cap D.
\end{aligned}
$$
Then:
\begin{itemize}
    \item[(i)] if $B=\mathcal{B}$, then $w:=\min\left\{w_1,w_2\right\}\in\mathcal{A}$;
    \item[(ii)] if $\overline B\subset\mathcal{B}$ and $w_2\ge w_1$ on $\partial B$, then $w:=\min\left\{w_1,w_2\right\}\in\mathcal{A}$.
\end{itemize}
\end{lemma}

\begin{proof}
We observe that in both cases $w$ is continuous in $\overline{\mathcal{B}}$.  
Next, we show that $\Delta w\le 0$ in $\Omega_{w}\cap D$ and $|\nabla w|^2\le \nabla w\cdot V$ on $\partial \Omega_w\cap D$.
Let $\varphi\in C^\infty(\mathcal{B})$ be a test function such that $\varphi$ touches $w$ from below at $x_0\in \overline \Omega_w\cap D$. Without loss of generality, we can suppose that $w(x_0)=w_1(x_0)$, implying that $\varphi$ touches $w_1$ from below at $x_0$. We have two cases.
\begin{itemize}
    \item If $x_0\in \Omega_{w}\cap D$, then $x_0\in\Omega_{w_1}\cap D$. Since $\varphi$ touches $w_1$ from below at $x_0$, then, since $w_1 \in \mathcal{A}$, $\Delta \varphi(x_0)\le 0$.
    \item If $x_0\in \partial \Omega_{w}\cap D$, then $x_0\in\partial\Omega_{w_1}\cap D$. Indeed, there is a sequence $x_j\in\Omega_{w}\cap D$ such that $x_j\to x_0$, implying that $w_1(x_j)\ge w(x_j)>0$. Since $\varphi$ touches $w_1$ from below at $x_0$ and $w_1 \in \mathcal{A}$, then $|\nabla \varphi(x_0)|^2\le \nabla \varphi(x_0)\cdot V$.
\end{itemize}
Finally, since $w_1\equiv0$ on $K$, then $w\equiv0$ on $K$.
\end{proof}
We are finally in position to show the main result of this subsection.
\begin{proof}[Proof of \cref{lemma:wlog}]
We need to show that if $w\in\mathcal{A}$, then exists $\overline w\in\mathcal{A}$, satisfying  \eqref{eq:wlog0} and \eqref{eq:wlog1}.
We cannot directly apply \cref{lemma:classical-perron} since we do not know a priori the regularity of $\partial \Omega_w \cap D$. Thus, we proceed as follows.

Let $\mathcal{T}:= \Omega_w \cap D$ and let
$$ \mathcal{V}:=\left\{v\in C^0(\overline{\mathcal{T}},\R^+): \ \Delta v\le 0 \text{ in }\mathcal{T},\ v\ge w\text{ on }\partial \mathcal{T}\right\}.$$ 
Take 
$$\overline w(x):=\inf\left\{v(x): v\in\mathcal{V}\right\},$$ 
which is well defined since $w\in\mathcal{V}$.
Since $\mathcal{T}$ is an open bounded set, then, by \cref{lemma:classical-perron}, we have that
\be\label{overlinewisharmonic}\Delta \overline w=0\quad\text{in }\mathcal{T}.\ee 
 To show \eqref{eq:wlog0}, we need to prove that
\be\label{eq:tesilemma1}
\overline w\in C^0(\overline{\mathcal{T}})\quad\text{and}\quad \overline w=w\quad\text{on } \partial \mathcal{T}\ee
Since $\mathcal{B}$ and $K$ are smooth, by \cref{lemma:classical-perron}, it is enough to show that $\overline w$ is continuous up to the boundary $\partial \Omega_w\cap D$. 
Let $x_0\in \partial \Omega_w\cap D$ and take a sequence $x_j\to x_0$.
Since $\overline w\le w$, then $0\le\overline w(x_0)\le w(x_0)=0$. Moreover, $$0\le \lim_{j\to+\infty}\overline w(x_j)\le \lim_{j\to+\infty} w(x_j)=w(x_0)=0.$$ 
Therefore \eqref{eq:tesilemma1} holds true.
With a slight abuse of notation, we denote
\begin{equation}
    \label{e:estensione}
    \overline{w}=\left\{
\begin{aligned}
    &\overline{w} &&\text{in } \overline{\mathcal{T}},\\
    &0 &&\text{in } \overline{\mathcal{B}} \setminus \mathcal{T}.
\end{aligned}\right.
\end{equation}
By \eqref{e:estensione}, we observe that $\overline w \equiv 0$ in $K$. Next, combining \eqref{eq:tesilemma1} and \eqref{e:estensione}, we have \eqref{eq:wlog0}. 
In addition, since $\Omega_{\overline w}\cap D=\mathcal{T}$, then \eqref{overlinewisharmonic} is exactly \eqref{eq:wlog1}. To conclude, 
it is left to show that $\abs{\nabla \overline{w}}^2 \le \nabla \overline w\cdot V$ on $\partial\Omega_{\overline w}\cap D$.
Let $\varphi\in C^\infty(\mathcal{B})$ be a test function such that $\varphi$ touches $\overline w$ at $x_0\in \partial\Omega_{\overline w}\cap D$. By construction of $\overline w$, we have that $\Omega_{\overline w} = \Omega_{w}$, implying that $x_0\in \partial\Omega_{\overline w}\cap D= \partial\Omega_{w}\cap D$.
Since $\overline w\le w$, then $\varphi$ touches $w$ at $x_0$, thus
$$|\nabla \varphi(x_0)|^2\le \nabla \varphi(x_0)\cdot V,$$
which is the desired conclusion and it yields the thesis.
\end{proof}

\subsection{The function \texorpdfstring{$u$}{u} is Lipschitz and harmonic in \texorpdfstring{$\Omega_u\cap D$}{omegau}}
\label{subsec:u_lipschitz_armonica}

In this subsection, we prove the following proposition.
\begin{proposition}
\label{lemma:u_viscotity_sol_lipschitz}
     Let $u:\overline{\mathcal{B}}\to\R^+$ be the function defined in \eqref{e:funzione_u_teorema_esistenza}, then $u$ is locally Lipschitz in $\mathcal{B}$. Precisely, for every compact set $\mathcal{C}\subset\mathcal{B}$, there exists a constant $L = L(\mathcal{C},\|u\|_{L^\infty(\mathcal{B})})>0$ such that
     \begin{equation}
     \label{e:lipschitz_u_proposizione_enunciato}
         |u(x)-u(y)|\le L|x-y|\quad\text{for every}\quad x,y\in \mathcal{C}.
     \end{equation}
    Moreover, 
    \be\label{eq:u-harmonic-in-Omegau}\Delta u=0 \quad\text{in }\Omega_u\cap D.\ee
\end{proposition}
To obtain the Lipschitz inequality \eqref{e:lipschitz_u_proposizione_enunciato}, we use the following lemma, which is a classical result of regularity theory for elliptic PDEs.

\begin{lemma}\label{lemma:classical-result-pde}
   Let $B,\Omega \subset \R^d$ be two bounded open sets and let $f:B\cap\Omega \to \R^+$ be a continuous harmonic function in $B\cap\Omega$. 
   Take a compact set $\mathcal{C}\subset B\cap\overline\Omega$ and suppose that there are constants $C_0>0$ and $0<\eta<\text{dist}(\mathcal{C},\partial B)$
   such that
    \be\label{eq:stima-lip}
    f(x) \leq C_0 {\rm dist}(x,\partial \Omega),
    \ee
    for every $x \in \mathcal{C}$ such that ${\rm dist}(x,\partial \Omega)\le \eta$. Then $f$ is Lipschitz in $\mathcal C$. 
    Precisely,
    \begin{equation}
        \label{e:lipschitz_f}
       \exists \, L = L(d, B,\mathcal{C}, C_0, \norm{f}_{L^\infty\left(B_1\cap \Omega\right)})>0 \quad \hbox{such that } \quad   |f(x)-f(y)|\le L|x-y|\quad\text{for every}\quad x,y\in\mathcal{C}.
    \end{equation}
\end{lemma}
\begin{proof}
    Let $x, y \in \mathcal{C}$ and $x_0, y_0$ be the projection on $\partial \Omega$ of $x, y$ respectively. Precisely
    $$
    d_x:=\abs{x- x_0} = {\rm dist}(x, \partial \Omega) \qquad \hbox{ and }\qquad d_y:=\abs{y-y_0} = {\rm dist}(y, \partial \Omega).
    $$
    We set $\eta_0 := \text{dist}(\mathcal{C},\partial B)$ and $\Lambda:=\text{diam}(B)$.
    Without loss of generality, we can suppose that $d_y\le d_x$ and $|x-y|\le \frac{\eta \eta_0}{2\Lambda}$. Indeed, for $\abs{x-y} \geq \frac{\eta \eta_0}{2\Lambda}$, \eqref{e:lipschitz_f} is a consequence of the boundedness of the function $f$ in $B \cap \Omega$. In the other case, we need to analyze three cases.
\begin{itemize}
    \item \textit{Case 1. $d_x \geq {\eta}$.} In this case 
    $$|x-y|\le \frac{\eta\eta_0}{2 \Lambda}\le \frac{d_x\eta_0}{2 \Lambda},$$ 
    namely $y\in B_{\frac{d_x\eta_0}{2 \Lambda}}(x)$.
    Since $ d_x\le \Lambda$, then $\frac{d_x\eta_0}{\Lambda}\le \eta_0$ and thus $B_{\frac{d_x\eta_0}{\Lambda}}(x)\subset B$. Moreover, $B_{\frac{d_x\eta_0}{\Lambda}}(x)\subset B_{d_x}(x)\subset\Omega$, then $f$ is harmonic in $B_{\frac{d_x\eta_0}{\Lambda}}(x)$. By classical gradient estimates for harmonic functions, we have that
$$
    \norm{\nabla f}_{L^\infty\Big(B_{\frac{d_x \eta_0}{2\Lambda}}(x)\Big)}\leq\frac{C \Lambda}{d_x\eta_0} \norm{f}_{L^\infty\Big(B_{\frac{d_x \eta_0}{\Lambda}}(x)
   \Big)} \le \frac{C \Lambda}{\eta \eta_0}.
$$ 
By the fundamental theorem of calculus, we obtain
$$
\abs{f(x)- f(y)} \leq \norm{\nabla f}_{L^\infty\left(B_{\frac{d_x \eta_0}{2\Lambda}}(x)\right)}\abs{x-y}\le \frac{C \Lambda}{\eta\eta_0}|x-y|.
$$
    \item \textit{Case 2. $d_x\le\eta$ and $|x-y|\ge \sfrac{d_x}{2} \geq \sfrac{d_y}{2}$.}
    By \eqref{eq:stima-lip}, we have that
$$\abs{f(x)-f(y)}\le \abs{f(x)}+\abs{f(y)}\le C_0d_x+C_0d_y\le 4C_0|x-y|.$$
\item \textit{Case 3. $d_x \leq \eta$ and $\abs{x-y}\leq \sfrac{d_x}{2}$.} In this case $y\in B_{\sfrac{d_x}{2}}(x)$.
Since $B_{d_x}(x) \subset B_{\eta_0}(x)\subset B$, then $f$ is harmonic in $B_{d_x}(x) \subset\Omega$. By classical gradient estimates for harmonic functions, we have that
$$
    \norm{\nabla f}_{L^\infty\left(B_{\sfrac{d_x}{2}}(x)\right)}\leq C \frac{\norm{ f}_{L^\infty\left(B_{d_x}(x)\right)}}{d_x}.
$$
Let $z\in\overline B_{d_x}(x)$ be such that $f(z)=\norm{ f}_{L^\infty\left(B_{d_x}(x)\right)}$. Then, by \eqref{eq:stima-lip}, we get
$$
\norm{\nabla f}_{L^\infty\left(B_{\sfrac{d_x}{2}}(x)\right)}\leq C \frac{f(z)}{d_x} \leq CC_0 \frac{{\rm dist}(z, \partial \Omega)}{d_x} \leq CC_0 \frac{\abs{z-x_0}}{d_x} \leq 2CC_0.
$$
By the fundamental theorem of calculus, we have
$$
\abs{f(x)- f(y)} \leq \norm{\nabla f}_{L^\infty\left(B_{\sfrac{d_x}{2}}(x)\right)}\abs{x-y}\le 2CC_0|x-y|.
$$
\end{itemize}
By choosing $L:=\max\left\{\frac{C\Lambda}{\eta\eta_0},4C_0, 2CC_0\right\}$, the thesis holds true.
\end{proof}

In the following, we are going to prove that in our setting \eqref{eq:stima-lip} holds true. 

\begin{lemma}
\label{lemma:w_viscotity_sol_lipschitz}
     Let $w\in\mathcal{A}^+$ be an modified admissible supersolution to \eqref{e:viscosity-sol-K_D}, according with \cref{def:a-supersolution-modificata}. 
     Then, for every compact set $\mathcal{C}\subset \mathcal{B}\cap \overline\Omega_w$, there exists a constant $C_0=C_0\left(\mathcal{C},\|w\|_{L^\infty(\mathcal{B})}\right)>0$  such that
     \be\label{eq:stima-parto} 
     w(x)\le C_0\,\text{dist}\left(x,\partial \Omega_w\right)
     \ee 
     for every $x \in \mathcal{C}$ such that ${\rm dist}(x,\partial \Omega_w)\le \frac{\text{dist}(\mathcal{C},\partial \mathcal{B})}{2}$.
\end{lemma}

\begin{proof} 
First of all we notice that
$$
\partial \Omega_w\cap\mathcal{B} = \left(\partial \Omega_w \cap D \right) \cup \left(\partial \Omega_w \cap K\right).
$$
Thus, we need to show \eqref{eq:stima-parto} in two cases: either the projection of $x\in\mathcal{C}$ on $\partial\Omega_w$ belongs to $D$ or it belongs to $K$.
We divide the proof into two steps.\\
\\
{\it Step 1. Projection on $\partial\Omega_w \cap D$.}
We aim to prove that there exists a constant $C_0>0$ such that 
\be\label{tesi-lipschitz}
w(x_0)\le C_0\text{dist}(x_0,\partial\Omega_w)\ee
for every $x_0 \in \mathcal{C}$
 such that 
 \be\label{eq:distaze-uguali}\text{dist}(x_0,\partial\Omega_w) = \text{dist}(x_0,\partial\Omega_w\cap D)\le \frac{\text{dist}(\mathcal{C},\partial \mathcal{B})}{2}.\ee
 Let us fix $x_0\in\mathcal{C}$ such that \eqref{eq:distaze-uguali} holds true, and call $r:=\text{dist}(x_0,\partial\Omega_w)$.
 We consider
 $$w_r(x):=\frac{w(x_0+rx)}{r},$$ 
 and we observe that $w_r$ is harmonic in $B_1$.
 By Harnack inequality, there exists a dimensional constant $c>0$ such that 
\bea w_r(x)&\ge c w_r(0)\quad\text{for every}\quad x\in \overline B_{\sfrac{1}{2}}.\eea
Let us define $$
    \chi: \overline B_1 \to \R^+ \quad \hbox{ such that } 
    \quad \left\{
\begin{aligned}
    \Delta \chi &= 0 &&\text{in } B_1 \setminus \overline B_{\sfrac{1}{2}},\\
    \chi&= 0  
    &&\text{on } \partial B_1,\\
    \chi &= 1 &&\text{in } \overline B_{\sfrac{1}{2}}.
 \end{aligned}
\right.$$ 
Since $w_r\ge c w_r(0)\chi$ in $\partial B_1\cup \overline B_{\sfrac{1}{2}}$, then, by the maximum principle 
\be\label{eq:vtoucheswr}w_r(x)\geq c w_r(0)\chi(x) := v(x) \quad\text{for every}\quad x\in \overline B_1.\ee
Let $z_0\in\partial\Omega_w\cap D$ be such that 
$$\text{dist}(x_0,\partial\Omega_w)=|x_0-z_0| = r.$$ 
Let us call $y_0:=\frac{z_0-x_0}{r}\in\partial B_1$.
By \eqref{eq:vtoucheswr} and $v(y_0)=w_r(y_0)=0$, we have that $v$ is a test function touching the supersolution $w_r$ from below at $y_0\in\partial\Omega_{w_r}\cap D$. Thus, since $\chi$ is a radial function, we have that
    $$cw_r(0)\|\nabla \chi(y_0)\|_{L^\infty(\partial B_1)}=cw_r(0)|\nabla \chi(y_0)|=|\nabla v(y_0)|\le \frac{\nabla v(y_0)}{|\nabla v(y_0)|}\cdot V\le |V|.$$ 
    Then, $w_r(0)\le C_0$, which is exactly \eqref{tesi-lipschitz}.
\\
\\
\textit{Step 2. Projection on $\partial\Omega_w \cap K$.}
We prove that there exists a constant $C_0>0$ such that 
\be\label{tesi-lipschitz2}w(x_0)\le C_0\text{dist}(x_0,\partial\Omega_w)\ee 
for every $x_0\in\mathcal{C}$ such that
\be\label{eq:dist-su-k}\text{dist}(x_0,\partial\Omega_w)=\text{dist}(x_0,\partial\Omega_w\cap K)\le \frac{\text{dist}(\mathcal{C},\partial \mathcal{B})}{2}.\ee
Let us fix $x_0\in\partial\Omega_w\cap K$ such that \eqref{eq:dist-su-k} holds true, and call $r:=\text{dist}(x_0,\partial\Omega_w)$.
Let us consider $\widetilde w:\overline{\mathcal{B}}\to\R^+$ defined as $$
 \left\{
\begin{aligned}
\Delta \widetilde w &= 0 &&\text{in } D,\\ 
\widetilde w &= w &&\text{on } \partial D,\\
\widetilde w &= 0
    &&\text{in } K,
 \end{aligned}
\right.$$ 
which is well-defined by \cref{lemma:classical-perron}.
We notice that $\widetilde w$ is Lipschitz in 
$$\mathcal{D}:=\left\{x \in \mathcal{B}: {\rm dist}(x, \partial \mathcal{B})> \frac{{\rm dist}(\mathcal{C}, \partial \mathcal{B})}{2}\right\},$$
with the following estimate
\bea
\norm{\nabla \widetilde{w}}_{L^\infty(\mathcal{D})}\leq C \norm{\widetilde{w}}_{L^\infty(\mathcal{B})} \leq C \norm{w}_{L^\infty(\mathcal{B})}\le C,
\eea by classical results up to the boundary for harmonic functions.
Moreover $B_r(x_0) \subset \mathcal{D}$. Indeed, let $y\in B_r(x)$ and $y_0\in\partial\mathcal{B}$ be such that $|y-y_0|=\text{dist}(y,\partial\mathcal{B})$, then we have
$$
\begin{aligned}
    \text{dist}(y,\partial\mathcal{B})&=|y-y_0|\ge |x_0-y_0|-|x_0-y|\ge \text{dist}(x_0,\partial\mathcal{B})-r\ge \text{dist}(\mathcal{C},\partial\mathcal{B})-\frac{\text{dist}(\mathcal{C},\partial\mathcal{B})}2 =\frac{\text{dist}(\mathcal{C},\partial\mathcal{B})}2.
\end{aligned}$$
Then $$\|\widetilde w\|_{L^\infty(B_r(x_0))}\le\|\widetilde w\|_{L^\infty(\mathcal{D})}\le C.$$
Let $z_0\in\partial\Omega_w\cap K$ be such that 
$$\text{dist}(x_0,\partial\Omega_w)=|x_0-z_0| = r,$$
then
$$\widetilde w(x_0)=\widetilde w(x_0)-\widetilde w(z_0)\le C|x_0-z_0|=C\text{dist}(x_0,\partial\Omega_w).$$ 
We observe that $w$ is subharmonic in $D$, indeed, if $\varphi\in C^\infty(D)$ touches $w$ in a point $z\in D$, either $z\in\Omega_w$ and thus $\Delta \varphi(z)\ge0$ by the harmonicity of $w$, or $z\in\{w=0\}$ and thus $\Delta \varphi(z)\ge0$, since $\varphi$ has a minimum point at $0$.
Then, since $w = \widetilde{w}$ on $\partial D$, by the maximum principle, we have that  $w\le\widetilde w$. Thus,
$$w(x_0)\le \widetilde w(x_0)\le C\text{dist}(x_0,K)\le C\text{dist}(x_0,\partial\Omega_w\cap K),$$ having used in the last inequality the fact that $\partial\Omega_w\cap K\subset K$. This is exactly \eqref{tesi-lipschitz2}. 
\end{proof}

Thus, we get the Lipschitz regularity for any replacement $w \in \mathcal{A}^+$ obtained in \cref{lemma:wlog} and for the function $u$ defined in \eqref{e:funzione_u_teorema_esistenza} in the following corollary.

\begin{corollary}\label{corollary:lipw}
    Let $w\in\mathcal{A}^+$ be a modified admissible supersolution to \eqref{e:viscosity-sol-K_D}, according with \cref{def:a-supersolution-modificata}. Let $u:\overline{\mathcal{B}}\to\R^+$ be the function defined in \eqref{e:funzione_u_teorema_esistenza}.
     Then, for every compact set $\mathcal{C}\subset \mathcal{B}$, there exists a constant $L=L\left(\mathcal{C},\|w\|_{L^\infty(\mathcal{B})}\right)>0$  such that
     \begin{equation}
         \label{e:w_lipschitz_enunciato}
         |w(x)-w(y)|\le L|x-y|\quad\text{for every}\quad x,y\in\mathcal{C}
     \end{equation}
     and
     \begin{equation}
         \label{e:u_lipschitz_enunciato}
         |u(x)-u(y)|\le L|x-y|\quad\text{for every}\quad x,y\in\mathcal{C}.
     \end{equation}
\end{corollary}
\begin{proof}
    On one hand, as a direct consequence of \cref{lemma:classical-result-pde} and \cref{lemma:w_viscotity_sol_lipschitz}, we get \eqref{e:w_lipschitz_enunciato} for any replacement $w \in \mathcal{A}^+$.
   On the other hand, to prove \eqref{e:u_lipschitz_enunciato}, we observe that for every $k\in\N$ and for every $y\in\mathcal{C}$, there exists a replacement $w_{k,y}\in\mathcal{A}^+$ such that
         $$w_{k,y}(y)\le u(y)+\frac{1}{k},$$
as a consequence of \cref{corollario:usosempre}.         
Then, applying \eqref{e:w_lipschitz_enunciato}, for every $x,y\in\mathcal{C}$, we have
         $$u(x)-u(y)\le w_{k,y}(x)-w_{k,y}(y)+\frac{1}k\le L|x-y|+\frac{1}{k},$$
         letting $k\to+\infty$, the thesis yield.
\end{proof}

Since the replacements $w \in \mathcal{A}^+$ and the function $u$ are equi-Lipschitz,
then we can locally approximate $u$ as uniform limit of $\{w_k\}\subset\mathcal{A}^+$.

\begin{lemma}\label{corollary:ex0}
     Let $u:\overline{\mathcal{B}}\to\R^+$ be the function defined in \eqref{e:funzione_u_teorema_esistenza}. For every compact subset $\mathcal{C}\subset \mathcal{B}$, there is a sequence $\{w_k\}\subset\mathcal{A}^+$ such that $w_k$ converges to $u$ uniformly in $\mathcal{C}$ as $k\to+\infty$.
     \end{lemma}
     
     \begin{proof}
         For every $k\in\N$ and for every $x_0\in\mathcal{C}$, by \cref{corollario:usosempre}, we can find $w_{k,x_{0}}\in\mathcal{A}^+$ such that
         $$ w_{k,x_0}(x_0)\le u(x_0)+\frac{1}{k}.$$
         Fixed $k_0\in\N$, by \eqref{e:w_lipschitz_enunciato},
         we have that, if $|x-y|\le \eta:=\frac{1}{ k_0 L}$, then
         $$|w(x)-w(y)|\le L|x-y|\le L\eta= \frac{1}{k_0},$$
         for every $w\in\mathcal{A}^+$. 
Moreover, by \eqref{e:u_lipschitz_enunciato},
we also have that if $|x-y|\le \eta$, then
         $$|u(x)-u(y)|\le L|x-y|\le L\eta= \frac{1}{k_0},$$
         Since $\mathcal{C}$ is a compact set, there are $x_1, \dots, x_m \in \mathcal{C}$ such that
         $$\mathcal{C}\subset \bigcup_{j=1}^m B_\eta(x_j).$$ 
        Let $w_{k_0}:=\min\{w_{{k_0},x_1},\ldots,w_{{k_0},x_m}\}\in\mathcal{A}$, by \cref{lemma:min2soprasol}.
        By \cref{lemma:wlog}, we can assume $w_{k_0}\in\mathcal{A}^+$.
         Thus, for every $x\in\mathcal{C}$, there exists $j= 1,\ldots, m$, such that $|x-x_j|\le \eta$. Then
         $$u(x)\le w_{k_0}(x)\le w_{k_0,x_j}(x)\le w_{k_0,x_j}(x_j)+\frac{1}{k_0}\le u(x_j)+\frac1{2 k_0}\le u(x)+\frac{1}{3k_0}.$$ 
         By the arbitrariness of $k_0 \in \N$, we get that as $k\to+\infty$, $w_k$ converges to $u$ uniformly in $\mathcal{C}$, concluding the proof.
         \end{proof}
         
         We are finally in position to prove the main result of this section.
         
     \begin{proof}[Proof of \cref{lemma:u_viscotity_sol_lipschitz}]
         By \eqref{e:u_lipschitz_enunciato}, the function $u$ is locally Lipschitz. 
         To show that $\Delta u=0$ in $\Omega_u\cap D$, we observe that if $x_0\in\Omega_u\cap D$, then there exists $\rho >0$ such that $u>0$ in $ \overline B_\rho(x_0)\subset\Omega_u\cap D$. By \cref{corollary:ex0}, there exists a sequence $\{w_k\}\subset\mathcal{A}^+$, such that $w_k\to u$ uniformly in $\overline B_\rho(x_0)$. Since $w_k\ge u>0$ in $ B_\rho(x_0)\subset D$, then $u$ is a uniform limit of harmonic function in $B_\rho(x_0)$, implying that $u$ is harmonic in $B_\rho(x_0)$.
     \end{proof}
     
     \begin{remark}
     \label{rem:lipschitz_tutte_viscose}
        We point out that our proof of the locally Lipschitz continuity \eqref{e:lipschitz_u_proposizione_enunciato} holds true for any viscosity solutions of \eqref{e:viscosity-sol-K_D}, not only for the Perron's solution (see the proof of \cref{lemma:w_viscotity_sol_lipschitz}). 
     \end{remark}

\subsection{The function \texorpdfstring{$u$}{u} is a solution on the free boundary}\label{subsec:freeboundarycond}

In order to show that the function $u$ in \eqref{e:funzione_u_teorema_esistenza} is a solution to \eqref{e:viscosity-sol-K_D}, we need to ensure that both
$\abs{\nabla u}^2 = \nabla u \cdot V$ on $\partial \Omega_u \cap D,$ and $\abs{\nabla u}^2 \ge \nabla u \cdot V$ on $\partial \Omega_u \cap K\cap\mathcal{B}$ hold true. We provide the following proposition.

\begin{proposition}
\label{prop:supersol}
    Let $u:\overline{\mathcal{B}}\to\R^+$ be the function defined in \eqref{e:funzione_u_teorema_esistenza}, then $u$ satisfies the conditions 
    \be\label{e:freeboundarycond}
\abs{\nabla u}^2 = \nabla u \cdot V \qquad \hbox{ on } \partial \Omega_u \cap D,
\ee and \be\label{e:freeboundarycondboundary}
\abs{\nabla u}^2 \ge \nabla u \cdot V \qquad \hbox{ on } \partial \Omega_u \cap K\cap\mathcal{B},
\ee in the sense of \cref{def:def-sol_sezione_esistenza}.
\end{proposition}
\begin{proof}
First we prove \eqref{e:freeboundarycond} by showing that both the upper and the lower bound hold true for all $x \in \partial \Omega_u \cap D$. Finally, we provide \eqref{e:freeboundarycondboundary} for all $x \in \partial \Omega_u \cap K\cap\mathcal{B}$.
   We divide the proof in some steps.\\
\\
    {\it Step 1. On $\partial \Omega_u \cap D$, $u$ solves $\abs{\nabla u}^2 \leq\nabla u\cdot V$.} 
    Let us suppose by contradiction that there is a test function $\varphi\in C^\infty(\mathcal{B})$ such that $\varphi$ touches $u$ from below at $x_0\in\partial\Omega_u\cap D$ and
    $$|\nabla \varphi(x_0)|^2>\nabla \varphi(x_0)\cdot V.$$ 
    By continuity, there exists $\rho >0$ such that \be\label{eq:supersol-cond}|\nabla \varphi(x)|^2>\nabla \varphi(x)\cdot V\qquad\text{for every}\quad x\in B_\rho(x_0).\ee 
    Without loss of generality, we suppose that $\varphi$ touches $u$ strictly from below at $x_0$. Indeed, if this is not the case, we consider as test function $\varphi - |x-x_0|^2$. 
    Moreover, by taking as test function 
    $$\varphi- \eps (x-x_0)\cdot e_d + \frac{1}{\eps} ((x-x_0)\cdot e_d)^2,$$ 
    for some $\eps>0$ small enough, we can also assume that 
    \begin{equation}
        \label{e:Delta_phi_0_grande}
        \Delta \varphi(x)>0 \qquad \hbox{ for every }\quad x \in B_\rho(x_0).
    \end{equation}
    
    By \cref{corollary:ex0}, there is a sequence $\{w_k\}\subset\mathcal{A}^+$ such that $w_k \to u$ uniformly on $\overline B_\rho(x_0)$.
    In particular, since $\varphi$ is touching $u$ strictly from below, there exists $\sigma= \sigma(\rho)>0$ such that $w_k>\varphi+\sigma$ on $\partial B_\rho(x_0)$ for all $k\in \N$. 
 Thus, for all $k \in \N$, we can find $c_k\in \R$ such that $\varphi+ c_k$ touches $w_k$ from below in a point $x_k\in \overline B_\rho(x_0)$. Precisely, since $w_k \to u$, there exists $k_0\in \N$ such that for all $k \geq k_0$, we have $c_k\in (-\sfrac{\sigma}{2}, \sfrac{\sigma}{2})$, implying that $x_k \in B_\rho(x_0)$.
 Summarizing, we have found a test function $\varphi$ such that for all $k \geq k_0$, $\varphi$ touches $w_k$  from below at $x_k\in B_\rho(x_0)$, with \eqref{eq:supersol-cond} and \eqref{e:Delta_phi_0_grande}. Since $w_k\in\mathcal{A}^+$, this gets the contradiction.
    \\
    \\ 
    \textit{\it Step 2. On $\partial \Omega_u \cap D$, $u$ solves $\abs{\nabla u}^2 \geq\nabla u\cdot V$.} 
Let us suppose by contradiction that there is a function $\varphi\in C^\infty(\mathcal{B})$ such that $\varphi^+:=\max\{\varphi,0\}$ touches $u$ from above at $x_0\in\partial\Omega_u\cap D$ and 
$$|\nabla \varphi(x_0)|^2<\nabla\varphi(x_0)\cdot V.$$
Similarly as in Step 1, we can suppose that $\varphi^+$ touches $u$ strictly from above at $x_0$.
By continuity, there exists $\rho>0$ such that 
\be\label{eq:subsol-cond}|\nabla \varphi(x)|^2< \nabla \varphi(x)\cdot V\qquad\text{for every}\quad x\in B_\rho(x_0).\ee 
We set 
$$\alpha:=|\nabla \varphi(x_0)|>0\qquad\text{and}\qquad \nu:=\frac{\nabla \varphi(x_0)}{|\nabla \varphi(x_0)|}.$$ 
Moreover, for all $r\in(0,\text{dist}(x_0,\partial D))$, we introduce
$$
u_r(x):=\frac{u(x_0+rx)}{r} \qquad \hbox{ and }\qquad \varphi_r(x):=\frac{\varphi(x_0+rx)}{r}.
$$
Since $\varphi^+$ touches $u$ from above at $x_0$, $\varphi_r^+$ touches $u_r$ from above at $0$ for all $r\in(0,\text{dist}(x_0,\partial D))$.
Since $\varphi$ is a smooth function, for all $x \in B_\rho(x_0)$, we have that 
$$\varphi(x)=\alpha (x-x_0)\cdot \nu+O\left(|x-x_0|^2\right).$$ 
Thus, for every $\sigma>0$ there exists $r_0(\sigma)>0$ such that 
\begin{align}
\label{e:phi_r+_bordo_B1}
\varphi_r^+\equiv0 \, \qquad \hbox{in }\,\{x\cdot\nu<-\sigma\}\cap \partial B_1\qquad \hbox{for every}\; r\in(0,r_0(\sigma)).
\end{align}
Indeed, for every $x\in \{x\cdot\nu<-\sigma\}\cap \partial B_1$, we have
$$\varphi_r(x)=\alpha x\cdot\nu+\frac{O\left(r^2|x|^2\right)}{r}\le -\alpha\sigma+O(r)\le 0$$
providing $r\leq r_0(\sigma)$. In addition, it holds  
\begin{equation}
\label{e:phi_0_definzione_piano}
\norm{\varphi_r^+ - \alpha (x\cdot\nu)^+}_{L^\infty\left(\overline{B}_1\right)}=: \norm{\varphi_r^+ - \varphi_0^+}_{L^\infty\left(\overline{B}_1\right)} \leq C r.
\end{equation}
 Let $\chi$ be a cutoff function such that $\chi\equiv1$ in $B_{\sfrac{1}{4}}$ and $\chi\equiv0$ in $B_{\sfrac{1}{2}}^c$.
For all $\sigma >0$ and for all $r \in (0, r_0(\sigma))$, let $v_{\sigma,r}$ be the function defined as 
 \be\label{eq:competitor-v} \begin{cases}
        \Delta v_{\sigma,r}=0&\text{in }\{x\cdot\nu>-\sigma+2\sigma\chi(x)\}\cap B_1=:\Xi_\sigma, \\
        v_{\sigma,r}=0&\text{in } \{x\cdot\nu\le -\sigma+2\sigma\chi(x)\}\cap B_1,\\
        v_{\sigma,r}(x)=\varphi_r^+(x)&\text{on } \partial B_1,
    \end{cases}\ee 
    which is continuous up to $\partial B_1$ by \eqref{e:phi_r+_bordo_B1}

   For all $\sigma >0$, by classical elliptic Schauder estimates up to the boundary (see \cite{gilbarg1977elliptic}) applied to the harmonic function $v_{\sigma,r}-\varphi_0(x+\sigma\nu)$ in $\Xi_\sigma$, where $\varphi_0$ is defined in \eqref{e:phi_0_definzione_piano}, we have that
    \begin{align*}
        \| v_{\sigma,r}-\varphi_0(\cdot+\sigma\nu)\|_{C^{1,\beta}\left(\Xi_\sigma\right)}&\le C\|v_{\sigma,r}-\varphi_0(\cdot+\sigma\nu)\|_{L^\infty(\partial \Xi_\sigma)}\\
        &= C \norm{\varphi_0(\cdot+\sigma\nu)}_{L^\infty(\partial \Xi_\sigma \cap B_1)} + C\norm{\varphi_r^+-\varphi_0(\cdot+\sigma\nu)}_{L^\infty(\partial \Xi_\sigma\cap \partial B_1)}\\
        &\leq C \norm{\varphi_0(\cdot+\sigma\nu)}_{L^\infty\left(\{\abs{x\cdot \nu}\leq \sigma \}\right)} + C (r+ \sigma)\\
        &\leq C (r+ \sigma),
    \end{align*}
    for some constant $C>0$, having used \eqref{e:phi_0_definzione_piano}.
    By the smallness of the above $C^{1,\beta}$-norm combined with \eqref{eq:subsol-cond}, it follows that 
    \begin{equation}
        \label{e:condizione_v_sta_in_A} |\nabla v_{\sigma,r}(x)|^2<\nabla v_{\sigma,r}(x)\cdot V\quad\text{for every}\quad x\in\partial\Xi_\sigma\cap B_1 = \partial \Omega_v \cap B_1,
    \end{equation}
    for $\sigma$ and $r$ small enough. 
    
    To conclude, we consider the function 
    \be\label{eq:def-w}w(x):=
    \begin{cases}
    \min\left\{ u(x),rv_{\sigma,r}\left(\frac{x-x_0}{r}\right)\right\}&\text{in } \overline B_{r}(x_0),\\
    u(x)&\text{in } \mathcal{B}\setminus  B_{r}(x_0).
    \end{cases}\ee
    Combining \eqref{eq:u-harmonic-in-Omegau} and Step 1, we get that $u\in\mathcal{A}$. Since $v_{\sigma, r}$ satisfies \eqref{eq:competitor-v} and \eqref{e:condizione_v_sta_in_A}, then $w\in \mathcal{A}$, by \cref{lemma:min2soprasol}. Moreover $ w\equiv0$ in $B_{\sigma r}(x_0)$, since $B_\sigma\subset\{x\cdot\nu\le -\sigma+2\sigma\phi(x)\},$ for $\sigma\le\sfrac{1}{4}$.
    Since $x_0\in\partial\Omega_u$, then there is $z\in B_{\sigma r}(x_0)$ such that $u(z)>0= w(z)$, which is a contradiction with the definition of $u$.\\
    \\
  {\it Step 3. On $\partial \Omega_u \cap K\cap\mathcal{B}$, $u$ solves $\abs{\nabla u}^2 \geq\nabla u\cdot V$.} To prove \eqref{e:freeboundarycondboundary}, we notice that we can repeat the same arguments as in Step 2, providing that $x_0 \in\partial\Omega_u\cap K\cap \mathcal{B}$ and concluding again that the function $w$ defined in \eqref{eq:def-w} belongs to $\mathcal{A}$.
\end{proof}

\subsection{The function \texorpdfstring{$u$}{u} is continuous up to the boundary}
\label{subsec:u_continua_fino_al_bordo}

In the following proposition, we prove that the function $u$ defined in \eqref{e:funzione_u_teorema_esistenza} is continuous up to the boundary $\partial \mathcal{B}$.

\begin{proposition}\label{prop:attacca}
Let $\phi:\partial \mathcal{B}\to\R^+$ be an assigned continuous boundary datum such that $\phi = 0$ on $\partial \mathcal{B} \cap \partial K$. Let $u:\overline {\mathcal{B}}\to\R^+$ be the function defined in \eqref{e:funzione_u_teorema_esistenza}. Then $$u\in C^0(\overline {\mathcal{B}})\qquad\text{and}\qquad u=\phi\quad\text{on }\partial\mathcal{B}.$$ 
            \end{proposition}
            
		\begin{proof}
		  By \cref{lemma:u_viscotity_sol_lipschitz}, the function $u$ is locally Lipschitz in $\mathcal{B}$ and, by definition, $u\equiv0$ in $K$.
          Thus to get the thesis, we need to show that
    \be\label{tesiprop}
          \lim_{x\to x_0}u(x)=u(x_0)=\phi(x_0) \qquad \hbox{ for every }\quad x_0\in\partial D \setminus\left(\partial K \cap \mathcal{B}\right). \ee 
            Let us fix $x_0\in\partial D \setminus\left(\partial K \cap \mathcal{B}\right) \subset \partial \mathcal{B}$.
            By the regularity assumptions on $\mathcal{B}$, there exists an external ball at $x_0$ with respect to $\mathcal{B}$. 
            Precisely, there are a radius $\rho>0$ and a point $y_0 \in \mathcal{B}^c$ such that
            $$B_\rho(y_0)\cap\overline{\mathcal{B}}=\{x_0\}.$$ 
            Moreover, since $\mathcal{B}$ is bounded, we can find a radius $R>0$ such that $\overline{\mathcal{B}}\subset B_R(y_0)$.
            We extend the assigned boundary datum as follows 
 $$
    \Phi: \partial\mathcal{B}\cup K \to \R^+ \quad \hbox{ such that } \quad \Phi(x):=\left\{
\begin{aligned}
     \phi(x)
   &&\text{on }\partial \mathcal{B},\\
     0 &&\text{in } K,\\
 \end{aligned}
\right.$$ which is well defined since $\phi =0$ on $\partial \mathcal{B} \cap \partial K$.
           Next, we consider the function 
            $$
    \Psi: \overline B_R(y_0) \to \R^+ \quad \hbox{ such that } \quad \left\{
\begin{aligned}
    \Delta \Psi &= 0 &&\text{in } B_R(y_0) \setminus \overline B_{\rho}(y_0),\\
    \Psi &= 1 
    &&\text{on } \partial B_R(y_0),\\
    \Psi &= 0 &&\text{in } \overline B_{\rho}(y_0).
 \end{aligned}
\right.$$          
            Since the extension $\Phi$ is continuous, for every $\eps>0$ there exists $\delta>0$ such that 
            $$|\Phi(x)-\Phi(x_0)|\le \eps\qquad\text{for every}\quad x\in \left(\partial \mathcal{B}\cup K\right)\cap B_\delta(x_0).$$
            Moreover, since in $\overline{\mathcal{B}}\setminus B_{\delta}(x_0)$ the function $\Psi$ is bounded from below, we can find a constant $M=M(\eps)>0$ large enough such that
    \be\label{stimaphi}
    |\Phi(x)-\Phi(x_0)|\le \eps+M \Psi \quad\text{for every}\quad x\in \partial \mathcal{B}\cup K.\ee 
            By \eqref{stimaphi}, we have 
            \begin{equation}
                \label{e:sigma_1}
                \Phi\le \Sigma_1:= \Phi(x_0)+\eps+M\Psi \quad\text{on }\partial\mathcal{B}\cup K.
            \end{equation}
            Next, we define
            \begin{equation}
                \label{e:sigma2}
                 \Sigma_2: \overline{\mathcal{B}} \to \R^+ \quad \hbox{ such that } \quad \left\{
\begin{aligned}
     \Delta \Sigma_2 &= 0 &&\text{in } D,\\
    \Sigma_2 &= \Phi  
    &&\text{on } \partial \mathcal{B},\\
    \Sigma_2 &= 0 &&\text{in } K.
 \end{aligned}
\right.
            \end{equation} 
Let $\Sigma:= \min\{\Sigma_1, \Sigma_2\}:\overline{\mathcal{B}}\to\R^+$.
  Then, by \eqref{e:sigma_1} and \eqref{e:sigma2}, $\Phi\le \Sigma$ on $\partial\mathcal{B}\cup K$. Moreover, since $\Sigma_2\in\mathcal{A},$ by \cref{lemma:min2soprasol}, $\Sigma\in\mathcal{A}$.
            Therefore, since $u$ is an infimum \eqref{e:funzione_u_teorema_esistenza}, we have 
            \be\label{stimattacca1}
            u\le \Sigma\le \Sigma_1=\Phi(x_0)+\eps+M\Psi\quad\text{in } \overline {\mathcal{B}}.\ee 
            On the other hand, by definition of the function $u$ \eqref{e:funzione_u_teorema_esistenza}, $u\ge\phi$ on $\partial\mathcal{B}$. 
            Moreover, on $K$, $u\equiv0\equiv\Phi$, then $u\ge\Phi$ on $\partial\mathcal{B}\cup K$. 
            Therefore, by \eqref{stimaphi}, we have 
            \be\label{eq:boundutile}
            \Phi(x_0)-\eps-M\Psi\le \Phi(x) \le u(x)\quad\text{for all } x \in \partial \mathcal{B}\cup K.\ee
            We claim that 
            \be\label{unastimafinale}
            \Phi(x_0)-\eps-M\Psi=:\Theta(x)\le u(x)\quad\text{for every}\quad x \in \overline {\mathcal{B}} .
            \ee
            Suppose by contradiction that there exists $y\in \overline{\mathcal{B}}$ such that $\Theta(y)> u(y)$. 
            We define for all $t>0$ the function $\Theta_t:= \Theta -t$. Since $u\geq 0$, for $t$ large enough, $\Theta_t(x) < u(x)$ for all $x \in \overline{\mathcal{B}}$. Thus, there exists $t_0>0$ and $z\in \overline{\mathcal{B}}$ such that $\Theta_{t_0}$ touches $u$ from below at $z$.
            By \eqref{eq:boundutile}, $z\not \in \partial \mathcal{B}\cup K$, then $z\in D$.
            To get the contradiction, we consider the following cases.
            \begin{itemize}
                \item If $z\in D\setminus \overline \Omega_u$, then $u \equiv 0$ in a neighborhood of $z$ and $|\nabla\Theta_{t_0}(z)|=0$. Since $0=|\nabla\Theta_{t_0}(z)| = M\abs{\nabla \Psi(z)}\neq 0$, we get the contradiction. 
                \item If $z\in\partial\Omega_u\cap D$, we have obtained a test function $\Theta_{t_0}$ which is touching $u$ from below at $z$. By \cref{prop:supersol}, we have
                $$
               M \abs{\nabla \Psi(z)}= \abs{\nabla \Theta_{t_0}(z)} \leq \frac{\nabla \Theta_{t_0}(z)}{\abs{\nabla \Theta_{t_0}(z)}}\cdot V\leq \abs{V},
                $$
                which is contradiction as a consequence of the  Hopf Lemma choosing the constant $M$ large enough.
                \item If $z\in \Omega_u\cap D$, then, by the strong maximum principle, $\Theta_{t_0}\equiv u$ in the connected component of $\Omega_u\cap D$ where $z$ belongs to.
                Since $\Theta_{t_0}$ is a radial decreasing function by the definition of $\Psi$, for every $x \in B_{|z|}(y_0)$ we have
                $$\Theta_{t_0}(x)\ge\Theta_{t_0}(z)=u(z)>0,$$
                implying that $u>0$ in $B_{|z|}(y_0)\cap \overline D$. 
                In particular, $\Theta_{t_0}(x_0)=u(x_0)$, getting the contradiction, since 
                $$\Theta_{t_0}(x_0)=\Phi(x_0)-\eps-t_0<\Phi(x_0)\le u(x_0).$$
            \end{itemize}
            This proves the claim \eqref{unastimafinale}.
            Then \eqref{stimattacca1} and \eqref{unastimafinale} give 
            $$ |u(x)- \Phi(x_0)|\le \eps+M\Psi \quad\text{for every}\quad x\in \overline {\mathcal{B}}.$$ 
           Passing to the limit on both sides and using the fact that $\Psi (x_0) = 0$, we have 
           $$\lim_{x\to x_0}|u(x)-\Phi(x_0)|\le \eps,$$
          implying that
          $\lim_{x\to x_0}u(x)=\Phi(x_0)=\phi(x_0)$.
           To conclude, we need to show that $\phi(x_0) = u(x_0)$. By definition $u(x_0) \geq \phi(x_0)$. On the other hand, by \cref{lemma:classical-perron}, we can take the function 
           $$
    v: \overline {\mathcal{B}}\to \R^+ \quad \hbox{ such that } \quad
    \left\{
\begin{aligned}
    \Delta v &= 0 &&\text{in } D,\\
    v &= \phi
    &&\text{on } \partial \mathcal{B},\\
    v&= 0 &&\text{in } K.
 \end{aligned}
 \right.$$
Then, $v\in\mathcal{A}$ and thus $u(x_0)\le v(x_0)=\phi(x_0)$, which yields the thesis.
        \end{proof}	

\section{Interior regularity of the free boundary}\label{sec:interior}
In this section, we study the regularity of viscosity solution of $u: \mathcal{B} \to \R^+$ of the problem \eqref{e:viscosity-sol-K_D} around a point $x_0\in\partial\Omega_u\cap D$, proving \cref{thm:piatto-implica-regolare-intro}.

After a rescaling,
the problem \eqref{e:viscosity-sol-K_D} can be set in $B_1$. Precisely, we consider viscosity solutions of 
\be\label{eq:viscosity-sol-B1}\left\{
    \begin{aligned}
         \Delta u &= 0 && \hbox{in }  \Omega_u\cap B_1,\\
     \abs{\nabla u}^2 &= \nabla u\cdot V&& \hbox{on } \pa \Omega_u\cap B_1.
    \end{aligned}
     \right.\ee
    In the following, we consider the rescalings
$$
u_{x_0, r}:= \frac{u(x_0+ r x)}{r} \qquad \hbox{and} \qquad u_{r}:= \frac{u(r x)}{r},
$$
    and  we prove the following regularity result.
     
\begin{theorem}[Flatness implies regularity]\label{thm:piatto-implica-regolare}
    Let $u:B_1\to\R^+$ be a viscosity solution of \eqref{eq:viscosity-sol-B1}.
    For all $\delta>0$, there is a constant $\overline \e>0$ such that the following holds.
  Let $\nu$ be a unit vector such that  
  $$c:=V\cdot \nu\ge 2\delta$$
  and 
$u$ is $\overline \eps$-flat along the direction $\nu$, namely, for every $x\in B_1$
        \begin{equation}
            \label{e:eps-piatto}
            \left(x \cdot c \, \nu - \overline \eps\right)^+\le u(x)\le  \left(x \cdot  c \, \nu+ \overline\eps\right)^+.
        \end{equation}
Then, $\partial\Omega_u$ is analytic in $B_{\sfrac{1}{2}}$.
\end{theorem}

    The proof is divided into two parts:
    \begin{itemize}
        \item in \cref{susec:piatto_implica_C2alpha}, if $u$ is flat along a direction $\nu$ in $B_1$, with $V\cdot\nu\ge2\delta,$ then both $u$ and $\pa \Omega_u$ are $C^{1,\alpha}$ regular in $B_{\sfrac{1}{2}}$, for some $\alpha \in (0,1)$ through an improvement of flatness lemma;
        \item in \cref{subsec:odograha}, through an hodograph-type transform, we provide the higher order regularity.
    \end{itemize}
    Moreover, we point out that \cref{thm:piatto-implica-regolare-intro} implies \cref{thm:piatto-implica-regolare}, indeed if $x_0\in\partial\Omega_u\cap D$ is a regular point, then for some $r>0$, the rescaling $u_{x_0,r}$ is $\overline \eps$-flat along the direction $\nu$, i.e.~\eqref{e:eps-piatto} holds true.

\subsection{Flatness implies \texorpdfstring{$C^{1, \alpha}$}{C1alpha}}
\label{susec:piatto_implica_C2alpha}
The main theorem of this section is an $\eps$-regularity theorem for the solutions of \eqref{eq:viscosity-sol-B1}. It is stated as follows.

\begin{theorem}[$\varepsilon$-regularity theorem]
\label{th:eps_reg_C2}
Let $u:B_1\to\R^+$ be a viscosity solution of \eqref{eq:viscosity-sol-B1}. For all $\delta>0$, there is a constant $\overline \e>0$ such that the following holds.
  Let $\nu\in\partial B_1$ be a unit vector such that 
  $$V\cdot \nu\ge \delta$$
  and $u$ is $\overline \eps$-flat along the direction $\nu$, namely, for every $x\in B_1$
    $$\left(x \cdot c \, \nu - \overline \eps\right)^+\le u(x)\le  \left(x \cdot  c \, \nu+ \overline\eps\right)^+.
    $$ 
Then, $\partial\Omega_u$ is $C^{1,\alpha}$ regular in $B_{\sfrac{1}{2}}$, for some $\alpha \in (0,1)$
\end{theorem}

The proof of this theorem is a consequence of an iteration of the improvement of flatness lemma shown in \cref{subsec:IOF}.

\subsubsection{Partial Harnack inequality}
\label{subsec:partialharnack}
The first fundamental ingredient to show \cref{th:eps_reg_C2} is the partial Harnack inequality which can be stated as follows.

\begin{lemma}[Partial Harnack inequality]
    \label{lemma:PH_inequality}
  Let $u:B_1\to\R^+$ be a viscosity solution of \eqref{eq:viscosity-sol-B1} and suppose that $0\in\overline \Omega_u$. For all $\delta>0$, there are constants $\overline \e>0$, $\rho>0$ and $\sigma>0$ such that the following holds.
  Let $\nu\in\partial B_1$ be a unit vector such that 
        \be\label{eq:proprieta-partial-Harnack}
    c:=V\cdot\nu\geq \delta.
    \ee
    If there are two real numbers $a\le b$ with $\e:=b-a\leq \overline{\e}$ such that for every $x\in B_1$
\begin{equation*}
\label{eq:hypothesis}
\left(x \cdot c \nu+ a\right)^{+} \leq u(x)\leq  \left(x \cdot c\nu + b\right)^{+},
\end{equation*}    
    then there exist two reals numbers ${a}'\le {b}'$, with $a \leq {a}' < {b}' \leq b$ and
    $$
    |{b}'-{a}'|\leq (1-\sigma)\abs{b-a},
    $$
     such that for every $x \in B_\rho$, it yields
    \begin{align}
    \label{eq:tesi_PH}
     \left(x \cdot c \nu +  a'\right)^{+} \leq u(x)\leq  \left(x \cdot c\nu+  b'\right)^{+}.
    \end{align}
\end{lemma}
\begin{proof} 
We divide the proof into two steps.\\
\\
{\it Step 1. $a \geq 2 \rho$.} 
In this cases the function $u$ is harmonic in $B_1 \cap \{x \cdot c\nu > -2 \rho\}$. To get the thesis is sufficient to apply the standard internal Harnack inequality.\\
\\
{\it Step 2. $\abs{a}\leq 2 \rho$.}
We can choose
$x_0 := \sfrac{\nu}{5}$ and $R := \sfrac{1}{5} + 3 \rho$, where the radius $\rho$ will be determined later. Hence, we consider the following function 
\begin{equation*}
    \psi: \R^d \to \R \quad \hbox{ such that } \quad \left\{
\begin{aligned}
    &\psi(x) = 1
    &&\text{in } \overline B_{\sfrac{1}{50}}(x_0),\\
    &\psi(x) = 0 &&\text{in } \R^d \setminus B_{R}(x_0),\\
    &\psi(x) = C 
    \left(\frac{1}{\abs{x- x_0}^{d}} -\frac{1}{R^{d}}\right)&&\text{in } B_{R}(x_0) \setminus \overline B_{\sfrac{1}{50}}(x_0),
 \end{aligned}
\right.
\end{equation*}
with $C$ a positive constant such that $\psi$ is continuous function.  
Hence, $\psi$ is non a vanishing function inside the ball $B_{R}(x_0)$
and the following holds.
\begin{itemize}
    \item In $B_{R}(x_0) \setminus \overline{B}_{\sfrac{1}{50}}(x_0)$, $\psi$ is subharmonic, i.e.~   \be\label{eq:laplacianw}
    \Delta \psi(x) >0.
    \ee
    \item 
    The function $\psi$ on $B_R(x_0)\setminus \overline B_{\sfrac{1}{50}}(x_0)\cap \{x\cdot c\nu\le 2\rho\}$ satisfies 
\be\label{eq:stima_gradientew}
\nabla\psi \cdot W\ge c_1 ,
\ee
where 
\begin{align}
    \label{e:W_def}
    W:=2(V\cdot \nu)\nu-V
\end{align}
and $c_1>0$. Indeed, in $B_R(x_0)\setminus \overline B_{\sfrac{1}{50}}(x_0)\cap \{x\cdot c\nu\le 2\rho\}$, by \eqref{e:W_def}, \eqref{eq:proprieta-partial-Harnack} and since $x_0 = \sfrac{\nu}{5}$, we get 
    \begin{align*}
        \nabla \psi\cdot W&=Cd\abs{x-x_0}^{-d-2}(x_0-x)\cdot \left(2(V\cdot \nu)\nu-V\right)\\
        &=Cd\abs{x-x_0}^{-d-2}\left(\frac15V\cdot \nu-x\cdot \left(2(V\cdot \nu)\nu-V\right)\right)
        \\
        &\ge Cd\abs{x-x_0}^{-d-2}\left(\frac\delta 5-C_\rho\right)\\
        &\geq c_1,
    \end{align*}
    where $c_1$ is a positive constant obtained by the fact that $C_\rho\to0^+$ as $\rho\to 0^+$. Indeed, since $x$ belongs to a ball with radius comparable to $\sqrt{\rho}$, it is sufficient to choose $\rho>0$ such that $C_\rho\le \sfrac{\delta}{10}$.
\end{itemize}
We set $P(x) := x \cdot c\nu+a$ and let us suppose that $u(x_0) \geq P(x_0) + \frac{\eps}{2}$.  
In the ball $B_{\sfrac{1}{25}}(x_0)$, the function $u-P$ is harmonic. Thus, by applying the Harnack inequality, we get
\begin{equation}
    \label{eq:Harnack}
    u(x) - P(x) \geq c_{\mathcal{H}} \eps \qquad \hbox{in } B_{\sfrac{1}{50}}(x_0),
\end{equation}
where $c_{\mathcal{H}}>0$ is the Harnack dimensional constant. Taking the family of competitors $\{v_t\}_{t \in [0,1]}$ given by
\begin{equation*}
    \label{eq:v_t}
    v_t(x) = P(x) + c_{\mathcal{H}} \eps \psi(x) -c_{\mathcal{H}} \eps + c_{\mathcal{H}} \eps t,
\end{equation*}
we want to show that $u(x) \geq v_t(x)$ for all $t \in [0,1]$ and for all $x \in B_1$. 
Using the definition of $v_t$ and \eqref{eq:Harnack}, the inequality $u(x) \geq v_t(x)$ holds true for all $t \in [0,1]$, for all $x\in B_1 \setminus B_{R}(x_0)$ and for all $x\in \overline{B}_{\sfrac{1}{50}}(x_0)$. 
Then, by contradiction, let us suppose that there exists $t \in [0,1]$ such that $v_t$ touches $u$ from below in $\overline{x}\in B_R(x_0) \setminus \overline{B}_{\sfrac{1}{50}}(x_0)$. Then, necessarily $\overline{x} \in \pa \Omega_u$ since by \eqref{eq:laplacianw}
\begin{equation*}
    \label{eq:Deltavt}
    \Delta v_t = \Delta P(x) + c_{\mathcal{H}} \eps \Delta \psi(x)>0.
\end{equation*}
Concerning what happens on $\pa \Omega_u$, we first observe that
$$
\nabla v_t=\nabla P(x) + c_{\mathcal{H}} \eps \nabla \psi(x) = c\nu+c_{\mathcal{H}}\eps \nabla \psi.
$$
Thus, we get
\begin{align*}
\nonumber
    \abs{\nabla v_t}^2  - \nabla v_t \cdot V &= \abs{c \nu + c_{\mathcal{H}} \eps \nabla \psi}^2 -   c\nu\cdot V - \eps c_{\mathcal{H}} \nabla \psi \cdot V\\
   \nonumber
    &= c_{\mathcal{H}} \eps \nabla \psi \cdot \left(2 c \nu - V \right)+ o(\eps)\\
     \label{eq:boundary_cond_vt}
    &= c_{\mathcal{H}} \eps \nabla \psi \cdot W+ o(\eps).
\end{align*}
Since $\partial\Omega_u\subset B_R(x_0)\setminus \overline B_{\sfrac{1}{50}}(x_0)\cap \{x\cdot c\nu\le 2\rho\}$, by \eqref{eq:stima_gradientew}, for $\eps$ small enough we get
\begin{align*}
    \abs{\nabla v_t}^2  - \nabla v_t \cdot V >0.
\end{align*}
This is a contradiction, namely there is no $\overline{x}$ touching point from below. Then, for all $x \in B_1$ and taking $t = 1$, we get
$$
u(x) \geq v_1(x) \qquad \Longrightarrow\qquad u(x) \geq P(x) + c_{\mathcal{H}} \eps \psi(x).
$$
Since $\psi$ is strictly positive in $B_{\rho}$, we get
$$
u(x) \geq P(x) + c_d \eps,
$$
which is exactly the left inequality taking ${a}' := a + c_d \eps$ in \eqref{eq:tesi_PH} and $c_d$ is a dimensional constant. In the case $u(x_0)\le P(x_0)+\frac\eps 2$, applying a similar argument as the one above, we obtain the upper bound in \eqref{eq:tesi_PH}.
\end{proof}

In order to show the improvement of flatness lemma, a compactness result is required.

\begin{lemma}
    \label{lemma:compattezza}
 For all $j \in \N$, let $u_j:B_1\to\R^+$ be a viscosity solutions of \eqref{eq:viscosity-sol-B1}. For all $\delta>0$, there exists a constant $\overline \e>0$ such that the following holds.
 Let $\nu_j\in\partial B_1$ be a sequence of unit vectors and $\eps_j\in(0,\overline\eps]$ such that 
    \bea
    \lim_{j\to+\infty}\eps_j=0,\qquad c_j:=V\cdot\nu_j\geq \delta,\eea
and $u_j$ are $\eps_j$-flat along the directions $\nu_j$, namely for every $x\in B_1$
    $$\left(x \cdot c_j \, \nu - \eps_j\right)^+\le u_j(x)\le  \left(x \cdot  c_j \, \nu+ \eps_j\right)^+.
    $$
   Then, there exists a H\"older continuous function $\widetilde{u}: B_{\sfrac{1}{2}} \cap \{x\cdot \nu \geq 0\} \to \R$ such that
    \begin{align}
    \label{e:utilde_j}
        \widetilde{u}_j : \left\{
    \begin{aligned}
         & \overline{\Omega}_{u_j}\cap B_1\to \R\\
         &x\mapsto  \widetilde{u}_j (x):= \frac{u_j(x) -x\cdot c_j \nu_j}{\eps_j},
    \end{aligned}
    \right.
    \end{align}
    and for all $\eta>0$
    $$
    \widetilde{u}_j\to 
    \widetilde{u}\qquad\text{uniformly in}\quad B_{\sfrac{1}{2}}\cap \{x \cdot \nu \geq \eta\},
    $$
    where $\nu_j \to \nu \in \pa B_1$.
    Moreover, the sequence of graphs of $\widetilde u_j $ converges in the Hausdorff distance to $\widetilde u$, namely
    $$
    \mathcal{G}_j = \left\{(x, \widetilde{u}_j(x)): x \in B_{\sfrac{1}{2}} \cap \overline{\Omega}_{u_j} \right\} \qquad \xrightarrow{H} \qquad \mathcal{G}= \left\{(x, \widetilde{u}(x)): x \in B_{\sfrac{1}{2}} \cap \{x \cdot \nu \geq 0\} \right\}.
    $$
\end{lemma}

\begin{proof}
 We divide the proof in two steps.\\
    \\
{\em Step 1. H\"older type estimate.} Let $\rho>0$, $\overline{\eps}>0$ and $\sigma>0$ be the constants in \cref{lemma:PH_inequality}. We fix $j\in\N$.
We claim that, for every $x_0\in B_{\sfrac{1}{2}}\cap\overline\Omega_{u_j}$
\begin{equation}
    \label{eq:stima_holder}
    \abs{\widetilde{u}_j(x) - \widetilde{u}_j(x_0)}\le C|x-x_0|^\gamma\qquad\text{for every}\qquad x\in B_{\sfrac{1}{2}}(x_0)\setminus B_{\sfrac{\eps_j}{\overline \eps}}(x_0),
\end{equation}
for some constant $C>0$ and $\gamma\in(0,1)$.
Let $m\geq 0$ such that
$$
\frac{1}{2}\rho^{m+1} < \frac{\eps_j}{\overline{\eps}} \leq \frac{1}{2}\rho^{m}.
$$
Then, for all $x_0 \in B_{\sfrac{1}{2}} \cap \overline\Omega_{u_j}$, we can iteratively apply \cref{lemma:PH_inequality}, getting that for all $k = 0,1,\dots, m$, there are
$$
-\eps_j = a_0 \leq a_1 \leq  \dots \leq a_k \leq \dots \leq a_m < b_m \leq \dots \leq b_k \leq \dots \leq b_1 \leq b_0 = \eps_j
$$
such that 
$$
\abs{b_k - a_k} \leq (1- \sigma)^{k} \abs{b_0 - a_0},
$$
and for every $x\in B_{\rho^k}(x_0)\cap\overline \Omega_{u_j}$
$$
\left(x \cdot c_j \nu_j+ a_k\right)^{+} \leq u_j(x)\leq  \left(x \cdot c_j\nu_j +  b_k\right)^{+}.
$$
Then, for every $x\in B_{\rho^k}(x_0)\cap\overline \Omega_{u_j}$, we have $$|u_j(x)-x\cdot c_j\nu_j-a_k|\le \abs{b_k-a_k}\le(1-\sigma)^k|b_0-a_0|=2(1-\sigma)^k\eps_j.$$
Defining $\widetilde{u}_j$ as in \eqref{e:utilde_j}, we get that, for every $k=0,1,\dots,m$ and for every $x\in B_{\rho^k}(x_0)\cap\overline \Omega_{u_j}$
$$
\abs{\widetilde{u}_j(x) - \widetilde{u}_j(x_0)} \leq 4 (1- \sigma)^k,
$$
which by interpolation ensures the H\"older estimate for $\widetilde{u}_j$ concluding the first step of the proof.\\
    \\
{\em Step 2. Convergence of $\widetilde u_j$.} By the H\"older estimate \eqref{eq:stima_holder} and Ascoli-Arzel\`{a} theorem, we obtain the uniform convergence of $\widetilde u_j$ to $\widetilde u$ in any compact subsets of $B_{\sfrac12}\cap \{x\cdot \nu\ge 0\}$. For what concerns the convergence of the graphs in the Hausdorff distance, the proof is exactly the same as in \cite[Lemma 7.14, (ii)]{velichkov2023regularity} and it is a consequence of the uniform convergence obtained in Step 1.
\end{proof}

\subsubsection{Improvement of flatness}\label{subsec:IOF}

Before proving the improvement of flatness lemma, we need to show a regularity result for a linearized problem stated in the following lemma.

\begin{lemma}[Regularity for the linearized problem]
    \label{lemma:regolaita_prob_lin} 
    For all $\delta >0$, there exists a constant $C>0$ 
    such that the following holds. Let $\nu \in \pa B_1$ be a unit vector, with $$c:= V\cdot \nu\ge \delta$$ 
    and let us define the vector $W$ as 
    $$
W:=2(\nu \cdot V) \nu - V.
    $$
  Consider a viscosity solution $w$ of the following linear elliptic problem
    \begin{equation}\label{linearizzato_interno}
        \left\{
\begin{aligned}
    &\Delta w = 0 && \hbox{ in } B_{\sfrac{1}{2}} \cap \{x \cdot \nu > 0\},\\
    & \nabla w \cdot W =0 && \hbox{ on } B_{\sfrac{1}{2}}\cap \{x \cdot \nu = 0\},\\
\end{aligned}
        \right.
    \end{equation}
    with
    $$
    \norm{w}_{L^\infty(B_{1})} \leq 1\qquad\text{and}\qquad w(0)=0.
    $$
     Then, for every $x\in B_{\sfrac12}\cap \{x\cdot\nu\geq 0\}$
     \begin{equation}
         \label{eq:stima_prob_linearizzato_interno}
         \abs{{w}(x) - \nabla {w}(0) \cdot x} \leq  C |x|^2.
     \end{equation}
\end{lemma}

\begin{proof}
The PDE elliptic problem \eqref{linearizzato_interno} has an oblique boundary condition, since $W\cdot\nu=V\cdot\nu\ge\delta$.
Then, for some $\alpha \in (0,1)$, the $C^{2,\alpha}$ estimate up to the boundary $\{x\cdot\nu=0\}$ follows by classical results on regularity theory of PDE. We refer to \cite[Theorem 4.5]{lieberman2013oblique} for the $C^{1,\alpha}$ estimate, for some $\alpha\in(0,1)$; then, to get \eqref{eq:stima_prob_linearizzato_interno}, we apply \cite[Lemma 6.27]{gilbarg1977elliptic} and we conclude.
\end{proof}

Finally, we are in position to prove the main result of this section: the improvement of flatness lemma.

\begin{lemma}[Improvement of flatness]
\label{lemma:IOF}
Let $u:B_1\to\R^+$ be a viscosity solution of \eqref{eq:viscosity-sol-B1} and suppose that $0\in\partial \Omega_u$.  
For all $\delta>0$ and $\gamma\in(0,1)$, there are constants $\overline \e>0$, $C>0$ and $\eta\in(0,1)$ such that the following holds.
  Let $\nu\in\partial B_1$ be a unit vector such that 
    \be
        \label{eq:proprieta-improvement}
    c:=V\cdot\nu\geq \delta.\ee
Suppose that, for some $\eps\in(0,\overline\eps]$, $u$ is $\eps$-flat along the direction $\nu$, namely for every $x\in B_1$
    $$\left(x \cdot c \, \nu - \eps\right)^+\le u(x)\le  \left(x \cdot  c \, \nu+ \eps\right)^+.
    $$
    Then, for some unit vector $\nu'\in\partial B_1$, $u_\eta$ is $\eps \eta^{\gamma}$-flat along $\nu'$ in $B_{1}$, namely for every $x\in B_1$
    \begin{equation*}
        \label{e:flatness_scala_eta}
        \left(x \cdot c' \, \nu' - \eps\eta^{\gamma}\right)^+\le u_\eta(x)\le  \left(x \cdot c' \, \nu'+ \eps\eta^{\gamma}\right)^+,
    \end{equation*}
    where \be \label{e:tesi1}c' := V \cdot \nu'\qquad\hbox{ and }\qquad \abs{ \nu -  \nu'}\leq C \eps.\ee
    
\end{lemma}

\begin{proof}
    We show such a result by compactness. For all $j \in \N$, let $u_j$ be a sequence of solutions of \eqref{eq:viscosity-sol-B1} which are $\eps_j$-flat along the directions $\nu_j\in\pa B_1$ in $B_1$, namely
    $$\left(x \cdot c_j \, \nu_j - \eps_j\right)^+\le u_j(x)\le  \left(x \cdot  c_j \, \nu_j+ \eps_j\right)^+$$
    where \be\label{eq:proprieta}\lim_{j\to+\infty}\eps_j=0,\qquad c_j:=V\cdot\nu_j\geq \delta.\ee
    We consider the following functions
    \begin{equation}
        \label{eq:u_j}
      \widetilde{u}_j:
      \left\{
      \begin{aligned}
          \overline\Omega_{u_j}&\cap B_1\to\R\\
          x&\mapsto\widetilde{u}_j(x) = \frac{u_j(x) - c_j \nu_j \cdot x }{\eps_j}.
      \end{aligned}
      \right. 
    \end{equation}
    By compactness \cref{lemma:compattezza}, $\widetilde{u}_j\to \widetilde{u}$ uniformly on every compact subset of $B_{\sfrac{1}{2}} \cap \{x \cdot \nu > 0\}$, where the limit function 
    $$
    \widetilde{u}:B_{\sfrac{1}{2}} \cap \{x \cdot \nu \ge 0\}\to\R
    $$ 
    is H\"older continuous on $B_{\sfrac{1}{2}} \cap \{x \cdot \nu > 0\}$ and solves 
    $$
    \Delta \widetilde{u}= 0 \qquad \hbox{in } B_{\sfrac{1}{2}}\cap \{x \cdot \nu > 0\}.
    $$
    To get the boundary condition satisfied by $\widetilde u$, let $\varphi\in C^\infty(B_1\cap\{x\cdot\nu\ge0\})$ be a test function touching $\widetilde u$ from below at $x_0\in B_{\sfrac{1}{2}}\cap\{x\cdot\nu=0\}$.
    By \cref{lemma:compattezza}, as in Step 1 of \cref{prop:supersol}, we can suppose that there exists a sequence $x_j\in\overline\Omega_{u_j}\cap B_1$ and $\sigma_j\in\R$ such that $\varphi+\sigma_j$ touches $\widetilde u_j$ from below at $x_j$ and $\Delta \varphi>0$ in a neighborhood of $x_0$. By definition of $\widetilde u_j$, we have that, $x\cdot c_j\nu_j+\varphi\eps_j+\eps_j\sigma_j$ touches $u_j$ from below at $x_j$. Moreover, since $\Delta \varphi>0$ in a neighborhood of $x_0$, then $x_j\in\partial\Omega_{u_j}\cap B_1$.
    Then, 
    \begin{align}\label{eq:bordo}
        \abs{\nabla \left(\eps_j \varphi(x_j) + x\cdot c_j \nu_j  \right)}^2 &\le \eps_j \nabla \varphi(x_j) \cdot V + c_j \nu_j \cdot V .
        \end{align}
        The left hand side simplifies into 
        \begin{align*}
            \eps_j^2 \abs{\nabla \varphi(x_j)}^2 + c_j^2 \abs{\nu_j}^2  + 2 \eps_j   c_j\nabla \varphi(x_j) \cdot \nu_j,
        \end{align*}
        where using \eqref{eq:proprieta}, \eqref{eq:bordo} reads as 
        \begin{align*}
            \eps_j^2 \abs{\nabla \varphi(x_j)}^2   + 2 \eps_j   c_j\nabla(x_j) \varphi \cdot \nu_j \le \eps_j \nabla \varphi(x_j) \cdot V 
        \end{align*}
   Taking the leading order $\eps_j$, we obtain
   \begin{align*}
        &2 \eps_j   c_j\nabla \varphi(x_j) \cdot \nu_j - \eps_j \nabla \varphi(x_j) \cdot V + o(\eps_j)\le 0,
        \end{align*}
        which can be rewritten 
        as
        \begin{align*}
              \eps_jW_j \cdot \nabla \varphi(x_j) + o(\eps_j)\le 0,
        \end{align*}
        where $$W_j := 2 c_j\nu_j - V=2 (\nu_j \cdot V)\nu_j - V.$$
    Thus, we can pass to the limit: using the fact that $\nu_j \in \partial B_1$, there exists a direction $\nu$ with $\abs{\nu} = 1$ such that $\nu_j \to \nu$, then we have
    $$\nabla\varphi(x_0)\cdot W\le0,$$
    where $$W:=2(\nu \cdot V)\nu - V.$$ 
Analogously, if $\varphi\in C^\infty(B_1)$ is a test function touching $\widetilde u$ from above at $x_0\in B_{\sfrac{1}{2}}\cap\{x\cdot\nu=0\}$, we can repeat the same argument as above, getting that $$\nabla \varphi(x_0)\cdot W\ge0.$$
This implies that $\widetilde{u}$ solves the following oblique elliptic PDE problem
    \begin{equation*}
        \label{eq:prob_linearizzato}
        \left\{
\begin{aligned}
    &\Delta \widetilde{u} = 0 && \hbox{ in } B_{\sfrac{1}{2}} \cap \{x \cdot \nu \geq 0\}\\
    & \nabla \widetilde{u} \cdot W =0 && \hbox{ on } B_{\sfrac{1}{2}}\cap \{x \cdot \nu = 0\},\\
\end{aligned}
        \right.
    \end{equation*}
    with $$\norm{\widetilde{u}}_{L^\infty(B_{\sfrac{1}{2}})} \leq 1\qquad\text{and}\qquad \widetilde u(0)=0.$$
    Thanks to \cref{lemma:regolaita_prob_lin}, we can write the Taylor expansion of the solution, namely for some positive small constant $\eta$ and for all $ x \in B_{\eta}\cap \{x \cdot \nu \geq 0\}$ we have
    $$
    \begin{aligned}
        \abs{\widetilde{u}(x) - \nabla \widetilde{u}(0) \cdot x } &\leq  C \abs{x}^{2}\leq C\eta^{2}\leq \frac{\eta^{1+\gamma}}{4}.
    \end{aligned}
    $$
    Thus, along the sequences, for every $x\in B_\eta\cap\overline\Omega_{u_j}$
     $$
    \begin{aligned}
        \abs{\widetilde{u}_j(x)- \nabla \widetilde{u}(0) \cdot x } &\leq  \frac{\eta^{1+ \gamma}}{2},
    \end{aligned}
    $$
    which can be rewritten substituting the expression of $\widetilde{u}_j$ defined in \eqref{eq:u_j}, thus for all $x \in B_{\eta} \cap \overline\Omega_{u_j}$
    \begin{equation}
        \label{eq:finale1}
        \abs{u_j (x) - c_j \nu_j \cdot x - \eps_j \nabla \widetilde{u}(0) \cdot x} \leq \eps_j \frac{\eta^{1+\gamma}}{2}.
    \end{equation}
    We call
    \begin{equation}
        \label{e:nuova_piattezza}
        \nu'_j := \frac{c_j \nu_j + \eps_j \nabla \widetilde{u}(0) }{\abs{c_j \nu_j+ \eps_j \nabla \widetilde{u}(0)}}.
    \end{equation}
   Then, by \eqref{e:nuova_piattezza}, we observe that
        \begin{equation}
        \label{eq:nu_j'-nu_j}
            \begin{aligned}
            \nu_j'-\nu_j&=\frac{c_j\nu_j+\eps_j\nabla\widetilde u(0)-\nu_j\abs{c_j\nu_j+\eps_j\nabla\widetilde u(0)}}{\abs{c_j\nu_j+\eps_j\nabla\widetilde u(0)}}\\
            &=\eps_j\frac{\nabla \widetilde u(0)-(\nabla \widetilde u(0)\cdot\nu_j)\nu_j}{\abs{c_j\nu_j+\eps_j\nabla\widetilde u(0)}}+o(\eps_j)\\
            &=\frac{\eps_j}{c_j}\left(\nabla \widetilde u(0)-(\nabla \widetilde u(0)\cdot\nu_j)\nu_j\right)+o(\eps_j).
        \end{aligned}
        \end{equation}
   Hence, by \eqref{eq:proprieta} and \cref{lemma:regolaita_prob_lin}, we have that
   \begin{equation*}
   \label{e:cond_normali}
       |\nu_j'-\nu_j|\le C\eps_j,
   \end{equation*} 
   which is exactly the first estimate in \eqref{e:tesi1}.
    Then, by \eqref{eq:proprieta} and \eqref{eq:nu_j'-nu_j}, we obtain
    \begin{equation}
        \label{e:controllo_oscillazioni}
\begin{aligned}
    \abs{c_j'\nu_j' - c_j \nu_j - \eps_j \nabla \widetilde{u}(0)} &= \abs{c_j \nu_j + \eps_j \nabla \widetilde{u}(0)}\abs{\frac{c_j'}{\abs{c_j \nu_j + \eps_j \nabla \widetilde{u}(0)}}-1}\\
    &= \abs{c'_j - \abs{c_j \nu_j +\eps_j \nabla \widetilde{u}(0)}}\\
    & =\abs{V \cdot \nu'_j - V \cdot \nu_j - \eps_j \nabla \widetilde{u}(0) \cdot \nu_j} + o (\eps_j)\\
    &= \frac{1}{c_j}\abs{\eps_j V\cdot \nabla \widetilde{u}(0) - 2\eps_j c_j \nabla \widetilde{u}(0) \cdot \nu_j}+ o (\eps_j)\\
    &=  \frac{1}{c_j}\abs{\eps_j \nabla \widetilde{u}(0) \cdot W_j}+ o (\eps_j)\\
    &= o (\eps_j).
\end{aligned}
 \end{equation}
Then, using \eqref{eq:finale1} and \eqref{e:controllo_oscillazioni}, we have for all
$x \in B_{\eta}\cap \overline\Omega_{u_j}$
    \begin{equation*}
        \abs{u_j (x) -x\cdot c_j' \nu_j' } \leq \eps_j \frac{\eta^{1+\gamma}}{2}+o(\eps_j)\le\eps_j \eta^{1+\gamma},
    \end{equation*}
which concludes the proof.
\end{proof}

\subsubsection{Proof of \texorpdfstring{$\eps$}{eps}-regularity theorem}\label{eps-sub}

The following lemma gives the thesis of \cref{th:eps_reg_C2}.

\begin{lemma}[Uniqueness and rate of convergence to the blow-up limit]
    \label{lemma:proprieta_blowup}
   Let $u:B_1\to\R^+$ be a viscosity solution of \eqref{eq:viscosity-sol-B1}.
   For all $\delta>0$ and $\gamma\in(0,1)$ there are constants $\eps_0>0$ and $C>0$ such that the following holds.
  Let $\nu\in\pa B_1$ be a unit vector 
  such that $$c:=V\cdot\nu\geq 2\delta,$$
and $u$ is $\eps_0$-flat along $\nu$-direction, namely for every $x\in B_1$
    $$\left(x \cdot c \, \nu - \eps_0\right)^+\le u(x)\le  \left(x \cdot  c \, \nu+ \eps_0\right)^+.
    $$
    Then, for every $x_0\in \partial\Omega_u\cap B_{\sfrac{1}{2}}$, there is a unit vector $\nu_{x_0}\in\pa B_1$ such that, for every $r\in (0,\sfrac{1}{2})$,
    \begin{equation}
        \label{e:tesi_iterazione}
        \norm{u(x_0+x)-x\cdot c_{x_0}\nu_{x_0}}_{L^\infty(B_r\cap\overline\Omega_u)}\le C\eps_0 r^{1+\gamma},
    \end{equation}
where $c_{x_0} := V \cdot \nu_{x_0}\ge\delta$.
\end{lemma}
\begin{proof}
Let $\overline{\eps}$ be as in \cref{lemma:IOF} and let $\eps_0\le\frac{\overline \eps}{4}$ to be chosen later.
Since $u$ is $\eps_0$-flat along $\nu$ in $B_1$, then for all $x_0 \in \pa \Omega_u \cap B_{\sfrac{1}{2}}$, $u_{x_0,\sfrac{1}{2}}$ is $ \eps$-flat along $\nu$ in $B_1$, where $\eps:=4\eps_0$, namely for all $x \in B_1$
    \begin{equation*}
        \label{e:casoj=0}
        \left(x \cdot c \nu -  {\eps} \right)_{+} \leq u_{x_0,\sfrac{1}{2}}(x)\leq \left(x \cdot c \nu +  {\eps} \right)_{+}.
    \end{equation*}
    From now on, we fix $x_0 \in \pa \Omega_u \cap B_{\sfrac{1}{2}}$.
  We set $\nu_0:=\nu$ and $c_0:=c$.
  Let $j \in\N$.  We assume that $u_{x_0,\sfrac{\eta^j}{2}}$ is $ \eps \eta^{j\gamma}$-flat along the direction $\nu_j$, namely, for every $x\in B_1$
   $$
   \left(x \cdot c_j \nu_j  -  {\eps} \eta^{j\gamma} \right)^{+} \leq u_{x_0,\sfrac{\eta^j}{2}}(x)\leq \left(x \cdot c_j \nu_j +  {\eps} \eta^{j\gamma} \right)^{+}
    $$ 
    and $c_j:=V\cdot\nu_j\ge \delta$, like in \eqref{eq:proprieta-improvement}.
    Then, by the improvement of flatness \cref{lemma:IOF}, there exist $\nu_{j+1} \in \pa B_1$ such that $u_{x_0,\sfrac{\eta^{j+1}}{2}}$ is $ \eps \eta^{(j+1) \gamma}$-flat along the direction $\nu_{j+1}$, i.e.~for all $x \in B_1$
    \begin{equation}
        \label{e:improvement_scala_1}
        \left(x \cdot c_{j+1} \nu_{j+1}  -  {\eps} \eta^{(j+1) \gamma}\right)^{+} \leq u_{x_0,\sfrac{\eta^{j+1}}{2} }(x)\leq \left(x \cdot c_{j+1} \nu_{j+1}  + {\eps} \eta^{(j+1) \gamma} \right)^{+},
    \end{equation}
    where $c_{j+1} := V\cdot \nu_{j+1} $. Moreover, 
    \begin{equation}
        \label{e:stima_nu_p}
         \abs{\nu_{j+1}-\nu_{j}}\le C  \eps\eta^{j\gamma}.
    \end{equation} 
    In particular, if $c_j\ge \delta$ for every $j\in\N$, then \eqref{e:improvement_scala_1} holds for every $j\in\N$, namely
    \be\label{eq:stima1}
    \norm{u_{x_0,\sfrac{\eta^{j}}{2} }(x) - x \cdot c_{j} \nu_{j}}_{L^\infty(B_1 \cap \overline\Omega_u)} \leq  {\eps} \eta^{j \gamma}.
    \ee
    Thus, as $j\to +\infty$, by \eqref{e:stima_nu_p}, there exists $\nu_{x_0} \in \pa B_1$ such that $\nu_{j}\to \nu_{x_0}$.
    Then \eqref{eq:stima1} can be rewritten as
    $$
    \norm{u_{x_0,\sfrac{\eta^{j}}{2}}(x) - x \cdot c_{x_0} \nu_{x_0}}_{L^\infty(B_1 \cap \overline\Omega_u)} \leq C {\eps} \eta^{j\gamma}.
    $$ 
    where $c_{x_0}:=V\cdot \nu_{x_0}$. By interpolation, we obtain
    $$
    \norm{u_{x_0,\sfrac{r}{2} }(x) - x \cdot c_{x_0} \nu_{x_0} }_{L^\infty(B_1 \cap \Omega_u)} \leq C {\eps} r^{\gamma}.
    $$
    for every $r\in(0,\sfrac{1}{2})$, which is \eqref{e:tesi_iterazione}.

To conclude the proof, what it is left to show is that $c_j\ge\delta$ for every $j\in\N$. 
Without loss of generality, we can choose the constant $\eta$ in \cref{lemma:IOF} and $\eps_0$ small enough such that $$
    \eta^{\gamma}\le \frac12\qquad\text{and}\qquad C\eps \abs{V}\le \frac\delta 2.
$$
Then, by \eqref{e:stima_nu_p}, it holds
\begin{equation}
    \label{e:controllo_eta}
    |\nu_{j+1}-\nu_{j}|\abs{V}\le C |V|\eps\eta^{j\gamma}\le \frac{\delta}{2^{j+1}}.
    \end{equation}
Finally, by \eqref{e:controllo_eta}$$
\begin{aligned}
    c_{j}:=V\cdot \nu_j&= V\cdot \nu_{0}+\sum_{k=0}^{j-1} V\cdot (\nu_{k+1}-\nu_{k})\\
    &\ge 2\delta-\sum_{k=0}^j\abs{V}\abs{\nu_{k+1}-\nu_k}\\
    &\ge \delta,
\end{aligned}
$$
concluding the proof.
\end{proof}

\subsection{\texorpdfstring{$C^{1, \alpha}$}{C1alpha} implies 
\texorpdfstring{$C^{\infty}$}{Cinfty}}
\label{subsec:odograha}
In this section, we provide $C^{\infty}$ estimates and analiticity for the free boundary $\pa \Omega_u$.

\begin{theorem}
\label{thm:odografa_Cinfty}
    Let $\delta>0$ and $u$ as in \cref{th:eps_reg_C2}. Suppose that 
    $x_0\in\partial\Omega_u \cap B_{\sfrac{1}{2}}$ and 
    $\nabla u(x_0)=c\nu$, for some unit vector $\nu\in\pa B_1$ such that $c:=V\cdot \nu\ge \delta$.
    Then, there exists $\rho = \rho(\delta, x_0)>0$ such that $\pa \Omega_u$ is $C^\infty$ regular and analytic in $B_\rho(x_0)$.
\end{theorem}
\begin{proof}
For simplicity, we can suppose that $\nu=e_d$ and $x_0 = 0$. 
Let $\rho>0$ and we introduce the map
$$
\Phi:
\left\{
\begin{aligned}
    \overline \Omega_u\cap B_{\rho} &\to \R^d\cap \{y_d\ge0\},\\
     (x', x_d) &\mapsto \Phi(x', x_d) := (x', u(x)),
\end{aligned}
\right.
$$
where $x' = (x_1, \dots, x_{d-1})\in \R^{d-1}$. Since, by hypothesis, $\abs{\nabla u (0)}= \abs{V\cdot e_d} = \abs{ V_d}\ge\delta$,
by choosing a radius $\rho = \rho(\delta)>0$ small enough, $\Phi$ is invertible. Thus, there exists 
 $$
\Phi^{-1}:
\left\{
\begin{aligned}
    \mathcal{D}:=\Phi\left(\overline \Omega_u\cap B_{\rho}\right) &\to \overline \Omega_u\cap B_{\rho},\\
    (y', y_d) &\mapsto \Phi^{-1}(y', y_d) := (y', w(y)).
\end{aligned}
\right. 
$$ 
We notice that $w$ inherits the same regularity as $u$, thus $w \in C^{1, \alpha}(\mathcal{D})$ by \cref{th:eps_reg_C2}. Concerning points on the free boundary, we have
$$
x \in \pa \Omega_u \cap B_\rho\qquad \iff \qquad \Phi(x', x_d) = (x',0)  \qquad \iff \qquad \Phi^{-1}(y',0) = (y', w(y',0)).
$$
Thus, studying the regularity on the free boundary is equivalent to perform a regularity result on the function $w$. Since $x_d = w(x', u(x))$, computing derivatives with respect both $d$ and $i < d$, we get for all $x \in \Omega_u \cap B_\rho$
\begin{align}
   \label{e:der_d}
    \partial_{d} w \, \partial_{d} u =1,\\
    \label{e:der_i}
    \partial_{i} w +\partial_{d}\,w\partial_{i}u=0,
\end{align}
where the two functions are computed in $w=w(x', u(x))$ and $u = u(x)$, respectively.
Performing again derivatives of \eqref{e:der_d} with respect to $d$ and $i < d$ and \eqref{e:der_i} with respect to $j<d$, we obtain
\begin{align*}
    \partial_{dd}w\,(\partial_{d}u)^2+\partial_{dd}u\,\partial_{d}w=0,\\
    \partial_{id }w\,\partial_{d}u+\partial_{dd}w\,\partial_{i}u\,\partial_{d}u+\partial_{d}w\,\partial_{id}u=0,\\
    \partial_{ij}w+\partial_{di}w\,\partial_{j}u+\partial_{jd}w\,\partial_{i}u+\partial_{dd}w\,\partial_{i}u\,\partial_{j}u+\partial_{ij}u\,\partial_{d}w=0.
\end{align*}
Thus, solving the above system with respect to the derivatives of $u$, we have
\begin{align}
\label{e:der1}
     &\partial_{d} u =\frac{1}{\partial_{d} w},\\
     \nonumber
     &\partial_{i}u=-\frac{\partial_{i} w}{\partial_{d}\,w},\\
     \nonumber
     &\partial_{dd}u=- \frac{\partial_{dd}w}{(\partial_{d}w)^3},\\
     \nonumber
     &\partial_{id}u=-\frac{\partial_{id }w}{(\partial_{d} w)^2}+ \frac{\partial_{dd}w\,\partial_{i} w}{(\partial_{d}w)^3},\\
    \nonumber
     &\partial_{ij}u= -\frac{\partial_{ij}w}{\partial_{d}w}
     +\frac{\partial_{j} w\partial_{id}w}{(\partial_{d}w)^2}
     +\frac{\partial_{jd}w\partial_{i} w}{(\partial_{d}w)^2}
     -\frac{\partial_{dd}w\,\partial_{i} w\partial_{j} w}{(\partial_{d}w)^3}.
\end{align}
Hence, if $w=w(x', u(x))$, $u = u(x)$ and $x \in \Omega_u \cap B_\rho$, we obtain
$$
\begin{aligned}
   \Delta u =  - \frac{\partial_{dd}w}{(\partial_{d}w)^3} + \sum_{i= 1}^{d-1}
   \left(
   -\frac{\partial_{ii}w}{\partial_{d}w}
     +2\frac{\partial_{i} w\partial_{id}w}{(\partial_{d}w)^2}
     -\frac{\partial_{dd}w\,(\partial_{i} w)^2}{(\partial_{d}w)^3}\right).
\end{aligned}
$$ 
Concerning the boundary condition, for $x \in \pa \Omega_u \cap B_\rho$, we get
$$
\abs{\nabla u}^2 - \nabla u \cdot V = \frac{1}{(\pa_d w)^2}\left(1+ \sum_{i=1}^{d-1} (\pa_i w)^2\right) -\frac{1}{\pa_d w}\left(- \sum_{i}^{d-1} \pa_i w\cdot V_i + V_d\right),
$$
where $w=w(x', u(x))$ and $u = u(x)$.
Thus, the function $w$ solves the elliptic PDE problem
\begin{equation}
    \label{e:sistema_odografa}
    \left\{
    \begin{aligned}
&- \frac{\partial_{dd}w}{(\partial_{d}w)^3}\left(1 + \sum_{i= 1}^{d-1} (\partial_{i} w)^2\right)
+ \frac{2}{(\partial_{d}w)^2}\sum_{i= 1}^{d-1} \partial_{i} w\partial_{id}w - \frac{1}{\partial_{d}w} \sum_{i=1}^{d-1}\partial_{ii}w =0
   &&&\hbox{in } \mathcal{D}\cap \{y_d>0\},\\
     &\frac{1}{(\pa_d w)^2}\left(1+ \sum_{i=1}^{d-1} (\pa_i w)^2\right) -\frac{1}{\pa_d w}\left( V_d - \sum_{i=1}^{d-1} \pa_i w\cdot V_i\right)=0 &&&\hbox{on } \mathcal{D}\cap \{y_d=0\}.
    \end{aligned}
    \right.
\end{equation}

We observe that there exists $\sigma >0$ such that, in $\mathcal{D}\cap B_\sigma$, \eqref{e:sistema_odografa} is a quasilinear elliptic PDE problem with oblique boundary condition
$$B(\zeta):=\frac{1}{\zeta_d^2}\left(1+ \sum_{i=1}^{d-1} \zeta_i\right) -\frac{1}{\zeta_d}\left( V_d - \sum_{i=1}^{d-1} \zeta_i V_i\right).$$
Indeed, by \eqref{e:der1} and $\nabla u(0) = V_de_d$, we have 
$$\frac{\partial B(\zeta)}{\partial \zeta_d}\Bigr\rvert_{\zeta =\nabla w(0)}=
\frac{\partial }{\partial \zeta_d}\left(\frac{1}{\zeta_d^2}\left(1+ \sum_{i=1}^{d-1} \zeta_i\right) -\frac{1}{\zeta_d}\left( V_d - \sum_{i=1}^{d-1} \zeta_i V_i\right)\right)\Bigr\rvert_{\zeta =\nabla w(0)}=-V_d^3\not=0,$$
implying, by continuity, that for all $y_0\in \{y_d=0\}\cap B_\sigma$ it holds
$$\frac{\partial B(\zeta)}{\partial \zeta_d}\Bigr\rvert_{\zeta =\nabla w(y_0)}\not=0,$$ choosing $\sigma$ small enough. Thus, we can apply \cite[Proposition 11.21]{lieberman2013oblique}, getting that $w\in C^{2,\alpha}(\mathcal{D}\cap B_{\sfrac{\sigma}{2}})$.
Finally, by \cite[Theorem 11.1]{adn} and \cite[Section 6.7]{morrey-multiple-integrals}, the function $w \in C^\infty(\mathcal{D}\cap B_{\sfrac{\sigma}{4}})$ and it is analytic there, concluding that the same regularity results hold for $\pa \Omega_u$.
\end{proof}

\begin{proof}[Proof of  \cref{thm:piatto-implica-regolare}]
    Let $x_0\in\partial\Omega_u\cap B_{\sfrac{1}{2}}$ and 
    $u$ be $\overline \eps$-flat along the direction $\nu \in\partial B_1$. 
    By \cref{th:eps_reg_C2} and \cref{lemma:proprieta_blowup}, we have
    $\nabla u(x_0)=c_{x_0}\nu_{x_0}$, where $c_{x_0}:=V\cdot\nu_{x_0}\ge \delta$. 
    Then, by \cref{thm:odografa_Cinfty}, there exists $\rho=\rho(x_0)>0$ such that $\partial\Omega_u\cap B_{\rho}(x_0)$ is of class $C^{\infty}$ and analytic. Thanks to a covering argument, we obtain the same regularity result for $\partial\Omega_u\cap B_{\sfrac{1}{2}}$, yielding the thesis.
\end{proof}

\section{Boundary regularity at contact points}
\label{sec:reg_contact_points}

In this section, we study regularity of viscosity solutions $u:\mathcal{B}\to \R^+$ of the problem \eqref{e:viscosity-sol-K_D} around points like $x_0 \in \partial \Omega_u \cap K \cap \mathcal{B}$, proving \cref{th:boundary_reg_contact_points}. Precisely, we consider points of the following type.
\begin{definition}
\label{def:contact_branching}
    We say that $x_0$ is a {\em contact point} if 
    $$
    x_0\in\partial\Omega_u\cap K\cap\mathcal{B}.
   $$
We say that a contact point $x_0\in\partial\Omega_u\cap K\cap\mathcal{B}$ is a {\em branching point} if for every $r>0$ 
    $$\partial\Omega_u\cap D\cap B_r(x_0) \not=\emptyset.$$
    Finally, we say that a contact point $x_0$ is {\em a regular point} if there exists $\beta>0$ such that
    $$u_{x_0,r}(x):=\frac{u(x_0+rx)}{r}\to\beta (x\cdot \nu_K(x_0))^+,$$ 
    uniformly on every compact set of $\R^d$, where $\nu_K(x_0)$ is the outer normal to $\partial K$.
\end{definition}
 
The main result of this section is stated as follows.
\begin{theorem}
\label{th:eps_reg_bordo}
	Let $u:\mathcal{B}\to \R^+$ be a viscosity solution of \eqref{e:viscosity-sol-K_D} and let $x_0 \in \partial \Omega_u \cap K \cap \mathcal{B}$ be a contact point. 
    Let $\delta >0$ and suppose that $V\cdot\nu_K(x_0) \geq 4\delta >0$, where $\nu_K(x_0)$ is the outer normal to $\partial K$. Then there exists $\rho = \rho(x_0,\delta)>0$ such that $\partial \Omega_u\cap B_\rho(x_0)$ is a $C^{1, \alpha}$ graph for every $\alpha\in (0,\alpha_0)$,
    where $\alpha_0$ is defined in \cref{thm:oblthin2}.
\end{theorem}
\begin{proof}
    The proof is divided into four steps split into four subsections:
\begin{itemize}
    \item in \cref{subsec:flatproblem}, we change coordinates to make flat the boundary of $K$;
    \item in \cref{subsec:branching_points_are_regular}, we show that all contact points $x_0 \in \partial \Omega_u \cap B_1'$ are regular points;
    \item in \cref{subsec:iof_branching}, we provide an improvement of flatness lemma for branching points;
    \item in \cref{subsec:final_bordo}, we prove that, around contact points, the free boundary is a $C^{1,\alpha}$-graph. 
\end{itemize}
    Moreover, we point out that \cref{th:boundary_reg_contact_points} is an immediate consequence of \cref{th:eps_reg_bordo}.
\end{proof}

\noindent In the following, we consider the rescalings
$$
u_{x_0, r}:= \frac{u(x_0+ r x)}{r} \qquad \hbox{and} \qquad u_{r}:= \frac{u(r x)}{r}.
$$
Moreover, for all sets $E\subset\R^d$, we define
$$
E^+:= E\cap \{x_d >0\}, \qquad \hbox{ and } \qquad E':= E \cap \{x_d =0\}.
$$

\subsection{Setting and properties of the flat problem}\label{subsec:flatproblem}

According to the change of coordinates in \cref{sec:change_coordinates} and to align the notation as in the previous sections, we will consider Lipschitz viscosity solutions $u:B_1\to\R^+$ of
    \begin{equation}
\label{e:u_viscosity_sol_riscalato}
\left\{\begin{aligned}
&\mathcal{L}u:={\rm div}\left(\tens{A}(x) \nabla u\right) =0 &&\Omega_u \cap B_1^+,\\
&\tens{A}(x) \nabla u \cdot \nabla u = \tens{A}(x) \nabla u \cdot {V}(x) &&\partial \Omega_u \cap B_1^+,\\
&\tens{A}(x) \nabla u \cdot \nabla u \geq \tens{A}(x) \nabla u \cdot {V}(x) &&\partial \Omega_u \cap B_1',\\
&u = 0 &&\{x_d \leq 0\},
\end{aligned}
\right.
\end{equation}
for some smooth
symmetric uniform elliptic matrix $\tens{A}:B_1\to\R^d\times \R^d$ with $\tens{A}(0)= \tens{I}$ and for some smooth 
vector $ V:B_1\to\R^d$.
  We point out 
    that $u$ is a viscosity solution of \eqref{e:u_viscosity_sol_riscalato}
according to a slight modification of  
\cref{def:def-sol_sezione_esistenza}: precisely
$\mathcal{B}= B_1$, $K= \{x_d \leq 0\}$ and $D= \{x_d >0\}$.
Defining for every $x\in B_1$ 
    \begin{equation}
        \label{e:bound_omega_x}
        \omega_x:= \frac{\tens{A}(x)V(x)\cdot e_d}{\tens{A}(x)e_d\cdot e_d},
    \end{equation} 
    we have that
    \begin{equation*}
        \left\{
    \begin{aligned}
        \tens{A}(x) V(x)\cdot e_d&\ge 2\delta,\\
    \omega_x&\ge2\delta,
    \end{aligned}
    \right. \qquad  \hbox{for every } x \in B_1.
    \end{equation*}  
Moreover, by \eqref{e:widetildeV-diff}, we also suppose that 
    \be\label{eq:diff-v-v0}\| \tens{A}(x)-\tens{I}\|_{L^\infty(B_1)} \leq \Pi_0\qquad\text{and}\qquad\| \tens{A}(x)V(x) - V(0)\|_{L^\infty(B_1)} \leq \Pi_0,\ee    
    for some $\Pi_0>0$ small enough.
By \cref{lemma:graph}, \eqref{eq:oss-utile} and \cref{lemma:proprieta_blowup}, the following results hold true.
\begin{lemma}
\label{lemma:rimango_in_parte_positiva}
    Let $u:B_1\to\R^+$ be a viscosity solution of \eqref{e:u_viscosity_sol_riscalato}. Suppose that $x_0\in B_1^+\cap\Omega_u,$ and $\gamma:\R\to\R^d$ is such that 
    $$\left\{
    \begin{aligned}
         \gamma'(t)&= \tens{A}(\gamma(t))V(\gamma(t)),\\
        \gamma(0)&=x_0.
    \end{aligned}
    \right.$$
    Then $\gamma(t)\in\Omega_u$ for every $t>0$ such that $\gamma(t)\in B_1^+$. Equivalently, if $x_0\in B_1^+$ is such that $u(x_0)=0$, then $u(\gamma(t))=0$ for every $t<0$ such that $\gamma(t)\in D.$ In particular, $\partial\Omega_u\cap B_1$ is a graph.
\end{lemma}
\begin{lemma}\label{remark:interior}
    Let $u:B_1\to\R^+$ be a viscosity solution of
    \be\label{interior}
\left\{
\begin{aligned}
&{\rm div}\left(\tens{A}(x) \nabla u\right) =0 &&\Omega_u \cap B_1,\\
&\tens{A}(x) \nabla u \cdot \nabla u = \tens{A}(x) \nabla u \cdot {V}(x) &&\partial \Omega_u \cap B_1. 
\end{aligned}
\right.
\ee
for some smooth
symmetric uniform elliptic matrix $\tens{A}:B_1\to\R^d\times \R^d$ with $\tens{A}(0) = \tens{I}$ and for some smooth 
vector $ V:B_1\to\R^d$ satisfying \eqref{eq:diff-v-v0}.
For all $\delta>0$ and $\gamma\in(0,1)$, there are constants $\e_0>0$ and $C>0$ such that the following holds. 
  Let $\nu\in\pa B_1$ be a unit vector
  such that for every $x \in B_1$ we have
$$\widetilde \omega_{x,\nu}:=\frac{\tens{A}(x)V(x)\cdot \nu}{\tens{A}(x)\nu\cdot \nu}\ge2\delta,$$ 
and $u$ is $\eps$-flat along $\nu$-direction, with $\eps\in(0,\eps_0]$, namely for every $x\in B_1$
    $$\left(x \cdot \widetilde \omega_{0,\nu}\, \nu - \eps\right)^+\le u(x)\le  \left(x \cdot  \widetilde\omega_{0,\nu} \, \nu+ \eps\right)^+.
    $$
    Then, for every $x_0\in \partial\Omega_u\cap B_{\sfrac{1}{2}}$, there is a unit vector $\nu_{x_0}\in\pa B_1$ such that, for every $r\in (0,\sfrac{1}{2})$,
    \begin{equation*}
        \norm{u(x_0+x)-x\cdot \widetilde\omega_{x_0,\nu_{x_0}}\nu_{x_0}}_{L^\infty(B_r\cap\overline\Omega_u)}\le C\eps r^{1+\gamma}.
    \end{equation*}
\end{lemma}

\subsection{Contact points are regular points}
\label{subsec:branching_points_are_regular}
The main result of this subsection is proving that every contact point $x_0\in\partial\Omega_u\cap B_1'$ of the problem \eqref{e:u_viscosity_sol_riscalato} is a regular point. This is contained in the following proposition.

\begin{proposition}
\label{prop:punti_regolari_al_bordo}
    Let $u$ be a viscosity solution to \eqref{e:u_viscosity_sol_riscalato} and $x_0\in\partial\Omega_u \cap B_1'$ be a contact point. Let $\omega_{x_0}>0$ be as in \eqref{e:bound_omega_x}. Then, $x_0$ is a regular point. More precisely, as $r\to0^+$:
    \begin{enumerate}
        \item\label{1}
         either $u_{x_0,r} \to \omega_{x_0}x_d$,
    \item or $u_{x_0,r}\to \beta x_d$ for some $\beta> \omega_{x_0}$ and $u>0$ in $ B_\rho^+(x_0)$, for some $\rho>0$.
    \end{enumerate}
    In particular, if $x_0$ is a branching point, then only \ref{1} can occur.
\end{proposition}
Before proving \cref{prop:punti_regolari_al_bordo}, we show the following remark.

\begin{remark}
\label{rem:caratterizzazione_punti}
    By \cref{prop:punti_regolari_al_bordo}, in $B_1'$ there are three types of points.
    \begin{itemize}
        \item {\em Branching points} are points defined in \cref{def:contact_branching} and satisfying $u_{x_0,r} \to \omega_{x_0}x_d$.
        \item {\em Interior contact points} are points where the free boundary locally coincides with $\{x_d=0\}$. In this case $u_{x_0,r} \to \beta x_d$ and $\beta \ge \omega_{x_0}$. Moreover, they form an open set in the topology of $\{x_d=0\}$.
        \item {\em Vanishing points} are points belonging to the interior set of $\{u=0\}$. In particular $u_{x_0,r} \to 0$.
    \end{itemize}
\end{remark}
In the following, without loss of generality, we can assume that $x_0=0$.
We provide the proof of \cref{prop:punti_regolari_al_bordo} into two steps. First, in \cref{subsec:espansione_u}, we show that, near $0\in\partial\Omega_u$, a viscosity solution of \eqref{e:u_viscosity_sol_riscalato} has an expansion of type $u(x)=\beta x_d+o(|x|)$ for some $\beta \ge0$. Then, in \cref{subsec:bound_beta}, we show that $\beta \geq \omega_{0}$ and if $\beta > \omega_{0}$, then $u>0$ in a smaller ball $B_\rho^+$.

\subsubsection{Expansion of \texorpdfstring{$u$}{u}}
\label{subsec:espansione_u}
First, in the following two lemmas, we show that $u(x)=\beta x_d+o(|x|)$ for some $\beta \ge0$.
Such a proof is similar to \cite[Lemma 11.17]{caffarelli2005geometric}.

\begin{lemma}\label{lemma:caffarellisalsa_upper}
    Let $u:B_1\to\R^+$ be a viscosity solution to \eqref{e:u_viscosity_sol_riscalato}. Then, there exists $\beta\ge0$ such that 
    $$u(x)\leq\beta x_d+o(\abs{x})$$
   as $x\to 0$, with $x \in B_1^+$.
\end{lemma}

\begin{proof}
Let $y_0:=-\sfrac{e_d}{2}$ and let $\Psi$ be the function defined as follows
\begin{equation}
\label{eq:Psi_solita_funzione_armonica}
    \Psi: \overline B_1(y_0) \to \R^+ \quad \hbox{ such that } \quad \left\{
\begin{aligned}
   \mathcal{L}\psi &= 0 &&\text{in } B_1(y_0) \setminus B_{\sfrac{1}{2}}(y_0),\\
    \Psi &= 1
    &&\text{on } \partial B_1(y_0),\\
    \Psi &= 0 &&\text{in } \overline B_{\sfrac{1}{2}}(y_0).
 \end{aligned}
\right.
\end{equation}
As a consequence of Hopf lemma, around $0$ for all $x \in B^+_1(y_0)$, the function $\Psi$ reads as
\begin{equation}
    \label{e:equazione_hopf_Psi}
    \Psi(x)=c_dx_d+o(\abs{x}),
\end{equation}
where $c_d>0$ is a dimensional constant.
Let us introduce the sequence
$$\beta_k:=\inf\left\{m\ge0:\ u(x)\le m\Psi(x)\; \text{  in }B_{2^{-k}}^+\right\}.$$
We observe that $\beta_{1} < +\infty$. Indeed, since $u =0$ in $\{x_d = 0\}$ and $\Psi = 1$ in $\pa B_1(y_0)$, we get that
$$
u\le m\Psi  \quad \hbox{ in }\partial B_1^+(y_0)^.
$$
for some $m>0$ large enough. By the fact that in $B_1^+(y_0)$, $u$ is $\mathcal{L}$-subharmonic and $\Psi$ is $\mathcal{L}$-harmonic, by comparison, we obtain that  
$$
u\le m\Psi  \quad \hbox{ in }B_1^+(y_0),
$$
implying that $\beta_1<+\infty$.
Finally, since $\beta_k$ is a decreasing sequence and $\beta_1<+\infty$, the limit exists, namely
$$\lim_{k\to+\infty}\beta_k :=\widetilde\beta\in[0,+\infty).$$
Applying \eqref{e:equazione_hopf_Psi} and using the definition of $\widetilde{\beta}$, we end up with
\begin{equation*}
    \label{e:stima_u_step1}
    u(x)\le \widetilde \beta \Psi(x)+o(\abs{x})= c_d \widetilde \beta x_d+o(\abs{x})=: \beta x_d+o(\abs{x}).
\end{equation*}
as $x\to 0$ with $x\in B_1^+(y_0)$, which is exactly the claimed lower bound.
\end{proof}

\begin{lemma}\label{lemma:caffarellisalsa_lower}
    Let $u:B_1\to\R^+$ be a viscosity solution to \eqref{e:u_viscosity_sol_riscalato} and take $\beta\ge0$ as in \cref{lemma:caffarellisalsa_upper}. Then
    $$u(x)\ge \beta x_d+o(\abs{x})$$
   as $x\to 0$, with $x \in B_1^+$.
\end{lemma}
\begin{proof}
Suppose by contradiction that there exist $\eta>0$ and a sequence of points $\{x_k\}\subset B_1^+(y_0)$ with $\abs{x_k}=:r_k\to0^+$ such that
\begin{equation}
    \label{e:ipotesi_assurda}
    u(x_k)<\beta \, x_k \cdot e_d-\eta r_k.
\end{equation}
Let us fix $k\in\N$ and let us consider a ball $B_{\eps r_k}(x_k)$ with $\eps$ for the moment arbitrary. For all $x\in B_{\eps r_k}(x_k)$, using the fact that $u$ is $L$-Lipschitz and \eqref{e:ipotesi_assurda}, we have 
$$
\begin{aligned}
    u(x)&\le u(x_k)+L\abs{x-x_k}\\
    &< \beta x_k\cdot e_d-\eta r_k+L\abs{x-x_k}\\
    &\le \beta x_d+\beta\abs{x-x_k}-\eta r_k+L\abs{x-x_k}\\
    &\le \beta x_d-\eta r_k+2L\abs{x-x_k}.
\end{aligned}
$$
Then, by choosing $\eps>0$ small enough we get
\be\label{e:stimaeta}u(x)< \beta x_d -\frac \eta 2r_k,\ee
for every $x\in B_{\eps r_k}(x_k)$. For all $x \in \overline{B_1^+}$, let 
\begin{equation}
    \label{e:def_wk}
    w_k(x):= - \left(\frac{u(r_k x)}{r_k} - \beta x_d\right).
\end{equation}
We observe that, applying \cref{lemma:caffarellisalsa_upper}, we have that for some $\sigma_k\to0^+$ as $k\to+\infty$ and for every $x\in (\partial B_1)^+$
\begin{equation}
    \label{e:stima_w_k_sigma_k}
    w_k(x)\ge -\sigma_k.
\end{equation}
Hence, as a consequence of \eqref{e:stimaeta}, \eqref{e:def_wk} and \eqref{e:stima_w_k_sigma_k}, we get
$$
\left\{
\begin{aligned}
    &w_k(x)\ge \frac\eta2 &&\text{on } (\partial B_1)^+\cap B_\eps\left(\sfrac{x_k}{r_k} \right),\\
    &w_k(x)\ge -\sigma_k&&\text{on } (\partial B_1)^+\setminus B_\eps\left(\sfrac{x_k}{r_k} \right),\\
    &w_k(x)=0&&\text{on } B_1'. 
\end{aligned}
\right.
$$
For every $k \in \N$, let us consider functions $h_k$ defined as 
$$
\left\{
\begin{aligned}
 &\mathcal{L}h_k(x)=0&&\text{in } B_1^+,\\
    &h_k(x)=\frac\eta2 &&\text{on } (\partial B_1)^+\cap B_\eps\left(\sfrac{x_k}{r_k} \right),\\
    &h_k(x)= -\sigma_k&&\text{on } (\partial B_1)^+\setminus B_\eps\left(\sfrac{x_k}{r_k} \right),\\
    &h_k(x) = 0&&\text{on } B_1'.%\cap \{x_d = 0\}.
\end{aligned}
\right.
$$
We observe that $h_k=h_{1,k}+ h_{2,k}$, where
$$
\left\{
\begin{aligned}
 &\mathcal{L} h_{1,k}(x)=0&&\text{in } B_1^+,\\
    &h_{1,k}(x)=0 &&\text{on } (\partial B_1)^+\cap B_\eps\left(\sfrac{x_k}{r_k} \right),\\
    &h_{1,k}(x)= -\sigma_k&&\text{on } (\partial B_1)^+\setminus B_\eps\left(\sfrac{x_k}{r_k} \right),\\
    &h_{1,k}(x) = 0&&\text{on } B_1'.
\end{aligned}
\right. \qquad \qquad 
\left\{
\begin{aligned}
 &\mathcal{L} h_{2,k}(x)=0&&\text{in } B_1^+,\\
    &h_{2,k}(x)=\frac\eta2 &&\text{on } (\partial B_1)^+\cap B_\eps\left(\sfrac{x_k}{r_k} \right),\\
    &h_{2,k}(x)= 0 &&\text{on } (\partial B_1)^+\setminus B_\eps\left(\sfrac{x_k}{r_k} \right),\\
    &h_{2,k}(x) = 0&&\text{on } B_1'.
\end{aligned}
\right.
$$
Since, for every $k\in \N$, $w_k$ is $\mathcal{L}$-superharmonic and $h_k$ is $\mathcal{L}$-harmonic, by comparison principle in $B_1^+$, we have that 
$$w_k\ge h_k=h_{1,k}+h_{2,k}.$$
By Hopf Lemma, there are constants $c_{1,k},c_{2,k}>0$ such that
$$
h_{1,k}= -c_{1,k}x_d+ o(\abs{x}) \qquad \hbox{ and } \qquad h_{2,k}= c_{2,k}x_d+o(\abs{x})$$ as $x\to 0$ with $ x\in B_1^+$.
Moreover, since 
$$
\frac{x_k}{r_k} \to x_{\infty}\in (\partial B_1)^+ \qquad \hbox{ and }\qquad  \sigma_k\to0^+,
$$
then $c_{1,k}\to0$ and $c_{2,k}\to c_{\infty}>0$. Therefore, for $k$ large enough, we have
$$w_k\ge \frac{c_{\infty}}{2}x_d+o(\abs{x})$$
as $x\to 0$ with $x\in B_1^+$. 
Thus, replacing the function $w_k$ defined in \eqref{e:def_wk} and defining $\widetilde c:=\sfrac{c_\infty}{2c_d}$, we have
$$\frac{u(r_kx)}{r_k}\le \beta x_d-\frac{c_\infty}2x_d=(\widetilde \beta-\widetilde c)\Psi(x)+o(|x|)$$ 
as $x\to0$ with $x\in B_1^+,$ where $\Psi$ is defined in \eqref{eq:Psi_solita_funzione_armonica}.
This yields to a contradiction by the definition of $\widetilde \beta$ provided in \cref{lemma:caffarellisalsa_upper}.
\end{proof}

Combining \cref{lemma:caffarellisalsa_upper} and \cref{lemma:caffarellisalsa_lower}, the following corollary holds true.
\begin{corollary}
\label{cor:valeugualglianza_blow_up_u}
    Let $u:B_1\to\R^+$ be a viscosity solution to \eqref{e:u_viscosity_sol_riscalato}. Then, there exists $\beta\ge0$ such that 
    $$u(x)=\beta x_d+o(\abs{x})$$
   as $x\to 0$, with $x \in B_1^+$.
\end{corollary}

\subsubsection{Bounds for \texorpdfstring{$\beta$}{beta}}
\label{subsec:bound_beta}
Once \cref{cor:valeugualglianza_blow_up_u} holds true, to ensure that contact points are regular, we need to know the exact value of $\beta$.
In the following lemma, we prove that $\beta\ge \omega_{0}$, where $\omega_{0}$ is defined in \eqref{e:bound_omega_x}. Moreover, we remark that, throughout this subsection, we do not assume that $\tens{A}(0) = \tens{I}$ since such an argument is valid for all $x_0\in\partial\Omega_u\cap B_1'$ and without loss of generality, we are considering $x_0 = 0$.

\begin{lemma}\label{lemma:bound_lower-for-beta}
     Let $u:B_1\to\R^+$ be a viscosity solution to \eqref{e:u_viscosity_sol_riscalato} and $0\in\partial\Omega_u$. Let $\omega_0>0$ be as in \eqref{e:bound_omega_x}. 
     Suppose that, for some $\beta\ge0$, we have \begin{equation}
         \label{e:ipotesi_espansione}
         u(x)=\beta x_d+o(\abs{x})
     \end{equation}
     as $x\to 0$, with $x \in B_1^+$.
   Then 
   $\beta\geq\omega_{0}.$
\end{lemma}
\begin{proof} 
As a consequence of \eqref{e:ipotesi_espansione}, we have that, as $r \to 0^+$,
     $u_{r} \to \beta x_d^+$
 uniformly in every compact subsets of $\R^d$. 
Suppose by contradiction that $\beta <\omega_{0}$. Thus, we can find $\overline \beta>0$, such that $\beta<\overline\beta<\omega_{0}$,  namely
    \be
    \label{eq:contrad-for-beta}
    \tens{A}(0)(\overline{\beta} e_d)\cdot (\overline{\beta} e_d)<\tens{A}(0)(\overline{\beta} e_d)\cdot V(0),
    \ee 
    since $\overline\beta> 0$.
 Let $\chi:\partial B_1^+\to \R^+$ be a smooth cut off function defined as 
 $$
\chi:=\left\{
\begin{aligned}
 &0&&\text{on } B_{\sfrac{1}{4}}',\\
    &1&&\text{on } B_{\sfrac{1}{2}}'\cup \left(\partial B_1\right)^+.
\end{aligned}
\right.$$
    For every $\eta>0$, let $v$ be the function given by
    $$\left\{
\begin{aligned}
 &\mathcal{L} v=0&&\text{in } B_1^+,\\
    &v=\overline\beta x_d+\eta\chi(x) &&\text{on } \partial B_1^+,\\
\end{aligned}
\right.$$
By classical Schauder estimates (see e.g. \cite{gilbarg1977elliptic}), for some constant $C>0$, we have that 
\begin{equation}
    \label{e:stima_Schauder}
    \|\nabla v-\overline\beta e_d\|_{L^\infty(B_{\sfrac{1}{2}}^+)}\le C\eta.
\end{equation}
Thus, combining \eqref{e:stima_Schauder} with the hypothesis on $\overline\beta$ \eqref{eq:contrad-for-beta}, we can choose $\eta>0$ small enough such that
\be\label{eq:stima1000}\tens{A}(0)\nabla v(0)\cdot \nabla v(0)<\tens{A}(0)\nabla v(0)\cdot V(0).\ee 
Since $u_r$ converges to $\beta x_d^+$ as $r\to 0^+$, there exists a small radius $\rho = \rho (\eta) >0$ such that for every $x\in(\partial B_1)^+$
$$u_\rho(x)\le \beta x_d+\eta\le\overline \beta x_d+\eta$$
and $u_\rho$ is a viscosity solution of \eqref{e:u_viscosity_sol_riscalato}
 with matrix $\tens{A}_\rho(x):=\tens{A}(\rho x)$ and vector $V_\rho(x):=V(\rho x)$.
Take the family of functions
$$v_t(x):=v+t.$$ 
We observe that for $t>0$ large enough, we have $v_t>u_\rho$ on $\overline{B_{1}^+}$. By decreasing $t$, we can find a value $t_0 \geq 0$ and a touching point $z\in \overline {B_{1}^+}$ such that $v_{t_0}$ touches $u_\rho$ from above at $z$. Precisely, we notice that $t_0\ge0$ since $u_\rho(0)=v_0(0) = v(0)=0$. 
We aim to prove that $t_0=0$. 
Suppose by contradiction that $t_0>0$. 
Since for every $x\in B_1'$ 
$$u_\rho(x)=0\le v(x)<v_{t_0}(x)$$ 
and since for every $x\in( \partial B_1)^+$
$$u_\rho(x)\le \overline\beta x_d+\eta= v(x)<v_{t_0}(x),$$ 
we conclude that $z\in B_{1}^+$. 
More precisely, $z\in\Omega_{u_\rho}\cap B_{1}^+$ since $v_{t_0}>0$ in $B_1^+$.
Then, by the strong maximum principle 
$$v_{t_0}\equiv u_\rho\quad\text{in }\overline{B_1^+},$$
yielding to a contradiction since
$v_{t_0}(0)=t_{0}\not=0=u(0)$. 
Therefore $t_0=0$. 
Hence, we have built a function $v$ which touches $u_\rho$ from above at $0$. Moreover, by \eqref{eq:stima1000}, such a function $v$ satisfies $$\tens{A}_\rho(0)\nabla v(0)\cdot\nabla v(0)<\tens{A}_\rho(0)\nabla v(0)\cdot V_\rho(0),$$ which contradicts the property of being a viscosity solution for $u_\rho$.
    \end{proof}
    
\begin{lemma}
\label{lemma:upper_bound_beta}
     Let $u:B_1\to\R^+$ be a viscosity solution to \eqref{e:u_viscosity_sol_riscalato} and $0\in\partial\Omega_u$. Let $\omega_0>0$ be as in \eqref{e:bound_omega_x}. 
     Suppose that, for some $\beta\ge0$, we have
    \be\label{hyp:lemma}u(x)=\beta x_d+o(\abs{x})\ee
   as $x\to 0$ with $x \in B_1^+$.
   Then,
   \begin{itemize}
   \item either $\beta\leq \omega_{0}$,
   \item or $\beta>\omega_{0}$ and $u>0$ in $ B_\rho^+$, for some $\rho>0$.
   \end{itemize}
\end{lemma}
    \begin{proof} 
As a consequence of \eqref{hyp:lemma}, we have that, as $r \to 0^+$, $u_{r} \to \beta x_d^+$ uniformly in every compact subsets of $\R^d$.
 Suppose that 
 $\beta > \omega_{0}$. Since $\omega_{0}>0$, then
 \be\label{eq:beta-contr}\tens{A}(0)({\beta} e_d)\cdot ({\beta} e_d)<\tens{A}(0)({\beta} e_d)\cdot V(0).\ee
 We want to prove that $u>0$ in $B_\rho^+$, for some $\rho>0$. 
 For every $\eta >0$, take a cut off function 
 $\chi\in C^\infty([0,1])$ defined as 
 $$\chi(t):=\left\{
\begin{aligned}
 &0&&\text{for } t\in[0,\sfrac{\eta}{\beta}],\\
    &1&&\text{for } t\in[\sfrac{2\eta}{\beta},1],\\
\end{aligned}
\right.$$
   Let $v$ be the function given by
    \bea\label{def-v-barriera}\left\{
\begin{aligned}
 &\mathcal{L} v=0&&\text{in } B_1^+,\\
    &v=\chi(x_d)(\beta x_d-\eta) &&\text{on } \partial B_1^+,\\
\end{aligned}
\right.\eea
By classical Schauder estimates (see e.g. \cite{gilbarg1977elliptic}), for some constant $C>0$, we have that 
\be\label{eq:stima-schauder2}\|\nabla v-\beta e_d\|_{L^\infty(B_{1}^+)}\le C\eta.\ee
Hence, as a consequence of \eqref{eq:stima-schauder2} and the hypothesis on $\beta$ \eqref{eq:beta-contr}, we can choose $\eta>0$ and $\rho>0$ small enough such that for every $x\in B_1^+$
\be\label{e:stima_gradientr_Vrho}\tens{A}_\rho( x)\nabla v(x)\cdot \nabla v(x)<\tens{A}_\rho(x)\nabla v(x)\cdot V_\rho( x), \ee
where $\tens{A}_\rho(x):=\tens{A}(\rho x)$ and $V_\rho(x):=V(\rho x)$.
 Since $u_r \to \beta x_d^+$ as $r \to 0^+$, up to decreasing $\rho=\rho(\eta)$, it holds for every $x\in(\partial B_1)^+$ that
 $$u_\rho(x)\ge \beta x_d-\eta,$$ 
 and $u_\rho$ is a viscosity solution of \eqref{e:u_viscosity_sol_riscalato}
 with matrix $\tens{A}_\rho(x)$ and vector $V_\rho(x)$.
 Let us consider the family of functions
    $$v_t(x):=v-t.$$ 
    We observe that $v_t<u_\rho$ on $\overline{B_{1}^+}$ for some $t>0$ large enough. By decreasing $t$, we can find a value $t_0\ge0$ and a touching point $z\in \overline {B_{1}^+}$ such that $v_{t_0}$ touches $u_\rho$ from below at $z$.
    Precisely, we notice that $t_0\ge0$, since $u_\rho(0)=v_0(0)= v(0)=0$.
    We aim to show that $t_{0} = 0$. Let us suppose by contradiction that $t_0>0$. 
    Since for every $x\in\partial B_1^+\cap\{x_d\le \sfrac{\eta}{\beta}\}$,
    $$u_\rho(x)\ge 0=v(x)>v_{t_0}(x)$$ 
    and for every $x\in \partial B_1\cap\{x_d>\sfrac{\eta}{\beta}\}$
    $$u_\rho(x)\ge \beta x_d-\eta\ge \chi(x_d)(\beta x_d-\eta)=v(x)>v_{t_0}(x),$$ 
    then $z\in B_{1}^+$. Thus, the following three cases can appear.
    \begin{itemize}
        \item If $z\in \partial\Omega_{u_\rho}\cap B_{1}^+$, then we have built a test function $v_{t_0}$ which touches $u_\rho$ from below at $z$. Moreover, such a function satisfies 
        $\tens{A}_\rho(z)\nabla v_{t_0}(z)\cdot \nabla v_{t_0}(z)>\tens{A}_\rho(z)\nabla v_{t_0}(z)\cdot V_\rho(z)$, by \eqref{e:stima_gradientr_Vrho}, which yields to a contradiction since $u_\rho$ is a viscosity solution. 
        \item If $z\in\Omega_{u_\rho}\cap B_{1}^+$, then by the strong maximum principle, $u_\rho\equiv v_{t_0}$ in the connected component of $\Omega_{u_\rho}$ containing $z$. Thus, there exists a point $y\in \partial\Omega_{u_\rho} \cap B_{1}^+$ such that $v_{t_0}$ touches $u_\rho$ from below at $y$; getting the contradiction as in the previous case.
        \item If $z\in B_1^+\setminus\overline\Omega_u$, i.e.~$z$ belongs in the interior of $\{u=0\}$, then $v_{t_0} = 0$ in a neighborhood of $z$. 
        Thus, by the maximum principle, $v_{t_0}$ is identically $0$ in a neighborhood of $z$, which is a contradiction.  
    \end{itemize}
   Therefore $t_0=0$. Hence, we obtain that $u_\rho\ge v>0$ in $B_1^+$, which is exactly the thesis.
\end{proof}

We are finally in position to prove the main result of this section, namely \cref{prop:punti_regolari_al_bordo}, which shows that all contact points $x_0 \in \partial \Omega_u \cap \{x_d = 0\}$ are regular points.

\begin{proof}[Proof of \cref{prop:punti_regolari_al_bordo}]
Since $u$ is a viscosity solution to \eqref{e:u_viscosity_sol_riscalato}, as a consequence of \cref{cor:valeugualglianza_blow_up_u}, there exists $\beta \geq 0$ such that
$$
u(x)=\beta x_d+o(|x|).
$$
Combining \cref{lemma:bound_lower-for-beta} and \cref{lemma:upper_bound_beta} we get exactly
the dichotomy of the proposition, namely the thesis.
\end{proof}

\subsection{Improvement of flatness for branching points}\label{subsec:iof_branching}
In this subsection we prove the  following improvement of flatness lemma for branching points of viscosity solutions of \eqref{e:u_viscosity_sol_riscalato}.

\begin{lemma}[Improvement of flatness]
\label{lemma:IOF_bordo}

Let $u:B_1\to\R^+$ be a viscosity solution of \eqref{e:u_viscosity_sol_riscalato} and suppose that $0\in\partial \Omega_u$ is a branching point. 
Let $\alpha\in(0, \alpha_0)$, where $ \alpha_0$ is defined in \cref{thm:oblthin2}.
For all $\delta>0$ and $\alpha\in(0,\alpha_0)$, there are constants $\overline \e>0$, $C>0$ and $\eta\in(0,1)$ such that the following holds.
  Suppose that, for all $x\in B_1$, 
  \be\label{nuova4}
  \tens{A}(x)e_d\cdot V(x)\ge \delta 
  \ee  
and $u$ is $\eps$-flat along the direction $e_d$ for some $\eps\in(0,\overline\eps]$, namely, for every $x\in B_1$
    $$
    \left(\omega_0x_d - \eps\right)^+\le u(x)\le  \left(\omega_0x_d+ \eps\right)^+,
    $$ 
    where $\omega_0$ is defined in \eqref{e:bound_omega_x}.
    Moreover, we also suppose that
    \begin{equation}
    \label{e:ipotesi_PH_tensoreA}
        \| \tens{A}(x)- \tens{I}\|_{L^\infty(B_1)}\le \eps^{\sfrac{\alpha_0}{\alpha}}, \qquad \hbox{and} \qquad \tens{A}(0) = \tens{I},
    \end{equation}
    and
    \begin{equation}
    \label{e:ipotesi_PH_vettoreV}
        \| V(x)- V(0)\|_{L^\infty(B_1)}\le \eps^{\sfrac{\alpha_0}{\alpha}}.
    \end{equation}
    Then, $u_\eta$ is $\eps \eta^{\alpha}$-flat along $e_d$ in $B_{1}$, namely for every $x\in B_1$
    \begin{equation*}
        \label{e:flatness_scala_eta_bordo}
        \left(\omega_0x_d- \eps\eta^{\alpha}\right)^+\le u_\eta(x)\le  \left(\omega_0x_d+ \eps\eta^{\alpha}\right)^+.
    \end{equation*}
\end{lemma}

Similarly as done in \cref{susec:piatto_implica_C2alpha}, we provide the proof of \cref{lemma:IOF_bordo} by showing a partial Harnack inequality in \cref{subsec:partialharnack_bordo} and studying the regularity of the associated linearized problem in \cref{susec:reg_linarized_proble_bordo}. Finally, the proof of \cref{lemma:IOF_bordo} is presented in \cref{subsec:proof-of-boundary-iof}.

In the following, the next two technical lemmas will be useful.

\begin{lemma}
\label{lemma:useful_non_piace_Giulia}
    For all $\delta>0$, there exist constants $C_0>0$ and $\overline\eps>0$ such that the following holds.
    Let $u:B_1\to\R^+$ be a viscosity solution of \eqref{e:u_viscosity_sol_riscalato} and suppose that \eqref{nuova4}, \eqref{e:ipotesi_PH_tensoreA}, \eqref{e:ipotesi_PH_vettoreV} hold true, for some $\eps\in(0,\overline{\eps}]$. Suppose that, for every $x\in B_1$, we have 
    $$u(x)\ge (\beta x_d-\eps)^+\qquad\text{and}\qquad \beta\ge\omega_0+ C_0\eps,$$ where $\omega_0$ is defined in \eqref{e:bound_omega_x}. Then $u>0$ in $ B_1^+$.
\end{lemma}

\begin{proof}
Let $\chi\in C^\infty([0,1])$ be a cut-off function defined by
 $$\chi(t):=\left\{
\begin{aligned}
 &0&&\text{for } t\in\left[0,\frac{\eps}{(\omega_0+C_0\eps)}\right],\\
    &1&&\text{for } t\in\left[\frac{2\eps}{(\omega_0+C_0\eps)},1\right],\\
\end{aligned}
\right.$$
   and let $v$ be the function given by
    \bea\left\{
\begin{aligned}
 &\mathcal{L} v=0&&\text{in } B_1^+,\\
    &v=\chi(x_d)((\omega_0+C_0\eps) x_d-\eps) &&\text{on } \partial B_1^+.
\end{aligned}
\right.\eea
Similarly as in \cref{lemma:bound_lower-for-beta},
for some $\widetilde C>0$, we have that
$$\|\nabla v-(\omega_0+C_0\eps) e_d\|_{L^\infty(B_{1}^+)}\le \widetilde C\eps,$$ 
and $u\ge v$ on $\partial B_1^+$. 
To prove that $u\ge v$ in $ B_1^+$ it is sufficient to show that $v$ solves, for every $x\in B_1^+$, 
$$ \tens{A}(x)\nabla v(x)\cdot \nabla v(x)> \tens{A}(x)\nabla v(x)\cdot V(x).$$ 
First, we observe that, by \eqref{e:ipotesi_PH_tensoreA} and \eqref{e:ipotesi_PH_vettoreV}, it holds that
\begin{equation}
    \label{e:stima_1_lemma_terribile}
    \tens{A}(x)\nabla v(x)\cdot \nabla v(x)-\tens{A}(x)\nabla v(x)\cdot V(x)\ge\abs{\nabla v(x)}^2-\nabla v(x)\cdot V(0)+o(\eps)
\end{equation}
Moreover, using the definition of $\omega_0$ in \eqref{e:bound_omega_x}, we have 
$$\omega_0^2=\omega_0e_d\cdot V(0),$$ thus, by \eqref{nuova4}, we get
\begin{align}
\nonumber
\abs{\nabla v(x)}^2-\nabla v(x)\cdot V(0)&\ge (\omega_0+C_0\eps)^2e_d\cdot e_d-\widetilde C^2\eps^2-(\omega_0+C_0\eps)e_d\cdot V(0)-\|V\|_{L^\infty(B_1)}\widetilde C\eps\\
\nonumber
&=2\omega_0C_0\eps-C_0\eps e_d\cdot V(0)-\|V\|_{L^\infty(B_1)}\widetilde C\eps+o(\eps)\\
\nonumber
&=C_0\eps V(0)\cdot e_d -\|V\|_{L^\infty(B_1)}\widetilde C\eps+o(\eps)\\
\label{eq:stima34} 
&\ge \left(2C_0\delta -\|V\|_{L^\infty(B_1)}\widetilde C\right)\eps+o(\eps).
\end{align}
Combining \eqref{e:stima_1_lemma_terribile} and \eqref{eq:stima34}, choosing $C_0>0$ large enough and $\eps$ small enough, we get that $u\geq v >0$ in $B_1^+$ which is the thesis.
\end{proof}
\begin{lemma}
\label{lemma:utile}
Let $u:B_1\to\R^+$ be a viscosity solution of \eqref{e:u_viscosity_sol_riscalato}. Suppose that \eqref{e:ipotesi_PH_tensoreA}, \eqref{e:ipotesi_PH_vettoreV} hold true, for some $\eps>0$. Let $\varphi\in C^\infty_c(B_1)$ be a regular function, $\omega_0$ be as in \eqref{e:bound_omega_x} and $q:=\omega_0x_d+\eps \varphi$. 
Then, as $\eps\to0^+$, we have
\begin{equation}
    \label{e:sviluppo_divergenza_eps}
    {\rm div}(\tens{A}(x)\nabla q)=\eps \Delta \varphi+o(\eps)
\end{equation}
and
\begin{equation}
    \label{e:sviluppo_condizione_al_bordo_eps}
    \tens{A}(x)\nabla q\cdot\nabla q-\tens{A}(x)\nabla q\cdot V(x)=\eps\,\nabla \varphi\cdot W(0)+o(\eps),
\end{equation}
where 
\begin{equation}
    \label{e:W(0)_bordo}
    W(0):=2\omega_0 e_d-V(0).
\end{equation}
\end{lemma}

\begin{proof}
By using the definition of $q$, \eqref{e:ipotesi_PH_tensoreA} and \eqref{e:bound_omega_x}, for every $\zeta\in C^\infty_c(B_1)$, ${\rm div}(\tens{A}(x)\nabla q)$ reads as
    \bea \int_{B_1}\tens{A}(x)\nabla q\cdot \nabla \zeta\,dx&=\omega_0\int_{B_1}\tens{A}(x)e_d\cdot \nabla\zeta\,dx+\eps\int_{B_1}\tens{A}(x)\nabla \varphi\cdot \nabla \zeta\,dx\\&=\omega_0\int_{B_1}\left(\tens{A}(x)-\tens{I}\right)e_d\cdot \nabla\zeta\,dx+\eps\int_{B_1}\left(\tens{A}(x)-\tens{I}\right)\nabla \varphi\cdot \nabla \zeta\,dx+\eps\int_{B_1}\nabla \varphi\cdot\nabla \zeta\,dx\\&=\eps\int_{B_1}\nabla \varphi\cdot\nabla \zeta\,dx+o(\eps),\eea 
    as $\eps\to0^+$, getting exactly \eqref{e:sviluppo_divergenza_eps}.

To get \eqref{e:sviluppo_condizione_al_bordo_eps}, we again use the definition of $q$, \eqref{e:ipotesi_PH_vettoreV} and the definition of $\omega_0$ in \eqref{e:bound_omega_x}. Precisely, we obtain
\begin{align*}
    \tens{A}(x)\nabla q\cdot\nabla q-\tens{A}(x)\nabla q\cdot V(x)&=\omega_0^2\tens{A}(x)e_d\cdot e_d+2\eps\omega_0\tens{A}(x)e_d\cdot \nabla \varphi-\omega_0\tens{A}(x)e_d\cdot V(x)-\eps\tens{A}(x)\nabla\varphi\cdot V(x)+o(\eps)
\\&=\omega_0^2 e_d\cdot e_d+2\eps\omega_0 e_d\cdot \nabla \varphi-\omega_0 e_d\cdot V(0)-\eps \nabla\varphi\cdot V(0)+o(\eps)
\\&=\eps \left(2\omega_0e_d-V(0)\right)\cdot \nabla \varphi+o(\eps),
\end{align*}
as $\eps\to0^+$, which is \eqref{e:sviluppo_condizione_al_bordo_eps} having defined $W(0)$ as in \eqref{e:W(0)_bordo}.
\end{proof}

\subsubsection{Partial Harnack inequality}
\label{subsec:partialharnack_bordo}
The first fundamental ingredient to show \cref{lemma:IOF_bordo} is the partial Harnack inequality which can be stated as follows.

\begin{lemma}[Partial Harnack inequality]
    \label{lemma:PH_inequality_bordo}
  Let $u:B_1\to\R^+$ be a viscosity solution of \eqref{e:u_viscosity_sol_riscalato} and suppose that $y\in\overline \Omega_u\cap B_{\sfrac{1}{2}}$. For all $\delta>0$, there are constants $\overline \e>0$, $\rho>0$ and $\sigma>0$ such that the following holds. Suppose that \eqref{nuova4}, \eqref{e:ipotesi_PH_tensoreA}, \eqref{e:ipotesi_PH_vettoreV} hold true, for some $\e\in (0,\overline\eps]$.
    If there are two real numbers $a\le b$ with $\e=b-a\leq \overline{\e}$ such that for every $x\in B_{\sfrac{1}{2}}(y)$
\begin{equation*}
\label{eq:hypothesis-boundary}
\left(\omega_0x_d+ a\right)^{+} \leq u(x)\leq  \left(\omega_0x_d + b\right)^{+},
\end{equation*}    
    then there exist two reals numbers ${a}'\le {b}'$, with $a \leq {a}' < {b}' \leq b$ and
    $$
    |{b}'-{a}'|\leq (1-\sigma)\abs{b-a},
    $$
     such that for every $x \in B_\rho(y)$ it holds
    \begin{align}
    \label{eq:tesi_PH-boundary}
     \left(\omega_0x_d+  a'\right)^{+} \leq u(x)\leq  \left(\omega_0x_d+  b'\right)^{+}.
    \end{align}
\end{lemma}

\begin{proof}
 
We can suppose that $\abs{a}\leq 2 \rho$, otherwise the thesis holds true by interior Harnack inequality. 
We divide the proof in some steps.\\
\\
{\it Step 1. $y\in \{x_d=0\}$.} In this case, for simplicity, we take $y=0$.
We consider the following cases
\begin{itemize}
    \item[(i)] $a\le-\sfrac{\eps}{2}$. 
    Let $x_0 := \sfrac{e_d}{5}$ and $R := \sfrac{1}{5} + 3 \rho$. 
Similarly as done in \cref{lemma:PH_inequality}, choosing $\rho>0$ small enough and using \eqref{nuova4}, \eqref{e:ipotesi_PH_tensoreA}, \eqref{e:ipotesi_PH_vettoreV}, we can find a function $\psi: \R^d \to \R$ with the following properties
\begin{equation}
\label{eq:w-boundary}
    \quad \left\{
\begin{aligned}
    &\psi(x) = 1
    &&\text{in } \overline B_{\sfrac{1}{50}}(x_0),\\
    &\psi(x) = 0 &&\text{in } \R^d \setminus B_{R}(x_0),\\
    &\Delta \psi\ge c_0 
   &&\text{in } B_{R}(x_0) \setminus \overline B_{\sfrac{1}{50}}(x_0),\\
   &\nabla \psi\cdot W(0)\ge c_1&&\text{in }B_R(x_0)\setminus \overline B_{\sfrac{1}{50}}(x_0)\cap \{\omega_0x_d\le 2\rho\},
 \end{aligned}
\right.
\end{equation}
where $W(0)$ is defined in \eqref{e:W(0)_bordo} and $c_0,c_1>0$.

We set $P(x) := \omega_0x_d+a$ and let us suppose that $u(x_0) \geq P(x_0) + \frac{\eps}{2}$.
By Harnack inequality in $B_{\sfrac{1}{25}}(x_0)$, we get
\begin{equation}
    \label{eq:Harnack-boundary}
    u(x)-P(x) \geq c_{\mathcal{H}} \eps \qquad \hbox{in } B_{\sfrac{1}{50}}(x_0),
\end{equation}
where $c_{\mathcal{H}}>0$ is the Harnack dimensional constant. 
Taking the family of competitors $\{v_t\}_{t \in [0,1]}$ given by
\begin{equation*}
    \label{eq:v_t_bordo}
    v_t(x) = P(x) + c_{\mathcal{H}} \eps \psi(x) -c_{\mathcal{H}} \eps + c_{\mathcal{H}} \eps t,
\end{equation*}
we want to show that $ v_t(x)\le u(x) $ for all $t \in [0,1]$ and for all $x \in B_1$. 
Using the definition of $v_t$ and \eqref{eq:Harnack-boundary}, the inequality $ v_t(x)\le u(x)$ holds true for all $t \in [0,1]$, for all $x\in B_1 \setminus B_{R}(x_0)$ and for all $x\in \overline{B}_{\sfrac{1}{50}}(x_0)$. 
Moreover, since $a\le-\sfrac{\eps}{2}$, then, for every $t\in[0,1]$ and $x\in B_1\cap\{x_d\le0\}$, we have that 
$$v_t(x)=\omega_0x_d+a+c_{\mathcal{H}}\eps\psi(x)-c_{\mathcal{H}}\eps+c_{\mathcal{H}}\eps t\le-\frac\eps2+c_{\mathcal{H}}\eps\psi(x)<0= u(x).$$ 
Then, by contradiction, let us suppose that there exists $t \in [0,1]$ such that $v_t$ touches $u$ from above in $\overline{x}\in \left(B_R(x_0) \setminus \overline{B}_{\sfrac{1}{50}}(x_0)\right)\cap \{x_d>0\}$. 
Then, necessarily $\overline{x} \in \pa \Omega_u\cap B_1^+$. By \eqref{e:sviluppo_divergenza_eps} and \eqref{eq:w-boundary}, we have
\begin{equation*}
    {\rm div}(\tens{A}(\overline x)\nabla v_t(\overline x))={\rm div}(\tens{A}(\overline x)\nabla (\omega_0\overline x_d+c_{\mathcal{H}}\eps\psi(\overline x))=\eps\Delta \psi(\overline x)+o(\eps)>0.
\end{equation*}
Suppose that $\overline x\in\pa \Omega_u\subset B_R(x_0)\setminus \overline B_{\sfrac{1}{50}}(x_0)\cap \{\omega_0x_d\le 2\rho\}$, again by \eqref{e:sviluppo_condizione_al_bordo_eps} and \eqref{eq:w-boundary}, we get that
$$\tens{A}(\overline x)\nabla v_t(\overline x)\cdot\nabla v_t(\overline x)-\tens{A}(\overline x)\nabla v_t(\overline x)\cdot V(\overline x)=\eps c_{\mathcal{H}}\nabla \psi(\overline x)\cdot W(0)+o(\eps)>0.$$
This is a contradiction, namely there is no $\overline{x}$ touching point from below. Then, for all $x \in B_1$ and taking $t = 1$, we get
$$
u(x) \geq v_1(x) \qquad \Longrightarrow\qquad u(x) \geq P(x) + c_{\mathcal{H}} \eps \psi(x).
$$
Since $\psi$ is strictly positive in $B_{\rho}$, we have
$$
u(x) \geq P(x) + c_d \eps,
$$
which is exactly the left inequality taking ${a}' := a + c_d \eps$ in \eqref{eq:tesi_PH-boundary} and $c_d$ is a dimensional constant. In the case $u(x_0)\le P(x_0)+\frac{\eps}{2}$, applying a similar argument as the one above, we obtain the upper bound in \eqref{eq:tesi_PH-boundary}. We point out that in the proof of  the upper bound, it is not necessarily to require the additional hypothesis $a\le-\sfrac{\eps}{2}$, since competitors can touch even in $\{x_d=0\}$.
    \item[(ii)] $a\ge-\sfrac{\eps}{2}$. If $u>0$ in $B_1^+$, then the thesis follows by the interior Harnack inequality. If this is not the case, by \cref{lemma:useful_non_piace_Giulia}, we get 
    $$u(x)\le \left((\omega_0+C_0\eps)x_d-\frac{\eps}{2}\right)^+,$$
    where $C_0>0$ is the constant in \cref{lemma:useful_non_piace_Giulia}. 
    Indeed, if by contradiction $u(x)> \left((\omega_0+C_0\eps)x_d-\sfrac{\eps}{2}\right)^+$, then $u>0$ in $B_1^+$, since $\omega_0+C_0\eps\ge \omega_0+C_0\sfrac{\eps}{2}$. 
   By choosing $\rho>0$ small enough such that $C_0\rho\le\sfrac{3}{4}$, for every $x\in B_\rho$, we have that 
   $$ u(x)\le \left(\omega_0x_d+C_0\eps|x_d|-\frac{\eps}{2}\right)^+\le\left(\omega_0x_d+\frac{3}{4}\eps-\frac{\eps}{2}\right)^+ =\left(\omega_0x_d+\frac{\eps}{4}\right)^+.$$ 
  Letting $b':=\frac{1}{4}\eps$ and $a':=a$, and using that $a\ge-\sfrac{\eps}{2}$, we have 
  $$b'-a'=\frac{1}{4}\eps-a\le \frac{3}{4}\eps=\frac{3}{4}(b-a),$$ 
  concluding the proof.
\end{itemize}
\medskip
{\it Step 2. $y\in\{x_d>0\}$.} Suppose that $\text{dist}(y,\{x_d=0\})\ge \sfrac{\rho}{2}$, then the function $u_{y,\sfrac{\rho}{2}}$ is a solution of the interior problem \eqref{interior}. Thus, the thesis follows by the rate of convergence stated in \cref{remark:interior}.
On the other hand, if $\text{dist}(y,\{x_d=0\})\le \sfrac{\rho}{2}$, let $z\in \{x_d=0\}$ be the projection of $y$ in $\{x_d=0\}.$ Since $B_{\sfrac{\rho}{2}}(y)\subset B_\rho(z)$, we apply Step 1 in $z\in\{x_d=0\}$ and we  conclude.
\end{proof}

\subsubsection{Proof of the improvement of flatness}\label{subsec:proof-of-boundary-iof}
Now we are in position to prove the main result of this section, namely the improvement of flatness for branching points stated in \cref{lemma:IOF_bordo}.

\begin{proof}[Proof of \cref{lemma:IOF_bordo}]
 We show such a result by compactness.
 For all $j \in \N$, let $u_j$ be a sequence of solutions of \eqref{e:u_viscosity_sol_riscalato} with corresponding matrix $\tens{A}_j$ and vector $V_j$, and $\omega_0^j$ as in \eqref{e:bound_omega_x}, with $0\in \partial \Omega_{u_j}$ be a branching point. Suppose that $\tens{A}_j(x)e_d\cdot V_j(x)\ge \delta$ for all $x\in B_1$, and \eqref{e:ipotesi_PH_tensoreA}, \eqref{e:ipotesi_PH_vettoreV} hold true. We also suppose that $u_j$ are $\eps_j$-flat along the direction $e_d$ in $B_1$, namely for all $x \in B_1$
    $$\left(\omega_0^jx_d - \eps_j\right)^+\le u_j(x)\le  \left(\omega_0^jx_d + \eps_j\right)^+$$
    where $\eps_j\to0^+$.
    We consider the following functions
    \begin{equation}
        \label{eq:u_j_bordo}
      \widetilde{u}_j:
      \left\{
      \begin{aligned}
          \overline\Omega_{u_j}&\cap B_1\to\R\\
          x&\mapsto\widetilde{u}_j(x) = \frac{u_j(x) - \omega_0^jx_d }{\eps_j}.
      \end{aligned}
      \right. 
    \end{equation}
    Iterating the Partial Harnack inequality in \cref{lemma:PH_inequality_bordo} similarly as done \cref{lemma:compattezza}, we have that there exists a H\"older continuous function $\widetilde{u}: \overline {B_{\sfrac{1}{2}}^+} \to \R$ such that, for all $\eta>0$
    $$
    \widetilde{u}_j\to
    \widetilde{u}\qquad\text{uniformly in}\quad B_{\sfrac{1}{2}}\cap \{x_d\geq \eta\},
    $$
    Moreover, the sequence of graphs of $\widetilde u_j $ converges in the Hausdorff distance to $\widetilde u$, namely
    $$
    \mathcal{G}_j = \left\{(x, \widetilde{u}_j(x)): x \in B_{\sfrac{1}{2}} \cap \overline{\Omega}_{u_j} \right\} \qquad \xrightarrow{H} \qquad \mathcal{G}= \left\{(x, \widetilde{u}(x)): x \in \overline{ B_{\sfrac{1}{2}}^+} \right\}.
    $$
We need to show that $-\widetilde u$ solves (in a suitable sense, see \cref{def:def-sol_sezione_esistenza-oblique-thin} below)  
 \begin{equation}
 \label{e:thin_widetileu}
        \left\{
\begin{aligned}
    & \Delta (-\widetilde u) = 0 && \hbox{ in } B_{\sfrac{1}{2}}^+,\\
    & \nabla (-\widetilde u)\cdot W=0 && \hbox{ on } B_{\sfrac{1}{2}}'\cap\Omega_{-\widetilde u},\\
    & \nabla (-\widetilde u)\cdot W\le 0 && \hbox{ on } B_{\sfrac{1}{2}}',\\
    &-\widetilde u\ge0 && \hbox{ on } B_{\sfrac{1}{2}}',\\
\end{aligned}
        \right.
    \end{equation}
which is the oblique thin obstacle problem. 
   Indeed, the fact that $-\widetilde u\ge0$ on $B_{\sfrac{1}{2}}'$ follows by definition of $\widetilde u_j$.
   Moreover, let $\varphi\in C^\infty(\overline{B_{\sfrac{1}{2}}^+})$ be a test function touching $-\widetilde u$ from above at $x_0\in \overline {B_{\sfrac{1}{2}}^+}$. Similarly as done in \cref{lemma:IOF}, there exists a sequence of points $x_j\in \overline\Omega_{u_j}\cap \overline {B_{\sfrac{1}{2}}^+}$ such that $ \varphi$ touches $-\widetilde u_j$ from above at $x_j$.
    If $x_0\in B_{\sfrac{1}{2}}^+$, then $x_j\in \Omega_{u_j}\cap B_{\sfrac{1}{2}}^+$ and so, by \cref{lemma:utile}, we have that
    $$
    0\ge {\rm div}\left(\tens{A}_j(x) \nabla (\omega_0^j(x_j)_d-\eps_j\varphi(x_j))\right)= -\eps_j\Delta \varphi(x_j)+o(\eps_j).$$
    Passing to the limit we get 
    $$\Delta\varphi(x_0)\ge0.$$
    To get the boundary condition satisfied by $-\widetilde u$, suppose that $x_0\in B_{\sfrac{1}{2}}'$ with $-\widetilde u(x_0)>0$. Again, 
    there exists a sequence of points $x_j\in \partial\Omega_{u_j}\cap \overline {B_{\sfrac{1}{2}}^+}$ such that $ \varphi$ touches $-\widetilde u_j$ from above at $x_j$. Moreover, since $-\widetilde u(x_0)>0$, then $x_j\in \partial\Omega_{u_j}\cap B_{\sfrac{1}{2}}^+$.
   Then,
    we have that
    \begin{align*}
        \tens{A}_j(x_j)\nabla \left( \omega_0^j (x_j)_d -\eps_j \varphi(x_j) \right)\cdot \left( \omega_0^j (x_j)_d   -\eps_j \varphi(x_j)\right) &\le 
         \tens{A}_j(x_j)\nabla \left( \omega_0^j (x_j)_d-\eps_j \varphi(x_j)  \right)\cdot V_j(x_j).
        \end{align*}
        By \cref{lemma:utile}, we end up with
        $$-\nabla \varphi(x_j)\cdot W_j(0)+o(\eps_j)\le0,$$
       and passing to the limit, we get
        $$\nabla \varphi(x_0)\cdot W(0)\ge0.$$ 
        On the other hand, if $\varphi\in C^\infty(\overline{B_{\sfrac{1}{2}}^+})$ touches $-\widetilde u$, from below at $x_0\in \overline {B_{\sfrac{1}{2}}^+}$, then, similarly
        $$\Delta \varphi(x_0)\le0\quad\text{if }x_0\in B_{\sfrac{1}{2}}^+\qquad\text{and}\qquad\nabla \varphi(x_0)\cdot W(0)\le0\quad\text{if }x_0\in B_{\sfrac{1}{2}}'.$$
This implies that $-\widetilde{u}$ solves \eqref{e:thin_widetileu},
with
$$\norm{\widetilde{u}}_{L^\infty(B_{\sfrac{1}{2}})} \leq 1\qquad\text{and}\qquad \widetilde u(0)=0.$$
    By \cref{thm:oblthin2}, we can write the Taylor expansion of the solution. Precisely, since $\nabla_{x'}\widetilde u(0)=0$ and $\partial_d\widetilde u(0)\ge0$, we have that for some positive small constant $\eta$ and for all $ x \in\overline {B_{\eta}^+}$
    $$
    \begin{aligned}
        \abs{\widetilde{u}(x) - \partial_d\widetilde u(0) x_d } &\leq  C \abs{x}^{1+\alpha_0}\leq\frac{\eta^{1+\alpha}}{4}.
    \end{aligned}
    $$
    Thus, along the sequences and using the expression of $\widetilde{u}_j$ defined in \eqref{eq:u_j_bordo}, for every $x\in B_\eta\cap\overline\Omega_{u_j}$, we have
    \begin{equation*}
        \label{eq:finale1_bordo}
        \abs{u_j (x) - (\omega_0^j+\eps_j\partial_d\widetilde u(0))x_d  } \leq \eps_j \frac{\eta^{1+\alpha}}{2}.
    \end{equation*}
    By \cref{lemma:useful_non_piace_Giulia}, since $0\in \partial \Omega_{u_j}$ is a branching point for $u_j$, we have that
    $$\eps_j\partial_d\widetilde u(0)\le C_0\eps_j\frac{\eta^{1+\alpha}}{2}, \qquad \Longrightarrow \qquad  \partial_d\widetilde u(0)\le C_0\frac{\eta^{1+\alpha}}{2}.$$ 
Therefore, for every $x\in\overline{ B_{\eta}^+}\cap\overline\Omega_{u_j}$, 
we get
\begin{equation*}
        \abs{u_j (x) - \omega_0^jx_d  } \leq \eps_j \frac{\eta^{1+\alpha}}{2}+\eps_j\partial_d\widetilde u(0) |x_d|\le \eps_j\left(\frac{\eta^{1+\alpha}}{2}+C_0\frac{\eta^{2+\alpha}}{2}\right)\le \eps_j\eta^{1+\alpha},
    \end{equation*}
    getting the thesis.
\end{proof}

\subsection{The free boundary is a \texorpdfstring{$C^{1, \alpha}$}{alpha}-graph near contact points}
\label{subsec:final_bordo}

Iterating the improvement of flatness at branching points \cref{lemma:IOF_bordo}, we get the following rate of convergence.

\begin{lemma}[Uniqueness and rate of convergence to the blow-up limit]
    \label{lemma:proprieta_blowup_boundary}
   Let $u:B_1\to\R^+$ be a viscosity solution of \eqref{e:u_viscosity_sol_riscalato}.
   For all $\delta>0$ and $\alpha\in(0,\alpha_0)$, where $\alpha_0$ is defined in \cref{thm:oblthin2}, there are constants $\e_0>0$ and $C>0$ such that the following holds.
  Suppose that $\tens{A}(x)e_d\cdot V(x)\ge 2\delta$ for all $x\in B_1$, and \eqref{eq:diff-v-v0} holds true.
We also suppose that $u$ is $\eps_0$-flat along the direction $e_d$, namely for every $x\in B_1$
    $$\left(\omega_{0} x_d - \eps_0\right)^+\le u(x)\le  \left(\omega_{0} x_d+ \eps_0\right)^+.
    $$
    Then, for every branching point $x_0\in \partial\Omega_u\cap B'_{\sfrac{1}{2}}$ and for every $r\in (0,\sfrac{1}{2})$,
    \begin{equation*}
        \label{e:tesi_iterazione_bordo}
        \norm{u-\omega_{x_0}x_d}_{L^\infty(B_r(x_0)\cap\overline\Omega_u)}\le C\eps_0 r^{1+\alpha},
    \end{equation*} 
    where $\omega_{x_0}$ is defined in \eqref{e:bound_omega_x}.
\end{lemma}
\begin{proof}
    The proof is the same of \cref{lemma:proprieta_blowup}, once we know that the hypotheses \eqref{nuova4}, \eqref{e:ipotesi_PH_tensoreA} and \eqref{e:ipotesi_PH_vettoreV} hold true for the sequence of rescalings $u_{x_0,\sfrac{\eta^j}{2}}$ at every branching point $x_0\in\partial\Omega_u\cap B_{\sfrac{1}{2}}'$. Let us fix $x_0\in\partial\Omega_u\cap B_{\sfrac{1}{2}}'$ to be branching point. We divide the proof into two steps.
    \\
    \\
    \textit{Step 1. $\tens{A}(x_0)= \tens{I}$.}
     Suppose that the sequence of rescaling $u_{x_0,r_j}$, with $r_j:=\frac{\eta^j}{2}$, is a $\eps_j:=\eps_0\eta^{j \alpha}$-flat solution of \eqref{e:u_viscosity_sol_riscalato} with matrix $\tens{A}_j(x):=\tens{A}(x_0+ r_jx)$ and vector $V_j(x):=V(x_0+r_jx)$. We observe that \eqref{nuova4} is immediate since $\tens{A}_j(x)e_d\cdot V_j(x)\ge2\delta$. To prove \eqref{e:ipotesi_PH_tensoreA} and \eqref{e:ipotesi_PH_vettoreV}, since $\tens{A}$ and $V$ are regular and $\tens{A}(x_0)=\tens{I}$, we have that
     $$\|\tens{A}_j(x)-\tens{I}\|_{L^\infty(B_1)}=\|\tens{A}(x)-\tens{I}\|_{L^\infty(B_{r_j}(x_0))}\le Cr_j=C\frac{\eta^j}2\le \eps_0\eta^{j\alpha_0}=\eps_j^{\frac{\alpha_0}{\alpha}}$$
     and
     $$\|V_j(x)-V_j(0)\|_{L^\infty(B_1)}=\|V(x)-V(x_0)\|_{L^\infty(B_{r_j}(x_0))}\le Cr_j=C\frac{\eta^j}2\le \eps_0\eta^{j\alpha_0}=\eps_j^{\frac{\alpha_0}{\alpha}},$$
    providing $\eta$ small, which concludes Step 1.
    \\ 
    \\
    \textit{Step 2. $\tens{A}(x_0)\neq \tens{I}$.} 
    We can change the coordinates considering the function $v(y)=u(x_0+\tens{M}_{x_0}^{-1}y)$ with $\tens{M}_{x_0}\in\R^{d\times d}$ such that $\tens{M}_{x_0}\tens{A}(x_0)\tens{M}_{x_0}^T=\tens{I}$. Then the function $v$ solves in 
    \eqref{e:u_viscosity_sol_riscalato} in $B_{\sfrac{1}{2}}$, 
    with matrix $\tens{A}_{x_0}(x):=\tens{M}_{x_0}\tens{A}(x_0+\tens{M}_{x_0}^{-1}x)\tens{M}_{x_0}^T$and vector $V_{x_0}(x):=\tens{M}_{x_0}^{-T}V(x_0+\tens{M}_{x_0}^{-1}x)$. In particular $\tens{A}_{x_0}(0)=\tens{I}$. Thus, $v$ is satisfies the hypothesis of Step 1, yielding the thesis.
\end{proof}

The main result of this section is the following proposition.

\begin{proposition}\label{prop:grafico}
    Let $u:B_1\to \R^+$ be a viscosity solution of \eqref{e:u_viscosity_sol_riscalato} and let $x_0 \in \partial \Omega_u \cap B_1'$ be a contact point. For all $\delta>0$, there is a constant $\rho = \rho(x_0,\delta)>0$ such that if for all $x\in B_1$, we have 
    \be\label{stima-basso-nuova}\tens{A}(x)e_d\cdot V(x)\ge 2\delta,\ee 
    and \eqref{eq:diff-v-v0} holds true, then $\partial \Omega_u\cap B_\rho(x_0)$ is a $C^{1, \alpha}$ graph for every $\alpha\in(0,\alpha_0).$ 
\end{proposition}
\begin{proof}
    Without loss of generality $x_0=0$. We notice that if $0$ is not a branching point, then
    $\partial\Omega_u\cap B_\rho=B_\rho'$, concluding the proof. Thus, we consider $0$ to be a branching point.
    By \ref{1} of \cref{prop:punti_regolari_al_bordo}, $u_r\to \omega_{0}x_d^+ $, then, up to rescale $u$, we have that, for every $x\in B_1$
    $$\left(\omega_{0} x_d - \eps_0\right)^+\le u(x)\le  \left(\omega_{0} x_d+ \eps_0\right)^+,$$ 
    where $\eps_0>0$ is as in \cref{lemma:proprieta_blowup_boundary}.

      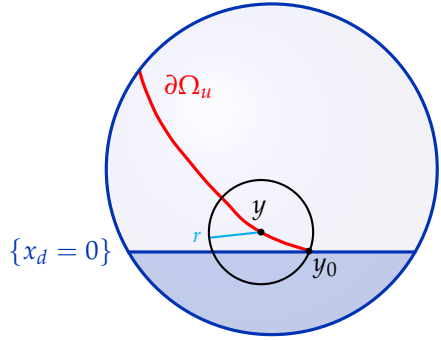
\begin{figure}[htbp]
    	\centering
    	\begin{tikzpicture}[rotate=90, scale= 1.1]
    	\coordinate (O) at (0,0);
    	\shade[ball color = blue, opacity = 0.05] (0,0) circle [radius = 2cm];
    	
    	%BLUE
    
    	\begin{scope}[transparency group,opacity=0.2]
    	\draw[draw=none, fill=french] (-1,0) -- +(90:1.73cm) arc (94:266:1.73cm);
    	%\draw[draw=none, fill=blue] (0,0) -- +(94:1.5cm) arc (94:-86:1.5cm);
    	\end{scope}
    	\draw[draw=none,name path = uno] (0,0) circle [radius = 2cm];
    	\draw[draw=none, name path = due] (0,0) circle [radius = 2.62cm];
    	\tikzfillbetween[of=uno and due] {color=white};
    	%	\draw[thick, red, very thick, name path=a] (0.9,1.78) arc [start angle=130, end angle = 151,x radius = 110mm, y radius = 110mm];
    %	\draw [very thick, color=red, name = due] plot [smooth, tension=0.9] coordinates {(-0.97,-0.45) (-0.75,0.15)(-0.5, 0.45)	(-0.3,0.6) (0.1,0.8) (0.5,1.3) (0.8,1.5) (1.2,1.6) };
    		\draw [very thick, color=red, name = due] plot [smooth, tension=0.9] coordinates {(-0.99,-0.48) (-0.75,0.15)
    		(-0.35,0.6) (0.1,1) (0.5,1.3) (0.9,1.5) (1.2,1.6) };
    	%\draw [very thick, color=red, name = tre] plot [smooth, tension=0.9] coordinates {(-1,-0.8) (-0.8,-0.65) (-0.64,-0.45)(-0.5, -0.35)	(-0.3,-0.2) (0.1,0) (0.5,0.5) (1.1,0.9) (1.29,1.53)};
    	\draw[very thick, color=french] (-1,-1.745) -- (-1,1.745);
    	\draw[very thick, color = french] (O) circle [radius=2cm];
    	
    	\draw[thick, color=cyan] (-0.76,0.13) -- (-0.83,0.77);
    	
    	%\draw[thick]  (-0.99,-0.48)  circle [radius=0.65cm];
        \draw[thick]  (-0.76,0.13)  circle [radius=0.63cm];
        \filldraw (-0.76,0.13) circle (1pt);
   	\filldraw (-0.99,-0.45) circle (1pt);
    	
    	\draw node at (-0.5,0.15) {$y$};
    %	\draw node at (-1.2,-0.85) {$O$};
    	\draw node at (-1.2,-0.65) {$y_0$};
    	
    	\draw node [ color=french] at (-1,2.55) {$\{x_d = 0\}$};
    	\draw node [color=red] at (1,1) {$\partial \Omega_u$};
    	\draw node [color=cyan] at (-0.82,0.9) {\footnotesize{$r$}};
    	\end{tikzpicture}
        
    	\caption{Graphical representation of the geometrical construction employed in \cref{prop:grafico} to show the $C^{1,\alpha}$ regularity around branching points.}
    	\label{fig:eps_reg_bordo}
    \end{figure}
    
    Let $\rho>0$, $y\in \partial\Omega_u\cap B_\rho^+$
    and $y_0$ be the projection of $y$ on the set of branching points, namely $r:=|y-y_0|\le |y-z_0|$ for every $z_0$ branching point. In particular, $r \leq \abs{y} \leq \rho$ and in the following, we assume that $\rho$ is small enough (for a geometrical representation see \cref{fig:eps_reg_bordo}).
    Since $u$ satisfies the hypothesis of \cref{lemma:proprieta_blowup_boundary}, then for every $x\in B_{2r}(y_0)\cap\overline\Omega_u$ we have
    $$u(x)\ge \omega_{y_0}x_d -Cr^{1+\alpha}.$$
    In particular, by \eqref{eq:diff-v-v0} and \eqref{stima-basso-nuova}, we have 
    $$\omega_{y_0}\ge \frac{2\delta}{\tens{A}(y_0)e_d\cdot e_d}\ge \delta.$$
    By taking $ C_\delta:=\frac{C}{\delta}$, then for every $x\in B_{2r}(y_0)\cap \overline\Omega_u\cap\left\{x_d>C_\delta r^{1+\alpha}\right\}$, we have $u(x)>0$. Thus $y_d\le C_\delta r^{1+\alpha}$. We divide the rest of the proof into some steps.\\
    \\
    {\it Step 1. In $B_r'(y)$ there are only vanishing points.}
According to the notation introduced in \cref{rem:caratterizzazione_punti} and since $y_0$ is the projection of $y$ on the set of branching points, $B_r'(y)$ contains only interior contact points or vanishing points. 
Since, both the set of interior contact points and the set of vanishing points are open sets in the topology of $\{x_d=0\}$, it follows that $B_r'(y)$ contains either only interior contact points or only vanishing points.
To conclude, we need to show that the first possibility does not occur.
By \cref{lemma:rimango_in_parte_positiva} and $u(y)=0$, we know that there exists a curve $\gamma$, with $\gamma'(t)= \tens{A}(\gamma(t))V(\gamma(t))$ and $\gamma(0)=y$ such that $u(\gamma(t))=0$ for every $t<0$ with $\gamma_d(t):=\gamma(t)\cdot e_d\ge0$. Then, by \eqref{stima-basso-nuova}, $\gamma'_d(t)= \tens{A}(\gamma(t))V(\gamma(t))\cdot e_d>0$, implying that $t\mapsto\gamma_d(t)$ is non decreasing in $t$.
Thus, there exists $\overline t<0$ such that $\gamma_d(\overline t)=0$ and we call $z:=\gamma(\overline t)$.
Hence, using \eqref{eq:diff-v-v0}, we end up with
\bea C_{\delta}r^{1+\alpha}&\ge y_d=\gamma_d(0)-\gamma_d(\overline t)=\int_{\overline t}^0\left(\tens{A}(\gamma(s)) V(\gamma(s))\cdot e_d-V_d(0)\right)\,ds-V_d(0)\overline t\\&\ge \Pi_0 \overline t-V_d(0)\overline t\ge -\frac{V_d(0)}{2}\overline t=\frac{V_d(0)}{2}|\overline t|,\eea
since $\Pi_0>0$ chosen in \eqref{eq:diff-v-v0} is small enough. Then, by \eqref{eq:diff-v-v0} and \eqref{stima-basso-nuova}, there is a constant $C>0$ such that 
$$|z-y|=|\gamma(\overline t)-\gamma(0)|\le C|\overline t|\le Cr^{1+\alpha}.$$ Hence $z\in B_r'(y)$.
Moreover, since $u(\gamma(t))=0$ for every $t\in[\overline t,0]$, then, by \cref{prop:punti_regolari_al_bordo}, $z$ cannot be an interior contact point, yielding that $z$ is a vanishing point.
\\ 
\\
{\it Step 2. Near branching points, all points are regular.} Since $B_r'(y)$ contains only vanishing points, by defining
$$\widetilde u(x):=\left\{
\begin{aligned}
&u(x)&&\text{in }B_r(y)^+,\\
    &0 &&\text{in }B_r(y)\cap\{x_d\le0\},
\end{aligned}
\right.$$
we have that $\widetilde u$ is a viscosity solution of \eqref{e:u_viscosity_sol_riscalato} replacing $B_1$ with $B_r(y)$.
Since $u$ satisfies the hypothesis of \cref{lemma:proprieta_blowup_boundary}, then
$$\|u-\omega_{y_0}x_d\|_{L^\infty\left(B_{2r}(y_0)\cap\overline\Omega_u\right)}\le Cr^{1+\alpha},$$ 
implying that
$$\|u-\omega_{y_0}x_d\|_{L^\infty\left(B_{r}(y)\cap\overline\Omega_u\right)}\le Cr^{1+\alpha}.$$ 
Thus, by rescaling, we have that
\be\label{eq:stimafinale1_boundary_reg}
\|u_{y,r}-\omega_{y_0}x_d\|_{L^\infty\left(B_1\cap\overline\Omega_u\right)}\le Cr^{\alpha}.\ee 
This means that $u_{y,r}$, and thus $\widetilde u_{y,r}$, is $Cr^{\alpha}$-flat along the direction $e_d$ in $B_1$. 
By \cref{remark:interior}, there exist $\widetilde\omega_{y,\nu_y}\in\R$ and $\nu_y\in \pa B_1$ such that 
\be\label{e:stima_finale2_boundary_reg}
\|u_{y,r}-x\cdot \widetilde\omega_{y,\nu_{y}}\nu_{y}\|_{L^\infty\left(B_1\cap\overline\Omega_u\right)}\le Cr^{\alpha},
\ee 
implying that $y$ is a regular point.
\\
\\
\textit{Step 3. $\partial\Omega_u$ is a $C^{1,\alpha}$ graph.}
Combining \eqref{eq:stimafinale1_boundary_reg} and \eqref{e:stima_finale2_boundary_reg}, and arguing as in \cite[Lemma 8.8]{velichkov2023regularity},
we obtain
$$|\widetilde\omega_{y,\nu_{y}}\nu_y-\omega_{y_0}e_d|\le Cr^\alpha.$$
Finally, using the fact that $\tens{A}$ is smooth and the definition of $\omega_{x}$ as in \eqref{e:bound_omega_x}, we have
$$|\widetilde\omega_{y,\nu_{y}}\nu_y-\omega_0e_d|\le|\widetilde\omega_{y,\nu_{y}}\nu_y-\omega_{y_0}e_d|+|\omega_{y_0}-\omega_0|\le Cr^{\alpha}+|y_0|^\alpha\le 2Cr^{\alpha}+|y|^\alpha\le (2C+1)|y|^\alpha,$$
which implies the $C^{1,\alpha}$ regularity of the graph of the free boundary $\partial\Omega_u$ around $0$ and concluding the proof.
\end{proof}
\begin{proof}[Proof of \cref{th:eps_reg_bordo}]
The $C^{1, \alpha}$ regularity of the free boundary $\partial \Omega_u$ is a consequence of the change of variable provided in \cref{sec:change_coordinates} and then by \cref{prop:grafico}. 
\end{proof}

  \section{\texorpdfstring{$C^{1,\alpha}$}{C1alpha} regularity for the oblique thin obstacle problem}
  \label{susec:reg_linarized_proble_bordo}

   In this section we establish the $C^{1,\alpha}$ regularity for viscosity solutions of the oblique thin obstacle problem. Precisely, we show the following theorem.

\begin{theorem}
    \label{lemma:regolaita_prob_lin-thin}
    For all $\delta >0$, there are constants $C>0$ and $\alpha_0 \in (0,1)$ such that the following holds. Suppose that for some unit vector $W\in\partial B_1$, we have $W_d:=W\cdot e_d\ge\delta$.
  Let $v:\overline {B_1^+}\to\R$ be a viscosity solution of the oblique thin obstacle problem
    \begin{equation}\label{linearizzato-bordo}
        \left\{
\begin{aligned}
    & \Delta v = 0 && \hbox{ in } B_{1}^+,\\
    & \nabla v\cdot W=0 && \hbox{ on } B_{1}'\cap\Omega_v,\\
    & \nabla v\cdot W\le 0 && \hbox{ on } B_{1}',\\
    &v\ge0 && \hbox{ on } B_{1}'.\\
\end{aligned}
        \right.
    \end{equation} 
     Then, $v\in C^{1,\alpha_0}(\overline{B_{\sfrac{1}{2}}^+})$, and the following estimate holds true 
     \bea\label{thesis:oblthin}\|v\|_{C^{1,\alpha_0}(\overline{B_{\sfrac{1}{2}}^+})}\le C\|v\|_{L^\infty(B_1^+)}.\eea
\end{theorem}
We remark that \cref{thm:oblthin2} is an immediate consequence of \cref{lemma:regolaita_prob_lin-thin}.
The strategy to prove \cref{lemma:regolaita_prob_lin-thin} is similar to the original proof of Caffarelli in \cite{caf79}, where the Author proved the $C^{1,\alpha}$ regularity for the classical thin obstacle problem, namely $W = e_d$. The proof of \cref{lemma:regolaita_prob_lin-thin} is divided into four steps split into four subsections:
\begin{itemize}
    \item in \cref{subsec:preliminare}, we recall some definitions of viscosity solutions of \eqref{linearizzato-bordo} and we prove a maximum principle;
    \item in \cref{lipandsemiconvest}, we show that the function $v$ is Lipschitz and semiconvex in the tangential directions;
    \item in \cref{estimatefinal0}, we define the function $\sigma_W$, which is the derivative along the oblique direction $W$ on the hyperplane $\{x_d =0\}$;
    \item in \cref{estimatefinal}, we deduce the $C^{0,\alpha}$ estimates for the function $\sigma_{W}$, which implies the $C^{1,\alpha}$ regularity for $v$ up to the boundary $\{x_d=0\}$.
\end{itemize}

\subsection{Preliminary results}
\label{subsec:preliminare}

First we recall the definition of viscosity solutions for the oblique thin obstacle problem.
   \begin{definition}[Viscosity solutions of \eqref{linearizzato-bordo}]
\label{def:def-sol_sezione_esistenza-oblique-thin}
Let $v:\overline {B_1^+}\to\R$ be a continuous function. We say that {\em $v$ is a viscosity solution} of the oblique thin obstacle problem \eqref{linearizzato-bordo} if $v\ge0$ on $B_1'$ and for every $x_0 \in B_1^+\cup B_1'$ and $\varphi \in C^{\infty}(\overline{B_1^+})$, we have
\begin{itemize}
    \item if $x_0 \in B_1^+$ and $\varphi$ touches $v$ from below at $x_0$, then $\Delta \varphi(x_0) \leq 0$;
    \item if $x_0\in B_1^+$ and $\varphi$ touches $v$ from above at $x_0$, then $\Delta \varphi(x_0) \geq 0$;
    \item if $x_0\in B_1'$ and $\varphi$ touches $v$ from below at $x_0$, then $\nabla \varphi(x_0)\cdot W\le0 $;
    \item if $x_0\in B_1'\cap\Omega_v$ and $\varphi$ touches $v$ from above at $x_0$, then $\nabla \varphi(x_0)\cdot W\ge0 $.
\end{itemize}
Without loss of generality, we suppose that a solution of \eqref{linearizzato-bordo} is defined in the entire $B_1$. 
Indeed, one can extend it in an even way in the $e_d$ direction by setting $$v(x',x_d)=v(x',-x_d).$$
According to the classical notations for the thin obstacle problem, we define the contact set of $v$ as
$$\Lambda(v):=\{v=0\}\cap B_{1}'.$$
\end{definition}

\noindent We show the following maximum principle.

\begin{lemma}[Maximum principle]\label{lemma:maximum-principle-thin}
    Let $W\in\partial B_1$ be a unit vector such that $W_d >0$.
    Let $A\subset B_1'$ be an open set and let $v_1,v_2:\overline {B_1^+}\to \R$ be viscosity solutions of  \begin{equation*}
        \left\{
\begin{aligned}
    & \Delta v_1 = 0 && \hbox{ in } B_{1}^+,\\
    & \nabla v_1\cdot W\ge0 && \hbox{ on } A,\\
\end{aligned}
        \right.\qquad \text{and}\qquad\left\{
\begin{aligned}
    & \Delta v_2 = 0 && \hbox{ in } B_{1}^+,\\
    & \nabla v_2\cdot W\le0 && \hbox{ on } A.\\
\end{aligned}
        \right.
    \end{equation*} 
    Suppose that $v_1\le v_2$ on $\partial (B_1^+ \cup A)$, 
    then $v_1\le v_2$ in $\overline{B_1^+}$. In particular, the maximum of $v_1$ in $B_1^+\cup A$ is attained on the boundary $\partial( B_1^+\cup A)$.
\end{lemma}
\begin{proof}
    Let $\chi_i\in C^\infty(B_1^+\cup A)$, with $i=1,2$, the functions defined as \begin{equation*}
        \left\{
\begin{aligned}
    & \Delta \chi_i = 0 && \hbox{ in } B_{1}^+,\\
    & \nabla \chi_i\cdot W=0 && \hbox{ on } A,\\
    & \chi_i=v_i && \hbox{ on } \partial (B_1^+ \cup A),\\
\end{aligned}
        \right.
    \end{equation*} We prove that $v_1\le \chi_1\le\chi_2\le v_2$ in $\overline {B_1^+}$. Since proofs are similar, we only show that $v_1\le \chi_1$. 
    For every $\eps >0$, let us define
    $$\chi^\eps_1:=\chi_1-\eps x_d-\eps x_d^2+2\eps,$$
    then $v_1\le \chi^\eps_1$ on $\partial (B_1^+ \cup A)$ since $\chi_1=v_1$ on $\partial (B_1^+ \cup A)$.
    Let us consider the function $\chi^\eps_1+t$ for $t>0$ large, and let us decrease $t\ge0$ until a contact point $x_0$ occurs in $\overline {B_1^+}$. 
    Since $\Delta \chi^\eps_1=-2\eps<0$ and $\nabla \chi^\eps_1\cdot W=-\eps W_d<0$, then $x_0\not \in B_1^+\cup A$. Thus, we can choose $t=0$ getting that $v_1\le\chi^\eps_1$. 
    By sending $\eps\to0^+$, we get the thesis.

    Finally, we observe that to show that the maximum of $v_1$ in $B_1^+\cup A$ is attained on the boundary $\partial( B_1^+\cup A)$, it is enough to choose $v_2:=\max_{\partial( B_1^+\cup A)} v_1$.
\end{proof}

\subsection{Lipschitz and semiconvex estimates}\label{lipandsemiconvest}

The next step to prove \cref{lemma:regolaita_prob_lin-thin} is to show that a solution of the oblique thin obstacle problem is locally Lipschitz and semiconvex in the tangential directions. We split these two results into the following two lemmas.

\begin{lemma}
    [Lipschitz continuity]\label{lemma:semiconv0}
    Let $v:{B_1}\to\R$ be a viscosity solution of \eqref{linearizzato-bordo}. Then, there is a constant $ C= C(\delta)>0$ such that
    $v$ is locally Lipschitz, namely
    $$\|\nabla v\|_{L^\infty(B_{\sfrac{1}{2}})}\le  C\|v\|_{L^\infty(B_1)}.$$
\end{lemma}

\begin{lemma}
    [Tangential semiconvexity]\label{lemma:semiconv1}
    Let $v:{B_1}\to\R$ be a viscosity solution of \eqref{linearizzato-bordo}. Then, there is a constant $ C=C(\delta)>0$ such that
    \begin{itemize}
        \item[i)] $v$ is semiconvex in the tangential directions, precisely, for every $e=(e',0)\in \mathbb{S}^{d-1}$, we have
        $$\inf_{B_{\sfrac{1}{2}}}\partial_{ee}v \ge - C\|v\|_{L^\infty(B_1)};$$
        \item[ii)] $v$ is semiconcave in the normal direction
        $$\sup_{B_{\sfrac{1}{2}}^+}\partial_{dd}v \le  C\|v\|_{L^\infty(B_1)}$$ and, if we denote by $\widetilde W:=(-W',W_d)$ the adjoint direction of $W$, then $$\sup_{B_{\sfrac{1}{2}}^+}\partial_{\widetilde WW}v\le C\|v\|_{L^\infty(B_1)}.$$ 
    \end{itemize}
\end{lemma}

    \begin{proof}[Proof of \cref{lemma:semiconv0}]     
    Without loss of generality, we suppose that $\|v\|_{L^\infty(B_1)}\le 1$.  
    Consider the function $h:B_1\to\R$ defined as
\begin{equation*}
        \left\{
\begin{aligned}
    & \Delta h = 0 && \hbox{ in } B_{{1}}^+\cup B_1^-,\\
    & h=0 && \hbox{ on } B_{{1}}',\\
    & h=-\|v\|_{L^\infty(B_1)} && \hbox{ on } \partial B_1.
\end{aligned}
        \right.
    \end{equation*}
Then, by boundary estimates for harmonic functions, $h$ is Lipschitz continuous in $B_{\sfrac{3}{4}}$. Thus, for every $x_0\in B_{\sfrac{1}{2}}'$ and $r\in(0,\sfrac{1}{4})$, we have
$$\sup_{B_r(x_0)}|h(x)|\le Cr,$$
for some $C>0$. 

Let $x_0\in B_{\sfrac{1}{2}}'\cap \Lambda(v)$ and $r\in (0,\sfrac{1}{4})$. By the maximum principle for harmonic functions, $v\ge h$ on $B_{\sfrac{3}{4}}$, then $$v\ge h\ge -Cr\quad\text{in }B_r(x_0).$$ Take the function $q:=v+Cr\ge 0$ in $B_r(x_0)$ and we define $\overline q$ such that 
\begin{equation}\label{e:qbar}
        \left\{
\begin{aligned}
    & \Delta \overline q = 0 && \hbox{ in } B_{{r}}^+(x_0)\cup B_{{r}}^-(x_0),\\
    & \nabla \overline q\cdot W = 0 && \hbox{ in } B_{{r}}'(x_0),\\
    & \overline q=q && \hbox{ on } \partial B_{{r}}(x_0).\\
\end{aligned}
        \right.
    \end{equation}
        Then, by maximum principle \cref{lemma:maximum-principle-thin}, we have that $0\le \overline q\le q$ in $B_r(x_0)$ since $q$ is a supersolution of \eqref{e:qbar}.
        It follows that $q\le \overline q+Cr$ on $\partial B_r(x_0)$ and $q\le \overline q+Cr$ on $B_r'(x_0)\cap\Lambda(v)$. 
        Since $\Delta q=0$ in $B_r^+(x_0)$ and $\nabla q\cdot W=0$ on $B_r'(x_0)\cap\Omega_v$, then by maximum principle \cref{lemma:maximum-principle-thin}, we have that 
        $$q\le \overline q+Cr\quad\text{in }B_r(x_0).$$
        Since $\overline q$ is Lipschitz continuous in $B_{\sfrac{r}{2}}(x_0)$ and $q(x_0)=Cr$, for every $x\in B_{\sfrac{r}{2}}(x_0)$, we obtain 
        $$\overline q(x)\le |\overline q(x)-\overline q(x_0)|+\overline q(x_0)\le L\|q\|_{L^\infty(B_r(x_0))}r+q(x_0)\le L\|v\|_{L^\infty(B_1)}r+LCr+Cr \leq \left(L+ LC + C\right)r.$$ 
        Therefore, 
        for every $x\in B_{\sfrac{r}{2}}(x_0)$ we have that
        $$v(x)\le q(x)\le \overline q(x)+Cr\le C_0 r.$$
        Thus, we have proved that for every $x_0\in B_{\sfrac{1}{2}}'\cap \Lambda(v)$ and $r\in (0,\sfrac{1}{8})$, we obtain
        $$\sup_{B_{r}(x_0)}|v|\le 2C_0 r.$$ 
        The conclusion follows using a similar argument as the one in \cref{lemma:classical-result-pde} where we need to combine the above estimate on the contact set $\Lambda(v)$ with the boundary estimates for harmonic functions with oblique boundary conditions \cite{lieberman2013oblique}.
        \end{proof}
        
        \begin{proof}[Proof of \cref{lemma:semiconv1}] 
        First of all, we observe that $ii)$ is an immediate consequence of $i)$.
        Indeed, if $i)$ holds, then since $\Delta v=0$ in $B_1^+$, for every $x\in B_{\sfrac{1}{2}}^+$, we have that
        $$\partial_{dd}v=-\sum_{i=1}^{d-1}\partial_{ii}v\le C\|v\|_{L^\infty(B_1)}$$ 
        and, if we denote by $\widehat W':=(\sfrac{W'}{|W'|},0)$, then 
        $$ \partial_{\widetilde WW}v=-|W'|^2\partial_{\widehat W'\widehat W'}v+W_d^2\partial_{dd}v\le C \|v\|_{L^\infty(B_1)}.$$

    Thus, it is sufficient to prove $i)$.
    For every $\ell\in(0,1)$, we consider the solution of the oblique thin obstacle problem in $B_{\sfrac{7}{8}}$ with boundary datum $v+\ell$, namely \begin{equation}\label{vell}
     \left\{
\begin{aligned}
    &\Delta v_\ell= 0  &&  \mbox{in }B_{\sfrac{7}{8}}^+,\\
     &\nabla v_\ell \cdot W =0 && \mbox{on } B_{\sfrac{7}{8}}'\cap\Omega_{ v_\ell},\\
        &\nabla v_\ell \cdot W \le0 && \mbox{on } B_{\sfrac{7}{8}}',\\
		 &v_\ell \ge 0 &&  \mbox{on } B_{\sfrac{7}{8}}',\\
		 &v_\ell= v+\ell &&\mbox{on } \partial B_{\sfrac{7}{8}},\\
         &v_\ell(x',x_d)=v_\ell(x',-x_d)&&  \mbox{in }B_{\sfrac{7}{8}}^+
\end{aligned}
    \right.
\end{equation}
Notice that $v_{\ell}\to v$ uniformly in $B_{\sfrac{7}{8}}$ as $\ell\to 0^+$ (as a consequence of Lipschitz regularity in \cref{lemma:semiconv0}). 
We point out that a solution of \eqref{vell} can be constructed by the classical Perron's method, similarly as done in \cref{t:existence}, as the minimum supersolution of
\begin{equation*}
     \left\{
\begin{aligned}
    &\Delta v_\ell\le 0  &&  \mbox{in }B_{\sfrac{7}{8}}^+,\\
     &\nabla v_\ell \cdot W \le0 && \mbox{on } B_{\sfrac{7}{8}}',\\
\end{aligned}
    \right.
\end{equation*}
satisfying $v_\ell\ge0$ on $B_{\sfrac{7}{8}}'$ and $v_\ell\ge v+\ell$ on $\partial B_{\sfrac{7}{8}}$.
Since $v_\ell >0$ in $B_{\sfrac{7}{8}}$, we observe that for some $\beta=\beta(\ell)>0$, the function $v_{\ell}$ is harmonic in $(B_{\sfrac{7}{8}}\setminus B_{\sfrac{7}{8}-2\beta})^+$, with an oblique boundary condition in $(B_{\sfrac{7}{8}}\setminus B_{\sfrac{7}{8}-2\beta})'$. 
We divide the rest of the proof into two steps.\\
\\
{\it Step 1.} We claim that, for every $\ell>0$,
\begin{equation}\label{claim:semiconvexity}
     \left\{
\begin{aligned}
    &\Delta(\partial_{ee}v_{\ell})^{-}\ge 0&& \text{in } B_{\sfrac{2}{3}}^+,\\ &\nabla(\partial_{ee}v_{\ell})^{-}\cdot W\ge  0&& \text{on } B_{\sfrac{2}{3}}'\setminus\Lambda(v_\ell),
    \end{aligned}
    \right.\qquad\text{and}\qquad  (\partial_{ee}v_{\ell})^{-}\equiv0\quad\text{on }\Lambda(v_\ell)\cap B_{\sfrac{2}{3}}.
\end{equation}
To define $(\partial_{ee}v_{\ell})^{-}$, first, for every $\eps>0$, we set $f^\ell_{\eps,e,h}:=\max\{-\delta^2_{e,h} v_\ell,\eps\}$, where
$$\delta^2_{e,h} w(x)=\frac{ w(x+he)+ w(x-he)-2 w(x)}{h^2},\quad\mbox{for every } x\in B_{\sfrac{7}{8}},$$
and for any $w:B_1\to\R$, $e=(e',0)\in \mathbb{S}^{d-1}$, $h\in(0,\sfrac{1}{8})$. To define $(\partial_{ee}v_{\ell})^{-}$ and get \eqref{claim:semiconvexity}, we want to send $\eps \to 0^+$ and then $h \to 0^+$ in $f^\ell_{\eps,e,h}$.
We notice that, since $v_\ell= 0$ on $\Lambda(v_\ell)$ and $v_\ell\ge0$ on $B_{\sfrac{7}{8}}'$, then
\be\label{up}\delta^2_{e,h} v_\ell\ge 0\quad\text{on }\Lambda(v_\ell).
\ee
Moreover, we have that $\Delta(\delta^2_{e,h}v_\ell)=0$ in $B_{\sfrac{7}{8}}^+$. 
Since the maximum of harmonic functions is subharmonic (similarly as in \cref{lemma:min2soprasol}), then
\be\label{eq:syst1}\Delta f^\ell_{\eps,e,h}\ge 0\quad\text{on }B_{\sfrac{7}{8}}^+.\ee
Regarding the boundary condition, since $\nabla v_\ell\cdot W=0$ on $B_{\sfrac{7}{8}}'\setminus\Lambda(v_\ell)$ and $\nabla v_\ell\cdot W\le0$ on $B_{\sfrac{7}{8}}'$, then $\nabla (\delta^2_{e,h} v_\ell)\cdot W\le0$ in $B_{\sfrac{7}{8}}'\setminus\Lambda(v_\ell)$.
By \eqref{up}, we have that
\be\label{neigh}f^\ell_{\eps,e,h}\equiv \eps\quad\text{in a neighborhood of }\Lambda(v_\ell),\ee then, arguing as above we get \be\label{eq:syst2}\nabla f^\ell_{\eps,e,h}\cdot W\ge0\quad\text{on }B_{\sfrac{7}{8}}'.\ee
To conclude the first step, namely proving \eqref{claim:semiconvexity}, we would like to pass to the limit as $\eps\to0^+$ and then $h\to0^+$ in \eqref{eq:syst1}, \eqref{neigh}, \eqref{eq:syst2}.
To pass to the limit in \eqref{eq:syst1}, it is sufficient to show that
\be\label{claim:seminconvexity2}
\|f_{\eps,e,h}^\ell\|_{L^\infty(B_{\sfrac{3}{4}})}\le C(\ell),
\ee for some $C(\ell)>0$ which does not depend on $\eps$ and $h$. 
Indeed, if \eqref{claim:seminconvexity2} holds, then we have a family of uniformly bounded subharmonic functions $\{f^\ell_{\eps,e,h}\}_{\eps,h}$ in $B_{\sfrac{2}{3}}^+$, which pointwise converges to $(\partial_{ee}v_\ell)^-$. Then $(\partial_{ee}v_\ell)^-$ is subharmonic in $B_{\sfrac{2}{3}}^+$. By passing to the limit in the mean value inequality for subharmonic functions, we get the first equation in \eqref{claim:semiconvexity}.
Moreover, we can pass to the limit also in \eqref{neigh}, getting that $(\partial_{ee}v_\ell)^-\equiv 0$ on $\Lambda(v_\ell)\cap B_{\sfrac{2}{3}}'$.
Finally, since the function $v_\ell$ is smooth up to $B_{\sfrac{2}{3}}'\cap\Omega_{v_\ell}$, by estimates up to the boundary for harmonic functions with oblique boundary conditions, see \cite{lieberman2013oblique}, then we get $\nabla (\partial_{ee}v_\ell)\cdot W=0$ on $B_{\sfrac{2}{3}}'\cap\Omega_{v_\ell}$. Then, the final inequality $\nabla (\partial_{ee}v_\ell)^-\cdot W\ge0$ on $B_{\sfrac{2}{3}}'\cap\Omega_{v_\ell}$ follows by the fact that maximum of supersolutions is a supersolution, concluding \eqref{claim:semiconvexity}.  

It remains to provide \eqref{claim:seminconvexity2}. We observe that since $f^\ell_{\eps,e,h}$ satisfies \eqref{eq:syst1}, \eqref{eq:syst2}, then by maximum principle \cref{lemma:maximum-principle-thin}, the maximum of $f^\ell_{\eps,e,h}$ in a ball $B_r$ with $r\le \sfrac{7}{8}$, is attained on the boundary of such a ball $\partial B_r$.
Since $v_\ell$ is harmonic in $(B_{\sfrac{7}{8}}\setminus B_{\sfrac{7}{8}-2\beta})^+$ with an oblique boundary condition in $(B_{\sfrac{7}{8}}\setminus B_{\sfrac{7}{8}-2\beta})'$, using the corresponding $C^{2,\alpha}$ estimates in \cite{lieberman2013oblique}, we have 
$$\sup_{B_{\sfrac{3}{4}}}|f^\ell_{\eps,e,h}|=\sup_{B_{\sfrac{3}{4}}}f^\ell_{\eps,e,h}\le \sup_{B_{\sfrac{7}{8}-\beta}}f^\ell_{\eps,e,h}\le\sup_{\partial B_{\sfrac{7}{8}-\beta}}f^\ell_{\eps,e,h} \le \sup_{\partial B_{\sfrac{7}{8}-\beta}}|\delta^2_{e,h}v_\ell|+\eps\le C(\ell),$$
getting \eqref{claim:seminconvexity2}.\\
\\
{\it Step 2.}
We now proceed by Bernstein's technique. Let $\eta'\in C^\infty_c(B_{\sfrac{6}{11}}')$ be a cutoff function such that $0\le \eta'\le 1$ and $\eta'\equiv1$ in $B_{\sfrac{7}{11}}'$. Let 
$$\eta(x):=\eta'\left(x'+\frac{W'}{W_d}x_d\right)$$
be defined in the strip $|x_d|\le \sigma$ for some $\sigma >0$ small enough.
We can extend $\eta$ in $\R^d$ in the following way
$$
\eta\in C^\infty_c(B_{\sfrac{2}{3}}), \qquad 0\le \eta\le 1, \qquad \eta\equiv1 \hbox{ in } B_{\sfrac{1}{2}}, \qquad \nabla \eta\cdot W=0 \hbox{ on } \{x_d=0\}.
$$ 
Introducing the auxiliary function 
$$\psi_\ell:=\eta^2((\partial_{ee}v_\ell)^-)^2+\mu(\partial_ev_\ell)^2,$$
and providing that $\mu>0$ large enough (independently on $\ell$), using \eqref{claim:semiconvexity}, we prove that $\Delta\psi_{\ell}\ge 0$ in $B_{\sfrac{2}{3}}^+$. Indeed, 
\bea \Delta \psi_\ell&=((\partial_{ee}v_\ell)^-)^2\Delta \eta^2 +\eta^2\Delta ((\partial_{ee}v_\ell)^-)^2+2\nabla \eta^2\cdot\nabla ((\partial_{ee}v_\ell)^-)^2+2\mu |\nabla \partial_ev_\ell|^2\\
&\ge ((\partial_{ee}v_\ell)^-)^2\Delta \eta^2+2\eta^2 (\partial_{ee}v_\ell)^-\Delta (\partial_{ee}v_\ell)^-+2|\nabla (\partial_{ee}v_\ell)^-|^2\eta^2-8|\eta||\nabla\eta||(\partial_{ee}v_\ell)^-||\nabla (\partial_{ee}v_\ell)^-|+2\mu |\nabla \partial_ev_\ell|^2\\
&\ge ((\partial_{ee}v_\ell)^-)^2\Delta \eta^2+2|\nabla (\partial_{ee}v_\ell)^-|^2\eta^2-8|\eta||\nabla\eta||(\partial_{ee}v_\ell)^-||\nabla (\partial_{ee}v_\ell)^-|+2\mu |\nabla \partial_ev_\ell|^2\\
&\ge ((\partial_{ee}v_\ell)^-)^2\Delta \eta^2 -8|\nabla \eta|^2|(\partial_{ee} v_\ell)^-|^2+2\mu |\nabla \partial_ev_\ell|^2\\&\ge ((\partial_{ee}v_\ell)^-)^2(2\mu-C_\eta)\\
&\ge0,
\eea 
having chosen $C_{\eta}>0$ such that $-\Delta\eta^2+8|\nabla \eta|^2\le C_{\eta}.$
Moreover, by \eqref{claim:semiconvexity} and the regularity of $v_\ell$ in $\Omega_{v_\ell}$, we have that $\nabla \eta\cdot W=0$, $\nabla (\partial_ev_\ell)\cdot W=0$ and $\nabla (\partial_{ee}v_\ell)^-\cdot W\ge0$ on $B_{\sfrac{2}{3}}'\cap\Omega_{v_\ell}$, then $$\nabla\psi_{\ell}\cdot W\ge 0 \quad\text{in }B_{\sfrac{2}{3}}'\cap \Omega_{v_\ell}.$$
By maximum principle \cref{lemma:maximum-principle-thin}, the maximum in $B_{\sfrac{2}{3}}\setminus\Lambda(v_{\ell})$ of $\psi_\ell$ is attained on the boundary. 
Since $\eta\equiv 1$ in $B_{\sfrac{1}{2}}$, $\eta\equiv0$ on $\partial B_{\sfrac{2}{3}}$ and by \eqref{claim:semiconvexity} $(\partial_{ee}v_\ell)^-\equiv0$ on $\Lambda(v_{\ell})\cap B_{\sfrac{2}{3}}$, we obtain
\bea \|\psi_\ell\|_{L^\infty(B_{\sfrac{2}{3}})}\le\|\psi_\ell\|_{L^\infty(\partial B_{\sfrac{2}{3}})}+\|\psi_\ell\|_{L^\infty(\Lambda(v_{\ell})\cap B_{\sfrac{2}{3}})}\le 2\mu\lVert \partial_ev_\ell\rVert_{L^{\infty}(B_{\sfrac{2}{3}})}^2\le C\|v_\ell\|_{L^\infty(B_1)}^2\le C\|v\|_{L^\infty(B_1)}^2.
\eea 
having used the Lipschitz regularity in \cref{lemma:semiconv0}. 
Therefore,
$$\|(\partial_{ee}v_\ell)^-\|_{L^\infty(B_{\sfrac{1}{2}})}\le\|\psi_\ell\|_{L^\infty(B_{\sfrac{1}{2}})}^{\sfrac{1}{2}}\le C\|v\|_{L^\infty(B_1)},$$ 
implying that the family of function $v_\ell$ is uniformly semiconvex. Letting $\ell\to0^+$ we conclude the proof.
    \end{proof}

\subsection{The function \texorpdfstring{$\sigma_{W}$}{sigmaW}}
\label{estimatefinal0} 

In the following, we show that the oblique derivative
is well-defined on the thin space. 
For technical reason, we introduce the modified functions $v_\ell$ for $\ell \in (0,1)$, as in the proof of \cref{lemma:semiconv1}.
Precisely, for every $\ell\in(0,1)$, we consider the solution of the following problem
\begin{equation}\label{vell2}
     \left\{
\begin{aligned}
    &\Delta v_\ell= 0  &&  \mbox{in }B_{1}^+,\\
     &\nabla v_\ell \cdot W =0 && \mbox{on } B_{1}'\cap\Omega_{ v_\ell},\\
        &\nabla v_\ell \cdot W \le0 && \mbox{on } B_{1}',\\
		 &v_\ell \ge 0 &&  \mbox{on } B_{1}',\\
		 &v_\ell= v+\ell &&\mbox{on } \partial B_{1},\\
         &v_\ell(x',x_d)=v_\ell(x',-x_d)&&  \mbox{in }B_{1}^+.
\end{aligned}
    \right.
\end{equation}
We point out that such a solution can be constructed by Perron's method, as already observed in the proof of \cref{lemma:semiconv1}.
In the rest of the proof of \cref{lemma:regolaita_prob_lin-thin}, we fix $\ell \in (0,1)$ and we establish $C^{1, \alpha}$ estimates for $v_\ell$ that are uniform in $\ell$, implying that the same $C^{1, \alpha}$ estimates hold true for $v$.
    
\begin{remark}\label{re:fundamental-rem} 
    Let $v:{B_1}\to\R$ be a viscosity solution of \eqref{linearizzato-bordo} and $v_\ell$ be as in \eqref{vell2}.
    By \cref{lemma:semiconv1}, for every $(x',0)\in B_1'$ the following quantity 
    \be\label{sigmadef}\sigma_W(x'):=\lim_{t\to0^+}\nabla v_\ell(x'+t\widetilde W)\cdot W\ee 
    is well-defined, where 
    $\widetilde W:=(-W',W_d)$ is the adjoint direction. 
    Indeed, for every $x=(x',0)\in B_1'$, there is $\rho>0$ such that $B_{\rho}'(x)\subset B_1'$, and, by \cref{lemma:semiconv1},
    there is a constant $C=C(\rho)>0$ such that 
    $$\partial_{\widetilde WW}v_\ell(x)\le C.$$
    Therefore, $\partial_{W}v_\ell(x'+t\widetilde W)-Ct$ is monotone non increasing in $t$, and thus \eqref{sigmadef} is well-defined. Moreover, by the elliptic regularity for the oblique boundary problems \cite{lieberman2013oblique}, we have 
    \be\label{sigma=0}\sigma_{W}(x')\equiv0\quad\text{for every }(x',0)\in\Omega_{v_\ell}.\ee
    \end{remark}
    
    We notice that the viscosity condition $\nabla v\cdot W\le0$ on $B_1'$ suggests that $\sigma_{W}(x')\le0$ on $B_1'$. This claim is proved in the next lemma by a penalization argument.
    
    \begin{lemma}\label{lemma:penalized}
        Let $v:{B_1}\to\R$ be a viscosity solution of \eqref{linearizzato-bordo} and $v_{\ell}$ be as in \eqref{vell2}.
        If $\sigma_{W}$ is defined as in \eqref{sigmadef}, then $$\sigma_{W}(x')\le 0\quad\text{for every}\quad (x',0)\in B_1'.$$
    \end{lemma}
    \begin{proof}
        The proof is based on a penalization argument, similarly as in \cite{milakissilvestre-signorini}. The function $v_\ell$ can be constructed as limit of the following problem
        \begin{equation*}
     \left\{
\begin{aligned}
    &\Delta v_\ell^k=0&& \text{in } B_{1}^+,\\ 
    &\nabla v_\ell^k\cdot W=  0&& \text{on } B_{1}'\cap\{v_\ell^k>0\},\\
    &\nabla v_\ell^k\cdot W=kv_\ell^k&& \text{on } B_{1}'\cap\{v_\ell^k\le0\},\\
    &v_\ell^k=v_\ell&& \text{on }\partial B_1,\\
    &v_\ell^k(x',x_d)=v_\ell^k(x',-x_d),&&
    \end{aligned}
    \right.
\end{equation*}
as $k\to +\infty$.
The existence of a solution can be obtained by Perron’s method. By maximum principle, regularity for oblique elliptic PDE \cite{lieberman2013oblique} and the fact that $v_\ell>0$ in $\partial B_1$, we can get the following conditions for $v_\ell^k$ exactly as in \cite[Section 2]{milakissilvestre-signorini}:
\begin{itemize}
    \item[(a)] $\|v_\ell^k\|_{L^\infty(B_1)}\le \|v_\ell\|_{L^\infty(B_1)}$;
    \item[(b)] $\nabla v^k_\ell\cdot W$ is uniformly bounded in $B_{1-\eta}'$, by a constant $C=C(\ell,\eta)>0$ independently on $k$, for every $\eta>0$;
    \item[(c)] $v_\ell^k\to v_\ell$ uniformly in $B_1$; 
    \item[(d)] $\nabla v_\ell^k\to \nabla v_\ell$ uniformly in $B_{1-\eta}\cap\{x_d>\eta\}$, for every $\eta>0$;
    \item[(e)] $v_\ell^k\in C^{1,\alpha}(\overline B_{1-\eta})$ for every $\eta>0$, but (a priori) the $C^{1,\alpha}$ norm possibly depending on $k$.
\end{itemize}
Since $v_\ell>0$ in $\partial B_1$, there is $\beta=\beta(\ell)>0$ such that $v_{\ell}>0$ and $(B_{1}\setminus B_{1-\beta})'$, thus it is harmonic in $(B_{1}\setminus B_{1-\beta})^+$, with an oblique boundary condition in $(B_{1}\setminus B_{1-\beta})'$. Moreover, combining $(b)$ and the oblique elliptic regularity in \cite{lieberman2013oblique}, there exists $M=M(\beta)>0$ such that 
$\nabla v_\ell^k\cdot W\le M$ in $B_{1-\sfrac{\beta}{3}}\setminus B_{1-\sfrac{2\beta}{3}}$.
Let $\Psi$ be the solution of the following elliptic problem
$$
\left\{
\begin{aligned}
    &\Delta \Psi = 0&&\hbox{in }B_{1- \sfrac{\beta}{2}}^+,\\
    &\Psi = 0&&\hbox{on }B_{1- \beta}',\\
     &\Psi= M&&\hbox{on }\partial B_{1- \sfrac{\beta}{2}},\\
     &\Psi=  \frac{2M\left(\abs{x}-1 + \beta\right) }{\beta}&&\hbox{on } B'_{1- \sfrac{\beta}{2}}\setminus B_{1- {\beta}}'.\\
\end{aligned}
\right.
$$
By comparison $\nabla v_\ell^k\cdot W \leq \Psi$ in $B_{1- \sfrac{\beta}{2}}^+$. Letting $k\to+\infty$ and using $(d)$, we have that $\nabla v_\ell\cdot W \leq \Psi$ in $B_{1- \sfrac{\beta}{2}}^+$.
Thus, for every $(x',0)\in B'_{1- \sfrac{\beta}{2}}$, 
$$\sigma_{W}(x')=\lim_{t\to0^+}\nabla v_\ell(x'+t\widetilde W)\cdot W \leq \Psi(x',0)=0.$$
     On the other hand, by \eqref{sigma=0}, $\sigma_W\equiv0$ in $B_1'\setminus B_{1-\beta}'$, concluding the proof.
    \end{proof}
    
    In the following lemma we show that we can link almost everywhere on the contact set the normal derivative with the one along the oblique obstacle $W$, on the hyperplane $\{x_d=0\}$.
    
   \begin{lemma}\label{rem:q.o.}
      Let $v:{B_1}\to\R$ be a viscosity solution of \eqref{linearizzato-bordo} and $v_{\ell}$ be as in \eqref{vell2}. Then there exists a set $\mathcal{G}\subset B_1'$ such that, for every $(x',0)\in \Lambda(v_\ell)\cap \mathcal{G}$ 
$$\sigma_d(x'):=\lim_{t\to0^+}\partial_{d}v_\ell(x'+t\widetilde W)=\sigma_{W}(x')$$ and $$\mathcal{H}^{d-1}(B_1'\setminus\mathcal{G})=0.$$ 
   \end{lemma}
    \begin{proof}
       In the rest of the proof, without loss of generality, we suppose that $\|v\|_{L^\infty(B_1)}\le 1$.
        We observe that, by the Lipschitz regularity of $v_\ell$ proved in \cref{lemma:semiconv0}, $v_\ell(x',0)$ is differentiable almost everywhere along the direction $\widehat W':=(\sfrac{W'}{|W'|},0)$.
        Namely, there exists a set $\mathcal{G}\subset B_1'$ such that for every $x_0\in\mathcal{G}$, the tangential derivative
        $$\partial_{\widehat W'}v_\ell(x_0)=\lim_{h\to0}\frac{v_\ell(x_0+h\widehat W')-v_\ell(x_0)}{h}$$
     is well-defined, and $\mathcal{H}^{d-1}(B_1'\setminus \mathcal{G})=0$. 
     In particular, for every $x_0\in\Lambda(v_\ell)\cap \mathcal{G}$, since $v_\ell(x_0)=0$ and $v_\ell\ge0$ on $B_1'$, we have that 
     \be\label{eq:richiama}\partial_{\widehat W'}v_\ell(x_0)=0.\ee
     Moreover, since $v_\ell$ is smooth in $B_1^+$, for every $x\in B_1^+$ and $h$ small enough, by \cref{lemma:semiconv1}, we have
     $$v_\ell(x+h \widehat W')=v_\ell(x)+\partial_{\widehat W'}v_\ell(x) h+\iint_{[x_0,x_0+h\widehat W']}\partial_{\widehat W'\widehat W'}v_\ell\ge v_\ell(x)+\partial_{\widehat W'}v_\ell(x) h-Ch^2.$$
     Here we used the notation
     \be\label{notation-doubleintegral}\iint_{[a,b]}w:=\int_{0}^{|b-a|}\int_0^s w\left(a+\frac{b-a}{|b-a|}t\right)\,dt\,ds\ee 
     for the double integral over the segment between points $a$ and $b$. It follows that, for $h>0$,
     $$\partial_{\widehat W'}v_\ell(x)\le \frac{v_\ell(x+h\widehat W')-v_\ell(x)}{h}+Ch,$$
     thus for $x = x'+t\widetilde W$, we get
    $$\limsup_{t\to0^+}\partial_{\widehat W'}v_\ell(x'+t\widetilde W)\le \frac{v_\ell((x',0)+h\widehat W')-v_\ell(x',0)}{h}+Ch.$$ 
    In particular, if $(x',0)\in\Lambda(v_\ell)\cap \mathcal{G}$, by \eqref{eq:richiama},
    we obtain 
    $$\limsup_{t\to0^+}\partial_{\widehat W'}v_\ell(x'+t\widetilde W)\le 0.$$
    On the other hand, arguing similarly in the case $h<0$, we get the reverse inequality for the liminf. Then, for $(x',0)\in\Lambda(v_\ell)\cap\mathcal{G}$, we have 
    $$\lim_{t\to0^+}\partial_{\widehat W'}v_\ell(x'+t\widetilde W)= 0,$$
    which concludes the proof.
     \end{proof}

\subsection{\texorpdfstring{$C^{1,\alpha}$}{C1alpha} estimates}
\label{estimatefinal} 
In this section we conclude the proof of \cref{lemma:regolaita_prob_lin-thin}.
In the following, we use the notation 
$$\tens{P}_{W^\perp}x:=x-(x\cdot W)W,$$ 
to denote the projection of $x\in\R^d$ onto the hyperplane orthogonal to $W$ and passing through the origin.

\begin{lemma}\label{lemmab6}
    Let $v:B_1\to\R$ be a viscosity solution of \eqref{linearizzato-bordo} and $v_\ell$ be as in \eqref{vell2}. For every $x_0\in B_1'\cap\Omega_{v_\ell}$, let $h_{x_0}$ be the function defined as 
    $$h_{x_0}(x):=|\tens{P}_{W^\perp}(x-x_0)|^2-(d-1)x_d^2.$$ 
    For every neighborhood $U_{x_0}\subset B_1$ of $x_0$, there exists $y\in\partial U_{x_0}\cap\{x_d>0\}$ such that 
    $$v_\ell(y)\ge h_{x_0}(y).$$ 
\end{lemma}

\begin{proof}
    Let $x_0\in B_1'\cap\Omega_{v_\ell}$. We observe that, for all $x \in \R^d$
    $$|\tens{P}_{W^\perp}(x-x_0)|^2=(x-x_0)^2-((x-x_0)\cdot W)^2.$$ 
    By simple computations $\Delta h_{x_0}=0$ in $B_1$ and $\nabla h_{x_0}\cdot W=0$ on $B_1'$.
    Therefore $p:=v_\ell-h_{x_0}$ is harmonic in $B_1^+$ and satisfies $\nabla p\cdot W=0$ in $B_1'\setminus \Lambda(v_\ell)$.
    Then, by maximum principle \cref{lemma:maximum-principle-thin}, in $U_{x_0}\setminus\Lambda(v_\ell)$, the function $p$ must achieve its maximum on the boundary $\partial(U_{x_0}\setminus\Lambda(v_\ell))$. Since $p(x_0)=v_\ell(x_0)>0$, then 
    $$\sup_{\partial (U_{x_0}\setminus\Lambda(v_\ell))\cap\{x_d\ge0\}}(v_\ell-h_{x_0})\ge p(x_0)>0.$$ 
    On the other hand, on $\Lambda(v_\ell)$, we have $v_\ell-h_{x_0}=-h_{x_0}<0$. Then 
    $$\sup_{\partial U_{x_0}\cap\{x_d\geq0\}}(v_\ell-h_{x_0})=\sup_{\partial (U_{x_0}\setminus\Lambda(v_\ell))\cap\{x_d\ge0\}}(v_\ell-h_{x_0})>0.$$ 
    To conclude the proof, if the supremum is attained in $z \in \{x_d >0\}$, then we choose $y = z$, otherwise if $z \in \{x_d = 0\}$, then we choose $y \in \{x_d >0\}$ close to $z$.
\end{proof}

Next we show that, given a point $x_0\in\Omega_{v_\ell}\cap B_1'$ and $\gamma>0$ small, we can find a ball of radius comparable to $\gamma$ which is almost everywhere contained in $\{\sigma_W(x')>-\gamma\}$.

\begin{lemma}\label{lemmab7}
    There are constants $C_1>0$, $C_2>0$, and $\gamma_0>0$ depending on $\delta$ and $\|v\|_{L^\infty(B_1)}$ such that the following holds.
    Let $v:B_1\to\R$ be a viscosity solution of \eqref{linearizzato-bordo}, $v_\ell$ be as in \eqref{vell2} and $\mathcal{G}$ be the set defined in \cref{rem:q.o.}.
    Let $x_0\in B_{\sfrac{3}{4}}'\cap\Omega_{v_\ell}$ and 
    $$S_\gamma:=\{(x',0)\in B_1':\sigma_{W}(x')>-\gamma\},$$
    for $\gamma\in(0,\gamma_0)$. 
    Then, there is $\overline x\in B_1'$ such that 
    $$B'_{C_1\gamma}(\overline x)\cap\mathcal{G}\subset B'_{C_2\gamma}(x_0)\cap S_\gamma.$$
\end{lemma}
\begin{proof}
    In the rest of the proof, without loss of generality, we suppose that $\|v\|_{L^\infty(B_1)}\le 1$.
    Let $U_{x_0}$ be the oblique cylinder defined as follows
    $$U_{x_0}:=\{(x',x_d)\in \R^d:\ |\tens{P}_{W^\perp}(x-x_0)|\le 2\gamma ,\ |x_d|\le \eps_1\gamma\},$$ for a constant $\eps_1>0$ to be chosen later.
    By \cref{lemmab6}, there exists $y \in \partial U_{x_0} \cap \{x_d >0\}$
    such that $ v_{\ell}(y)\ge h_{x_0}(y)$.
    We split two cases.\\
    \\
    \textit{Step 1.} Suppose that $y=(y',y_d)$ belongs to the lateral face of the oblique cylinder $U_{x_0},$ i.e.~$|\tens{P}_{W^\perp}(y-x_0)|= 2\gamma$ and $0< y_d\le \eps_1\gamma$. If $\eps_1$ is small, then 
    $$v_\ell(y)\ge h_{x_0}(y)=|\tens{P}_{W^\perp}(y-x_0)|^2-(d-1)y_d^2\ge 4\gamma^2-(d-1)\eps_1^2\gamma^2\ge \gamma^2.$$
         We claim that we can find $\eps_2>0$ such that 
         \be\label{claim-z'} 
        \hbox{if }\, (z',0)\in\mathcal{D}:=\left\{ (z',0)\in B'_{\eps_2\gamma}((y',0)):\ \nabla'v_\ell(y)\cdot ( z'- y')\ge0\right\}\cap\mathcal{G}\implies (z',0)\in S_\gamma,\ee
        where $\nabla'v_{\ell}(y):=(\partial_{1}v_\ell(y),\ldots,\partial_{d-1}v_\ell(y))$.
         If \eqref{claim-z'} holds true, then we get the conclusion choosing $\overline x$ and $C_1>0$, $C_2>0$ such that $B'_{C_1\gamma}(\overline x)\cap\mathcal{G}\subset \mathcal{D}$ and $C_2$ depending on $\delta$, $\eps_1$ and $\eps_2$.
        It is left to ensure that \eqref{claim-z'} yields.
        Take $(z',0)\in\mathcal{D}$ and let $e=(e',0)\in\mathbb{S}^{d-1},$ with $e':=\frac{z'- y'}{| z'- y'|}$. Since $v_\ell$ is smooth in $B_1^+$, by \cref{lemma:semiconv1}, we have that $$v_\ell(z',y_d)=v_\ell(y',y_d)+\nabla'v_\ell(y)\cdot (z'-y')+\iint_{[(y',y_d),(z',y_d)]}\partial _{ee}v_\ell\ge \gamma^2- C|z'-y'|^2\ge (1- C\eps_2^2)\gamma^2>0$$
        having chosen $\eps_2$ small enough. Here we used the notation \eqref{notation-doubleintegral} for the double integral.  
        Suppose by contradiction that $ (z',0)\not\in S_\gamma$, then $ (z',0)\in\Lambda(v_\ell)$, since $\sigma_W\equiv0$ in $B_1'\cap \Omega_{v_\ell}$ by \eqref{sigma=0}. 
        Then, using the semiconcavity \cref{lemma:semiconv1} together with \cref{rem:q.o.}, we get
        $$
        \begin{aligned}
            v_\ell(z',y_d)&=v_\ell(z',0)+\sigma_{d}(z')y_d+\iint_{[(z',0),(z',y_d)]}\partial_{dd}v_\ell=
        \sigma_{W}(z')y_d+\iint_{[(z',0),(z',y_d)]}\partial_{dd}v_\ell\\
        &\le 
        -\gamma y_d+Cy_d^2\le y_d\gamma (C\eps_1-1)<0,
        \end{aligned}$$
        where we have chosen $\eps_1$ small enough and we get the contradiction. We point out that the above expansion of $v_\ell(z',y_d)$ is obtained for points like $x'+t\widetilde W$ and then sending $t\to0^+$.
        \\
        \\
        \textit{Step 2.} Now, we suppose that $y=(y',y_d)$ belongs on the upper base of the oblique cylinder $U_{x_0},$ i.e.~$|\tens{P}_{W^\perp}(y-x_0)|\le 2\gamma$ and $y_d=\eps_1\gamma$. 
        Then 
        $$v_\ell(y)\ge h_{x_0}(y)=|\tens{P}_{W^\perp}(y-x_0)|^2-(d-1)y_d^2\ge -\eps_1^2(d-1)\gamma^2.$$ As above, if $(z',0)\in \mathcal{D}$, similarly as in the previous step, we have 
        $$v_\ell(z',y_d)=v_\ell(y',y_d)+\nabla'v_\ell(y)\cdot (z'-y')+\iint_{[(y'y_d),(z',y_d))]}\partial_{ee}v_\ell\ge -(d-1)\eps_1^2\gamma^2-C|z'-y'|^2\ge-\gamma^2((d-1)\eps_1^2+C\eps_2^2).$$ 
        If by contradiction $(z',0)\not\in S_\gamma$, then, since $y_d=\eps_1\gamma$, we have
        $$v_\ell(z',y_d)=v_\ell(z',0)+\sigma_{d}(z')y_d+\iint_{[(z',0),(z',y_d)]}\partial_{dd}v_\ell= 
        \sigma_{W}(z')y_d+\iint_{[(z',0),(z',y_d)]}\partial_{dd}v_\ell\le 
        -\gamma y_d+Cy_d^2=\gamma^2\eps_1(C\eps_1-1).$$ 
        Combining the two estimates, we get the contradiction by choosing $\eps_2=\eps_1$ small enough.
        \end{proof}

        In the following, we will use the cylinders
            $$\mathcal{C}(x_0,r,\gamma_1,\gamma_2):=\left\{x=(x',x_d)\in \R^d:\ (x',0)\in B'_{r}(x_0),\ \gamma_1\le x_d\le \gamma_2\right\},$$
            and 
            $$\mathcal{C}(x_0,r):=\left\{x=(x',x_d)\in \R^d:\ (x',0)\in B'_{r}(x_0),\ 0< x_d\le r\right\}.$$
            
        To conclude the proof of \cref{lemma:regolaita_prob_lin-thin}, it will be fundamental the following technical lemma.

        \begin{lemma}\label{teclemma-w} 
        Let $q:\mathcal{C}(0,1)\to\R$ be a bounded non-negative continuous function such that $\Delta q=0$ in $\mathcal{C}(0,1)$. 
        Assume that, given $\mathcal{G}\subset B_1'$ such that $\mathcal{H}^{d-1}(B_1'\setminus\mathcal{G})=0$ and $\widetilde W\in\partial B_1$ with $\widetilde W_d\ge\delta$, we have
            \begin{equation}
                \label{e:ipotesi_tecnica}
                \lim_{t\to0^+} q(x'+t\widetilde W)\ge 1\quad\text{for every}\quad (x',0)\in B_{\rho}'(\overline x)\cap\mathcal{G},
            \end{equation}
            for some ball $B_{\rho}'(\overline x)\subset B_1'$ (we do not assume continuity up to $\overline{\mathcal{C}(0,1)}$).
            For every $c>0$, there is a constant $\eps=\eps(\delta,\rho,c)>0$ such that
            $$q\ge \eps\quad\text{in }\mathcal{C}\left(0,\frac{1}{2},c,\frac{3}{4}\right).$$
        \end{lemma}
         
        \begin{proof}
           Let $\mathcal{T}\subset \mathcal{C}(0,1)$ be a smooth approximation of $\mathcal{C}(0,1)$ such that $B_\rho'(\overline x)\subset\partial\mathcal{T}\cap B_1'$. For every $g\in L^\infty(\partial \mathcal{T})$, we denote by
           $$\mathcal{S}_\mathcal{T}(g):=\int_{\partial \mathcal{T}} K(x,y) g(y) \, d\mathcal{H}^{d-1}(y)
           $$ the harmonic function with ``boundary datum'' $g$,
           where $K$ is the Green's function of the smooth domain $\mathcal{T}$. By \cite[Theorem 1]{stein-book}, since $q$ is a bounded harmonic function, there exists a function $f \in L^\infty(\partial \mathcal{T})$, such that 
           $$
           \mathcal{S}_\mathcal{T}(f) =q.
           $$
           Moreover, by \cite[Theorem 4]{stein-book} and \eqref{e:ipotesi_tecnica}, we have that $f\ge 1$ on $B'_\rho(\overline x)\cap\mathcal{G}$. Thus there exists $\widetilde f\in L^\infty(\partial \mathcal{T})$ such that $\widetilde f\ge 1$ on $B'_\rho(\overline x)$ and $\widetilde f=f$ almost everywhere. Then, in $\mathcal{T}$, we have 
           $$\widetilde q:=\mathcal{S}_{\mathcal{T}}(\widetilde f)=\mathcal{S}_{\mathcal{T}}(f)=q.$$
           Applying a standard barrier argument (see e.g. \cite[Lemma 3.5]{milakissilvestre-signorini}) to the function $\widetilde q$ and using that $\widetilde q=q$ in $\mathcal{T}$, we conclude the proof.
        \end{proof}

        The final lemma establishes the $C^{0,\alpha}$ regularity of $\sigma_{W}$ around points in $\Omega_{v_\ell}$, and it is the key step in concluding the proof of \cref{lemma:regolaita_prob_lin-thin}.
        \begin{lemma}\label{lemma:final-c}
            There are constants $C>0$ and $\alpha_0\in(0,1)$ depending on $\delta$ such that the following holds. 
            Let $v:B_1\to\R$ be a viscosity solution of \eqref{linearizzato-bordo}, $v_{\ell}$ be as in \eqref{vell2} and let $x_0\in B_{\sfrac{1}{2}}'\cap\Omega_{v_\ell}$. Then, for every $x'\in B_{\sfrac{1}{2}}'$, we have that 
            \be\label{eq:stimac0alpha-persigma}\sigma_{W}(x')\ge -C\|v\|_{L^\infty(B_1)}|x'-x_0'|^{\alpha_0}.\ee
        \end{lemma}
        \begin{proof}
            In the rest of the proof, without loss of generality, we suppose that $\|v\|_{L^\infty(B_1)} \le 1$.
            To prove \eqref{eq:stimac0alpha-persigma}, it is sufficient to show that there are constants $\theta\in(0,1)$ and $\eta\in(0,1)$ such that for every $k\in\N$
            \be\label{indhyp-k}\nabla v_\ell(x)\cdot W\ge -\theta^k\quad\text{for every}\quad x\in \mathcal{C}(x_0,\eta^k).\ee
            By induction, we suppose that \eqref{indhyp-k} holds for $k \in \N$, we prove that \eqref{indhyp-k} holds for $k+1$. 
            By \cref{lemmab7} with $\gamma:=C_2^{-1}\eta^k$, there exists $\overline x\in B_1'$ such that 
            $$B'_{C_1C_2^{-1}\eta^k}(\overline x)\cap\mathcal{G}\subset B'_{\eta^k}(x_0)\cap S_{C_2^{-1}\eta^k},$$ 
            where $\mathcal{G}$ is the set defined in \cref{rem:q.o.}.
            For a constant $c>0$ to be chosen later
            and $\rho:=C_1C_2^{-1}$, we can apply \cref{teclemma-w} to the function $q:B_{\eta^k}^+(x_0)\to \R$ defined as
            $$q:=\frac{\nabla v_\ell(x)\cdot W+\theta^k}{\theta^k-C_2^{-1}\eta^k}$$
            to get that
            $$\nabla v_\ell(x)\cdot W\ge -\theta^k+\eps(\theta^k-C_2^{-1}\eta^k)\ge -\theta^k+\frac{1}{2}\eps\theta^k
            \quad\text{in }
            \mathcal{C}\left(x_0,\frac{\eta^k}2,c\eta^k,\frac{3\eta^k}{4}\right),$$ choosing $\eta<< \theta$, where $\eps>0$ depends only on $d$, $\delta$ and $c$. 
            Precisely, we choose the constant $c= c(\delta)$ such that if
            $$y\in \mathcal{C}\left(x_0,\frac{\eta^k}{4},0,c\eta^k\right),\quad y_d>0,$$ 
            there is $t>0$ such that along the adjoin direction $\widetilde W$ we can find a point $\overline{y}$ such that
            $$\overline y=y+t\widetilde W\in\mathcal{C}\left(x_0,\frac{\eta^k}2,c\eta^k,\frac{3\eta^k}{4}\right)\quad\text{and}\quad \overline y_d=c\eta^k.$$
            In particular 
            $$t=\frac{c\eta^k-y_d}{W_d}.$$
            By the semiconcavity in \cref{lemma:semiconv1} and the fundamental theorem of calculus, we have
            $$\partial_W v_\ell(y)=-t\int_{0}^{1}\partial_{\widetilde WW}v_\ell(sy+(1-s)\overline y)\,ds+\partial_{W}v_\ell\left(\overline y\right)\ge -C\left(\frac{c\eta^k-y_d}{{W_d}}\right)-\theta^k+\frac{1}{2}\eps\theta^k\ge -C \frac{c}{\delta}\eta^k-\theta^k+\frac{1}{2}\eps\theta^k.$$
            Combining the above inequalities, we get  that
            $$\nabla v_\ell(y)\cdot W\ge -C \frac{c}{\delta}\eta^k-\theta^k+\frac{1}{2} \eps\theta^k\quad\text{for every}\quad y\in\mathcal{C}\left(x_0,\frac{\eta^k}4\right). $$
            To conclude, it is sufficient to choose $\theta$ such that
            $$\theta^{k+1}\ge C \frac{c}{\delta}\eta^k+\theta^k-\frac{1}{2}\eps\theta^k.$$ 
            This is possible, since if we choose
            $\eta<<\theta$ and $\theta>1-\frac{1}{4}\eps$, we end up with
            $$
            C \frac{c}{\delta}\eta^k+\theta^k-\frac{1}{2}\eps\theta^k \leq\left(1- \frac{1}{4} \eps\right)\theta^{k}\le \theta^{k+1}, 
            $$
            concluding the proof.
        \end{proof}
        Now we are ready to prove \cref{lemma:regolaita_prob_lin-thin}, which is a consequence of \cref{lemma:final-c}.
        \begin{proof}[Proof of \cref{lemma:regolaita_prob_lin-thin}]
        Notice that $\sigma_W \equiv 0$ in $\Omega_{v_\ell}$ by \eqref{sigma=0}, and that $\sigma_W$ is $C^{0,\alpha}$ regular in the interior of the contact set $\Lambda(v_\ell)$. By \cref{lemma:final-c}, and since $\sigma_W \le 0$ by \cref{lemma:penalized}, we deduce that $\sigma_W$ is $C^{0,\alpha}$ also up to points on the free boundary $\partial \Omega_{v_\ell} \cap B_{\sfrac{1}{2}}'$. Therefore, $\sigma_W \in C^{0,\alpha}(B_{\sfrac{1}{2}}')$.
            
            Then (see \cite{milakissilvestre-signorini,fullyfernandez} for more details) the $C^{1,\alpha}$ regularity up to the boundary for $v_{\ell}$ follows by classical estimates for oblique boundary problem \cite{lieberman2013oblique}.
            Since the $C^{1,\alpha}$ estimates for $v_{\ell}$ do not depend on $\ell$, the same conclusion holds for $v$, namely $v\in C^{1,\alpha}(\overline{B_{\sfrac{1}{2}}^+})$ with uniform estimates.
        \end{proof}

\begin{appendix}
    \section{Change of coordinates around contact points}
\label{sec:change_coordinates}

This appendix provides a detailed description for the change of coordinates done in \cref{subsec:flatproblem} to make the boundary of $K$ flat. Without loss of generality, let $x_0=0 \in \partial \Omega_u \cap K\cap \mathcal{B}$ be a contact point. 
For all $\eta_0 >0$, we can find a radius $\sigma >0$ and $g \in C^\infty(B_\sigma \cap \{x_d = 0\})$, such that
\begin{equation}
\label{e:ipotesi_boundary_reg_wlog_g}
\partial K \cap B_\sigma = \left\{x_d = g(x')\right\}\qquad \hbox{and} \qquad \norm{g}_{L^\infty(B_\sigma \cap \{x_d = 0\})}\leq \eta_0.
\end{equation}
Moreover, let $\delta>0$ such that $V\cdot \nu_K(x_0)\ge 4\delta$  and  suppose that
\begin{equation}
\label{e:ipotesi_boundary_reg_wlog}
    V\cdot\nu_K(x) \geq 3\delta >0 \quad \hbox{ for all } x \in B_\sigma.
\end{equation}
We consider the diffeomorphism 
$$
\Theta:\left\{\begin{aligned}
B_\sigma &\to \Theta(B_\sigma)\\
(x', x_d)&\mapsto \Theta(x', x_d):= (x', x_d - g(x')),
\end{aligned}
\right.
$$
and we define the function $v:\Theta(B_\sigma)\to \R^+$ such that 
\begin{equation*}
v(y):=u(\Theta^{-1}(y)).
\end{equation*}
Up to rescaling $v$, 
we get that $v:B_1\to\R$. Moreover, the function $v$ is a viscosity solution of
\begin{equation*}
\left\{\begin{aligned}
&{\rm div}\left(\tens{A}(y) \nabla v\right) =0 &&\Omega_v \cap B_1 \cap\{y_d >0\},\\
&\tens{A}(y) \nabla v \cdot \nabla v = \tens{A}(y) \nabla v \cdot \widetilde{V}(y) &&\partial \Omega_v \cap B_1 \cap\{y_d >0\},\\
&\tens{A}(y) \nabla v \cdot \nabla v \geq \tens{A}(y) \nabla v \cdot \widetilde{V}(y) &&\partial \Omega_v \cap B_1 \cap\{y_d =0\},\\
&v = 0 &&\{y_d \leq 0\},
\end{aligned}
\right.
\end{equation*}
where, since $D\Theta^{-1}$ is a triangular matrix,
\begin{equation*}	
	\tens{A}(y) :=\left(D\Theta^{-1}(y)\right)^{-T}
	\left(D\Theta^{-1}(y)\right)^{-1}
	 \abs{
		{\rm det}\left(D \Theta^{-1}(y)\right)}=\left(D\Theta^{-1}(y)\right)^{-T}
	\left(D\Theta^{-1}(y)\right)^{-1}
\end{equation*}
 is a smooth uniform elliptic symmetric matrix with $\tens{A}(0)=\tens{I}$ and for some smooth vector
\begin{equation*}
    \widetilde V(y):=D\Theta^{-1}(y)V.
\end{equation*} 
We observe that the function $u$ is locally Lipschitz continuous in $\mathcal{B}$, by \cref{rem:lipschitz_tutte_viscose}, thus also $v$ is locally Lipschitz continuous in $B_1$.
Given $\Pi_0=\Pi_0(\delta)>0$ small enough, we can take $\eta_0=\eta_0(\Pi_0,\delta)>0$ in \eqref{e:ipotesi_boundary_reg_wlog_g} small enough such that
\be\label{e:widetildeV-diff}
\| \tens{A}(y)-\tens{I}\|_{L^\infty(B_1)}\le \Pi_0\qquad\text{and}\qquad
\| \tens{A}(y)\widetilde V(y) - V\|_{L^\infty(B_1)} \leq \Pi_0.
\ee
Moreover, by \eqref{e:ipotesi_boundary_reg_wlog} and \eqref{e:widetildeV-diff}, choosing the constant $\eta_0$ in \eqref{e:ipotesi_boundary_reg_wlog_g} small enough, for every $y\in B_1$, we have that  
$$\tens{A}(y)\widetilde V(y) \cdot e_d \ge 2\delta.$$ 
In addition, by \cref{lemma:graph}, we have that if $x_0\in\Omega_u\cap D$, then $0<u(x_0+tV)=v(\Theta(x_0+tV))$, for every $t>0$ such that $x_0+tV\in D$, where
 \be\label{eq:oss-utile}
    \gamma(t):=\Theta(x_0+tv):\R\to\R^d \quad \hbox{ such that } \quad
    \left\{
\begin{aligned}
    \gamma'(t) &= \tens{A}(\gamma(t))\widetilde V(\gamma(t)),\\
    \gamma'(0) &= x_0.
 \end{aligned}
 \right.\ee
Finally, we remark that the rate of convergence in \cref{lemma:proprieta_blowup} is still valid for $v = u\circ \Theta^{-1}$, where the statement and the notations must by modified according to the change of variable $\Theta$, as done in \cref{remark:interior}. 
\end{appendix}

\section*{Acknowledgements}
The authors are deeply grateful to Bozhidar Velichkov for the numerous and insightful discussions that greatly contributed to the development of this work. His support and useful advice were very helpful in the preparation of this paper.

The authors are supported by the European Research Council (ERC), under the European Union's Horizon 2020 research and innovation program, through the project ERC VAREG - {\em Variational approach to the regularity of the free boundaries} (grant agreement No. 853404). They also acknowledge the MIUR Excellence Department Project awarded to the Department of Mathematics, University of Pisa, CUP I57G22000700001.
Moreover, they are supported by Gruppo Nazionale per l'Analisi Matematica, la Probabilit\'a e le loro Applicazioni (GNAMPA) of Istituto Nazionale per l'Alta Matematica (INdAM). GB is partially supported through the INdAM-GNAMPA project 2025 CUP E5324001950001.

\printbibliography
\end{document}